\documentclass[10pt,reqno]{amsart}
\usepackage{dsfont,amsmath,amsfonts,amscd,amssymb,graphicx,mathrsfs,eufrak,enumitem,euscript,colonequals,upgreek,pifont}
\usepackage[all,arc,curve,color,frame]{xy}
\usepackage[usenames]{color}

\newtheorem{thm}{Theorem}[section]
\newtheorem{theorem}[thm]{Theorem}
\newtheorem{prop}[thm]{Proposition}
\newtheorem{lemma}[thm]{Lemma}
\newtheorem{definition}[thm]{Definition}
\newtheorem{cor}[thm]{Corollary}
\newtheorem{corollary}[thm]{Corollary}

\theoremstyle{definition} 

\newtheorem{example}[thm]{Example}

\newtheorem{remark}[thm]{Remark}

\newtheorem{notation}[thm]{Notation}

\makeatletter
\newcommand{\raisemath}[1]{\mathpalette{\raisem@th{#1}}}
\newcommand{\raisem@th}[3]{\raisebox{#1}{$#2#3$}}
\makeatother

\numberwithin{equation}{section}

\usepackage{subfiles}

\usepackage{dsfont,amsmath,amsfonts,amscd,amssymb,graphicx,mathrsfs,eufrak,enumitem,euscript,upgreek,colonequals}
\usepackage[usenames]{color}
\usepackage{oubraces}

\usepackage{setspace}
\usepackage{array,multirow,booktabs,longtable}

\setlist[enumerate]{format=\normalfont}

\allowdisplaybreaks

\usepackage{tabu}
\usepackage[dvipsnames]{xcolor}   		             		
\usepackage{graphicx}			
\usepackage{stmaryrd}
\usepackage{tikz}
\usepackage{tikz-cd}
\usepackage{accents}
\usepackage{upgreek}
\usepackage{mathtools}
\usepackage{todonotes}
\usepackage{bbm}

\usepackage[cal=cm]{mathalfa}
\usepackage{xparse}
\usepackage{comment}
\usepackage{cite}
\usetikzlibrary{arrows}

\tikzset{
  commutative diagrams/.cd, 
  arrow style=tikz, 
  diagrams={>=stealth}
}

\makeatletter
\def\@tocline#1#2#3#4#5#6#7{\relax
  \ifnum #1>\c@tocdepth % then omit
  \else
    \par \addpenalty\@secpenalty\addvspace{#2}%
    \begingroup \hyphenpenalty\@M
    \@ifempty{#4}{%
      \@tempdima\csname r@tocindent\number#1\endcsname\relax
    }{%
      \@tempdima#4\relax
    }%
    \parindent\z@ \leftskip#3\relax \advance\leftskip\@tempdima\relax
    \rightskip\@pnumwidth plus4em \parfillskip-\@pnumwidth
    #5\leavevmode\hskip-\@tempdima
      \ifcase #1
       \or\or \hskip 1em \or \hskip 2em \else \hskip 3em \fi%
      #6\nobreak\relax
    \dotfill\hbox to\@pnumwidth{\@tocpagenum{#7}}\par
    \nobreak
    \endgroup
  \fi}
\makeatother

\usetikzlibrary{calc}
\usetikzlibrary{fadings}
\usetikzlibrary{decorations.pathmorphing}
\usetikzlibrary{decorations.pathreplacing}

\newcounter{marginnote}
\usepackage[pagebackref]{hyperref}

\hypersetup{
  colorlinks   = true,          %Colors links instead of ugly boxes
  urlcolor     = blue,          %Color for external hyperlinks
  linkcolor    = purple,          %Color of internal links
  citecolor   = blue             %Color of citations
}

\DeclareMathOperator{\ev}{ev}

\DeclareMathOperator{\Hom}{Hom}

\DeclareMathOperator{\Pic}{Pic}

\newcommand{\QQ}{\mathbb{Q}}
\newcommand{\PP}{\mathbb{P}}
\newcommand{\ZZ}{\mathbb{Z}}

\newcommand{\formaldisc}{\scrD \mathrm{isc}}

\newcommand{\Achow}{\mathrm{A}}
\newcommand{\Hyp}{\mathsf{H}}
\newcommand{\Fund}{\mathsf{Fund}}

\newcommand{\ol}[1]{\overline{#1}}

\newcommand{\bcd}{\begin{center}\begin{tikzcd}}
\newcommand{\ecd}{\end{tikzcd}\end{center}}

\newcommand{\C}{\mathbb{C}}

\newcommand{\Z}{\mathbb{Z}}

\newcommand{\R}{\mathbb{R}}

\newcommand{\Spec}{\operatorname{Spec}}

\NewDocumentCommand{\compatibilitydatum}{m m m m m m O{} O{} O{}}{
\begin{equation*} \begin{tikzcd}[ampersand replacement=\&]
  \: \arrow{r} \& {#1} \arrow{r} \arrow{d}{#7} \& {#2} \arrow{r} \arrow{d}{#8} \& {#3} \arrow{r}{[1]} \arrow{d}{#9} \& \: \\
  \: \arrow{r} \& {#4} \arrow{r} \& {#5} \arrow{r} \& {#6} \arrow{r} \& \:
\end{tikzcd} \end{equation*}}

\NewDocumentCommand{\commutingsquare}{m m m m o O{} O{} O{} O{}}{
\begin{equation}\begin{tikzcd}[ampersand replacement=\&] \label{#5}
  #1 \arrow{r}{#6} \arrow{d}{#7} \& #2 \arrow{d}{#8} \\
  #3 \arrow{r}{#9} \& #4
\end{tikzcd}\IfValueTF{#5}{\label{#5}}{} \end{equation}}

\NewDocumentCommand{\cartesiansquare}{m m m m O{} O{} O{} O{}}{
\begin{equation*}\begin{tikzcd}[ampersand replacement=\&]
  #1 \arrow{r}{#5} \arrow{d}{#6} \arrow[dr, phantom, "\square"] \& #2 \arrow{d}{#7} \\
  #3 \arrow{r}{#8} \& #4
\end{tikzcd} \end{equation*}}

\NewDocumentCommand{\cartesiansquarelabel}{m m m m m O{} O{} O{} O{}}{
\begin{tikzcd}[ampersand replacement=\&]
  #1 \arrow{r}{#6} \arrow{d}{#7} \arrow[dr, phantom, "\square"] \& #2 \arrow{d}{#8} \\
  #3 \arrow{r}{#9} \& #4
\end{tikzcd}\IfValueTF{#5}{\label{#5}}{}
}

\NewDocumentCommand{\triangleofspaces}{m m m O{} O{} O{}}{
\begin{tikzcd} [ampersand replacement=\&]
#1 \arrow{r}{#4} \arrow[bend right]{rr}{#5} \& #2 \arrow{r}{#6} \& #3
\end{tikzcd}}

\tikzset{
        DB/.style={circle,draw=black,circle,fill=white,inner sep=0pt, minimum size=5pt},
        DW/.style={circle,draw=black,fill=black,inner sep=0pt, minimum size=5pt},
        cvertex/.style={circle,draw=black,fill=white,inner sep=1pt,outer sep=3pt},
        vertex/.style={circle,fill=black,inner sep=1pt,outer sep=3pt},
        star/.style={circle,fill=yellow,inner sep=0.75pt,outer sep=0.75pt},
        tvertex/.style={inner sep=1pt,font=\scriptsize},
	pvertex/.style={circle,inner sep=1pt,outer sep=2pt,font=\scriptsize},
        gap/.style={inner sep=0.5pt,fill=white}}
\tikzset{
Bquiv/.style={circle,draw=black!80!white,circle,fill=black!80!white,inner sep=0pt, outer sep=4pt, minimum size=3pt},
W/.style={circle,draw=black,circle,fill=white,inner sep=0pt, minimum size=4pt},
B/.style={circle,draw=black!80!white,circle,fill=black!80!white,inner sep=0pt, minimum size=4pt},
BL/.style={circle,draw=blue!60!white,circle,fill=blue!60!white,inner sep=0pt, minimum size=4pt},
R/.style={circle,draw=red!60!white,circle,fill=red!60!white,inner sep=0pt, minimum size=4pt},  
G/.style={circle,draw=green!65!black,circle,fill=green!65!black,inner sep=0pt, minimum size=4pt},     
Rs/.style={circle,draw=red!60!white,circle,fill=red!60!white,inner sep=0pt, minimum size=2pt}, 
Bs/.style={circle,draw=black!80!white,circle,fill=black!80!white,inner sep=0pt, minimum size=2pt},
BLs/.style={circle,draw=blue!60!white,circle,fill=blue!60!white,inner sep=0pt, minimum size=2pt},
Gs/.style={circle,draw=green!65!black,circle,fill=green!65!black,inner sep=0pt, minimum size=2pt},  }

\tikzstyle{mybox} = [draw=black, fill=blue!10, very thick,
    rectangle, rounded corners, inner sep=10pt, inner ysep=20pt]
\tikzstyle{boxtitle} =[fill=blue!50, text=white,rectangle,rounded corners]

\newcommand{\marginparstretch}{0.6}
\let\oldmarginpar\marginpar
\renewcommand\marginpar[1]{\-\oldmarginpar[\framebox{\setstretch{\marginparstretch}\begin{minipage}{\marginparwidth}{\raggedleft\tiny #1}\end{minipage}}]{\framebox{\setstretch{\marginparstretch}\begin{minipage}{\marginparwidth}{\raggedright\tiny #1}\end{minipage}}}}

\definecolor{Pink}{RGB}{230 56 243}
\definecolor{Blue}{RGB}{0 19 147}
\definecolor{Green}{RGB}{66 147 41}
\definecolor{Grey}{RGB}{102 102 102}
\definecolor{Orange}{RGB}{237 107 45}
\definecolor{Red}{RGB}{234 53 159}

\def\mythick{0.5}

\newcommand\citetype[1]{}

\def\Cl{\mathop{\rm Cl}\nolimits}
\def\Pic{\mathop{\rm Pic}\nolimits}

\def\coh{\mathop{\rm coh}\nolimits}

\def\Hom{\mathop{\rm Hom}\nolimits}

\def\RHom{\mathop{\rm {\bf R}Hom}\nolimits}
\def\End{\mathop{\rm End}\nolimits}
\def\Ext{\mathop{\rm Ext}\nolimits}

\def\Spec{\mathop{\rm Spec}\nolimits}

\def\Db{\mathop{\rm{D}^b}\nolimits}

\newcommand{\aff}{\mathsf{aff}}

\newcommand{\con}{\mathrm{con}}

\newcommand{\CA}{\mathrm{A}_{\con}}

\def\RHom{{\rm{\bf R}Hom}}

\newcommand\Curve{\mathrm{C}}

\newcommand{\Xfrak}{\mathfrak{X}}

\newcommand{\cStab}[1]{\mathrm{Stab}_{#1}^{\kern -0.5pt \circ}\kern -0.25pt}

\newcommand{\cAut}{\mathrm{Aut}^{\kern 0pt \circ}\kern -0.25pt}

\newcommand{\qsf}{\mathsf{q}}
\newcommand{\rsf}{\mathsf{r}}

\newcommand{\Msf}{\mathsf{M}}
\newcommand{\Nsf}{\mathsf{N}}

\usepackage{euscript}

\newcommand{\scrD}{\EuScript{D}}
\newcommand{\scrE}{\EuScript{E}}
\newcommand{\scrF}{\EuScript{F}}
\newcommand{\scrG}{\EuScript{G}}
\newcommand{\scrH}{\EuScript{H}}
\newcommand{\scrI}{\EuScript{I}}

\newcommand{\scrK}{\EuScript{K}}
\newcommand{\scrL}{\EuScript{L}}

\newcommand{\scrO}{\EuScript{O}}
\newcommand{\scrP}{\EuScript{P}}

\newcommand{\scrR}{\EuScript{R}}
\newcommand{\scrS}{\EuScript{S}}
\newcommand{\scrT}{\EuScript{T}}

\newcommand{\scrV}{\EuScript{V}}
\newcommand{\scrW}{\EuScript{W}}
\newcommand{\scrX}{\EuScript{X}}
\newcommand{\scrY}{\EuScript{Y}}
\newcommand{\scrZ}{\EuScript{Z}}

\newcommand{\scrIc}{\scrI^{\kern 0.5pt\mathrm{c}}}

\newcommand{\EeightSixNoLattice}[1][]{%
\begin{tikzpicture}[xscale=0.25,yscale=0.5]
\coordinate (A) at (0,0);
\coordinate (B) at (0,2);
\coordinate (C) at (0,4);
\coordinate (Ap) at (8,0);
\coordinate (Bp) at (8,2);
\coordinate (Cp) at (8,4);
\coordinate (a) at (2.66,4);
\coordinate (b) at (4,4);
\coordinate (c) at (5.33,4);
\coordinate (ap) at (2.66,0);
\coordinate (bp) at (4,0);
\coordinate (cp) at (5.33,0);
\draw (A)--(C);
\draw (Ap)--(Cp);
\draw (A)--(Ap);
\draw (B)--(Bp);
\draw (C)--(Cp);
\draw (b)--(bp);
\draw (a)--(ap);
\draw (c)--(cp);
\draw (A)--(b);
\draw (A)--(Cp);
\draw (C)--(bp);
\draw (C)--(Ap);
\draw (Cp)--(bp);
\draw (Ap)--(b);
\end{tikzpicture}
}

\newcommand{\EeightSixNoLatticeGREY}[1][]{%
\begin{tikzpicture}[xscale=0.25,yscale=0.5]
\coordinate (A) at (0,0);
\coordinate (B) at (0,2);
\coordinate (C) at (0,4);
\coordinate (Ap) at (8,0);
\coordinate (Bp) at (8,2);
\coordinate (Cp) at (8,4);
\coordinate (a) at (2.66,4);
\coordinate (b) at (4,4);
\coordinate (c) at (5.33,4);
\coordinate (ap) at (2.66,0);
\coordinate (bp) at (4,0);
\coordinate (cp) at (5.33,0);
\draw[gray] (A)--(C);
\draw[gray] (Ap)--(Cp);
\draw[gray] (A)--(Ap);
\draw[gray] (B)--(Bp);
\draw[gray] (C)--(Cp);
\draw[gray] (b)--(bp);
\draw[gray] (a)--(ap);
\draw[gray] (c)--(cp);
\draw[gray] (A)--(b);
\draw[gray] (A)--(Cp);
\draw[gray] (C)--(bp);
\draw[gray] (C)--(Ap);
\draw[gray] (Cp)--(bp);
\draw[gray] (Ap)--(b);
\end{tikzpicture}
}

\newcommand{\Eseven}[7]{%
\begin{tikzpicture}[scale=0.21]
\node at (0,0) [#1] {};
\node at (1,0) [#2] {};
\node at (2,0) [#3] {};
\node at (2,1) [#4] {};
\node at (3,0) [#5] {};
\node at (4,0) [#6] {};
\node at (5,0) [#7] {};
\end{tikzpicture}
}

\newcommand{\Eeight}[8]{%
\begin{tikzpicture}[scale=0.21]
\node at (0,0) [#1] {};
\node at (1,0) [#2] {};
\node at (2,0) [#3] {};
\node at (2,1) [#4] {};
\node at (3,0) [#5] {};
\node at (4,0) [#6] {};
\node at (5,0) [#7] {};
\node at (6,0) [#8] {};
\end{tikzpicture}
}
\tikzset{
W/.style={circle,draw=black,circle,fill=white,inner sep=0pt, minimum size=4pt},
B/.style={circle,draw=black!80!white,circle,fill=black!80!white,inner sep=0pt, minimum size=4pt},
Or/.style={circle,draw=Orange,circle,fill=Orange,inner sep=0pt, minimum size=4pt},
P/.style={circle,draw=Pink,circle,fill=Pink,inner sep=0pt, minimum size=4pt},
R/.style={circle,draw=black!80!white,circle,fill=red!80!white,inner sep=0pt, minimum size=4pt},  
}

\begin{document}
\title{\textsc{GV and GW invariants via the enhanced movable cone}}
\author{Navid Nabijou}
\address{Navid Nabijou, School of Mathematical Sciences, Queen Mary University of London, Mile End Road, London E1 4NS, UK.}
\email{n.nabijou@qmul.ac.uk}
\author{Michael Wemyss}
\address{Michael Wemyss, The Mathematics and Statistics Building, University of Glasgow, University Place, Glasgow, G12 8QQ, UK.}
\email{michael.wemyss@glasgow.ac.uk}
\begin{abstract} 
Given any smooth germ of a threefold flopping contraction, we first give a combinatorial characterisation of which Gopakumar--Vafa (GV) invariants are non-zero, by prescribing multiplicities to the walls in the movable cone. On the Gromov--Witten (GW) side, this allows us to describe, and even draw, the critical locus of the associated quantum potential. We prove that the critical locus is the infinite hyperplane arrangement of \cite{IyamaWemyssTits}, and moreover that the quantum potential can be reconstructed from a finite fundamental domain. We then iterate, obtaining a combinatorial description of the matrix which controls the transformation of the non-zero GV invariants under a flop. There are three main ingredients and applications: (1) a construction of flops from simultaneous resolution via cosets, which describes how the dual graph changes, (2) a closed formula which describes the change in dimension of the contraction algebra under flop, and (3) a direct and explicit isomorphism between quantum cohomologies of different crepant resolutions, giving a Coxeter-style, visual proof of the Crepant Transformation Conjecture for isolated cDV singularities.
\end{abstract}
\thanks{The authors were supported by EPSRC grant~EP/R009325/1. The first author was additionally supported by the Herchel Smith Fund, and the second by ERC Consolidator Grant 101001227 (MMiMMa). For the purpose of open access, the authors have applied a Creative Commons Attribution (CC BY) licence to any Author Accepted Manuscript version arising.}

\maketitle
\tableofcontents
\parindent 20pt
\parskip 0pt

\thispagestyle{empty}
%\tableofcontents

\section{Introduction}

\noindent Many key results in algebraic geometry can be established using topological and combinatorial descriptions of a given variety, or of its degenerations and deformations. However, even with a clear combinatorial model on such a degeneration or deformation, determining which properties of the original variety can be controlled by combinatorics is still in general a difficult question.

This paper considers arbitrary smooth $3$-fold flopping contractions, which form a fundamental building block of the minimal model programme. Our main point is that, as far as their enumerative geometry is concerned, all such flopping contractions \emph{are} combinatorial, provided we are content with describing only the shape of the enumerative invariants, rather than their precise values. This qualitative perspective allows us to extract, and prove rather easily, many fundamental results. We determine which curve classes give rise to non-zero invariants, then control how these invariants transform under flop, in a visually pleasing and satisfyingly combinatorial manner. Along the way it is necessary to enhance existing geometric structures, such as the movable cone.

\subsection{Gopakumar--Vafa: finite arrangements}
Let $f \colon \scrX \to \Spec \scrR$ be a crepant resolution of a $3$-fold isolated cDV singularity, equivalently a germ of a smooth $3$-fold flopping contraction. The morphism $f$ contracts a finite collection $\{ \Curve_i \subseteq \scrX \mid i \in \scrIc\}$ of complete curves to a point, and these freely generate the group of algebraic curve classes
\begin{equation*} 
\Achow_1(\scrX) = \langle \Curve_i \mid i \in \scrIc \rangle_{\ZZ}.
\end{equation*}
Given $\upbeta \in \Achow_1(\scrX)$, Katz defines an associated Gopakumar--Vafa (GV) invariant \cite{KatzGV} 
\[
n_{\upbeta}=n_{\upbeta}(\scrX) \in \ZZ_{\geq 0}
\]
which we recall in Subsection~\ref{GV section} below. It is an important invariant of flopping contractions, with close connections to other known invariants \cite{DW1, TodaGV}.

\medskip
Our first result determines those $\upbeta$ for which $n_{\upbeta} \neq 0$. The description is direct and combinatorial, encoded in the associated finite hyperplane arrangement $\scrH_\scrI$ of Iyama and the second author \cite{IyamaWemyssTits}. The chambers in $\scrH_\scrI$  are precisely the ample cones of the different birational models, and so we can identify $\scrH_\scrI$ with the movable cone.  It turns out that the GV invariants are to first approximation encoded by the walls of this cone. There is however a slight catch: \emph{combinatorially} the walls carry multiplicities, and this data is not part of the definition of the movable cone.  This multiplicity, which is new information (see Remark~\ref{rem: movable via alg geom}),  turns out to be the key to determining whether $n_\upbeta\neq 0$.

\medskip
As is standard, and recalled in Subsection~\ref{sec: Dynkin recap}, slicing $\scrX \to \Spec \scrR$ by a generic hyperplane section gives rise to a partial crepant resolution of an ADE surface singularity.  From this slicing we thus obtain the Dynkin diagram $\Delta$ of the ADE surface singularity, together with a subset $\scrI$ of nodes: the full minimal resolution dominates the partial resolution, and $\scrI$ are the curves which are contracted by this morphism.

\begin{example}\label{example: intro}
As the running example, consider a two-curve smooth $3$-fold flop for which the corresponding Dynkin data is $\Eeight{B}{P}{B}{B}{B}{B}{B}{P}$, where by convention $\scrI$ equals the six black nodes.  The Dynkin data gives rise to a finite \emph{intersection arrangement} $\scrH_{\scrI}\subseteq\mathbb{R}^{|\Delta|-|\scrI|}=\mathbb{R}^2$ \cite[Section~3]{IyamaWemyssTits}. One method of calculating $\scrH_{\scrI}$ is to first restrict all 120 positive roots of $E_8$ to the subset $\scrIc = \Delta \setminus\scrI$, and thus obtain the set 
\[
\{ 01, 11, 21, 42, 31, 41, 10, 20, 30\}.
\] 
These so-called \emph{restricted roots} give rise to hyperplanes in the dual space, where for example $42$ gives rise to the hyperplane $4x+2y=0$.  The output is thus the following hyperplane arrangement, which we emphasise is constructed entirely  from $\scrI\subseteq\Delta$.
\begin{equation}
\begin{array}{cccc}
\begin{array}{c}
\begin{tikzpicture}[scale=0.5]
%01, the x-axis axis 
\draw[->,densely dotted] (180:2cm)--(0:2cm);
\node at (0:2.5) {$\scriptstyle x$};
%10, the y-axis axis
\draw[->,densely dotted] (-90:2cm)--(90:2cm);
\node at (90:2.5) {$\scriptstyle y$};
%\draw[densely dotted,gray] (0,0) circle (2cm);
\end{tikzpicture}
\end{array}
%%%%%%%%%%%%%%%%%%%%%%%%%%%%%
&
%%%%%%%%%%%%%%%%%%%%%%%%%%%%%%
\begin{array}{c}
\begin{tikzpicture}[scale=1]
%01, the x-axis axis 
\draw[line width=\mythick mm,Pink] (180:2cm)--(0:2cm);
\node at (180:2.2) {$\scriptstyle 1$};
%11 
\draw[line width=\mythick mm,Orange] (135:2cm)--(-45:2cm);
\node at (135:2.2) {$\scriptstyle 1$};
%21
\draw[line width=\mythick mm, Blue] (116.57:2cm)--(-63.43:2cm);
\node at (116.57:2.2) {$\scriptstyle 2$};
%31
\draw[line width=\mythick mm, Green] (108.43:2cm)--(-71.57:2cm);
\node at (108.43:2.2) {$\scriptstyle 1$};
%41
\draw[line width=\mythick mm, Grey] (104.04:2cm)--(-75.96:2cm);
\node at (104.04:2.2) {$\scriptstyle 1$};
%10, the y-axis axis
\draw[line width=\mythick mm,Pink] (90:2cm)--(-90:2cm);
\node at (90:2.2) {$\scriptstyle 3$};
%\draw[densely dotted,gray] (0,0) circle (2cm);
\end{tikzpicture}
\end{array}&
\begin{array}{c}
\begin{tabular}{ccc}
\toprule
Restricted Root&\\
\midrule
$01$&$\tikz\draw[line width=\mythick mm, Pink] (0,0) -- (0.25,0);$\\
$11$&$\tikz\draw[line width=\mythick mm, Orange] (0,0) -- (0.25,0);$\\
$21, 42$&$\tikz\draw[line width=\mythick mm, Blue] (0,0) -- (0.25,0);$\\
$31$&$\tikz\draw[line width=\mythick mm, Green] (0,0) -- (0.25,0);$\\
$41$&$\tikz\draw[line width=\mythick mm, Grey] (0,0) -- (0.25,0);$\\
$10, 20, 30$&$\tikz\draw[line width=\mythick mm, Pink] (0,-0.15) -- (0,0.15);$\\
\bottomrule
\end{tabular}
\end{array}
\end{array}\label{running example finite}
\end{equation}
Note that the restricted root $42$ gives rise to the hyperplane $2(2x+y)=0$, and so the blue diagonal $2x+y=0$ line carries the list $[1,2]$ of multiplicities.  We write $2$ beside the blue line to emphasise this fact.  Similarly, the line $x=0$ carries the list $[1,2,3]$ of multiplicities, as a consequence of $20$ and $30$.

\end{example}

Returning to general flopping contractions $\scrX\to\Spec\scrR$, since by construction the nodes in $\Delta\backslash\scrI$ can be identified with the curves in $\scrX$, after some natural identifications $\scrH_{\scrI}$ (with multiplicities) can be viewed inside $\Pic\scrX \otimes \R$.  So can the movable cone.  After forgetting the multiplicities, $\scrH_{\scrI}$ is equal to the movable cone  \cite{Pinkham, HomMMP}.

The following is our first main result. It describes the non-zero GV invariants in an elementary combinatorial way, and asserts that it is the hyperplanes of $\scrH_{\scrI}$, counted with multiplicity, that control the non-zero GV invariants.

\begin{thm}[\ref{cor: GV nonzero}]\label{nonzero GV intro}
For $\upbeta \in \Achow_1(\scrX)$ the GV invariant $n_{\upbeta}$ is non-zero if and only if $\upbeta$ is a restricted root.
\end{thm}

The theme of this paper is that the shape of the enumerative geometry of $\scrX$ is controlled, in a very visual way, from this finite amount of initial data.

\subsection{Gromov--Witten: infinite arrangements}\label{sec: GW intro}
Given any subset $\scrI$ of a Dynkin diagram, the finite arrangement $\scrH_{\scrI}$ of the previous subsection has an infinite cousin $\scrH_{\scrI}^{\aff}$.  Given a restricted root $\upbeta=(\upbeta_i)_{i\in\scrIc}$, the hyperplane $\sum_{i \in \scrIc} \upbeta_i x_i = 0 $ appearing in the finite arrangement gets translated over the integers $\ZZ$, to give an infinite family
\begin{equation}
 \sum_{i \in \scrIc} \upbeta_i x_i \in \ZZ. \label{int translates}
\end{equation}
Repeating this over every restricted root results in an infinite arrangement of affine hyperplanes, written $\scrH_{\scrI}^{\aff}$. Note that multiplicities on hyperplanes of $\scrH_{\scrI}$ result in more translations, as if say $2\upbeta=(2\upbeta_i)$ is also a restricted root, then its translates give rise to the family $ \sum 2\upbeta_i x_i \in \ZZ$, i.e.\ to $ \sum \upbeta_i x_i \in \frac{1}{2}\ZZ$.  This is larger than \eqref{int translates}.

In the running Example~\ref{example: intro}, taking all the relevant translations of \eqref{running example finite} results in the following $\scrH^{\aff}_{\scrI}$.
\begin{equation}
\begin{array}{cc}
\begin{array}{c}
\begin{tikzpicture}[scale=0.5]
%01, the x-axis axis 
\draw[->,densely dotted] (0,-0.5)--($(0,-0.5)+(52.5:4.5cm)$);
\node at ($(0,-0.5)+(57.5:3.9cm)$) {$\scriptstyle x$};
%10, the y-axis axis
\draw[->,densely dotted] (0,-0.5)--(0,1.3);
\node at (90:1.75) {$\scriptstyle y$};
\end{tikzpicture}
\end{array}
%%%%%%%%%%%%%%%%%%%%%%%%%%%%
&
%%%%%%%%%%%%%%%%%%%%%%%%%%%%
%\hspace*{-12.55em}
\begin{array}{c}
\includegraphics[angle=0,width=7.5cm,height=4cm]{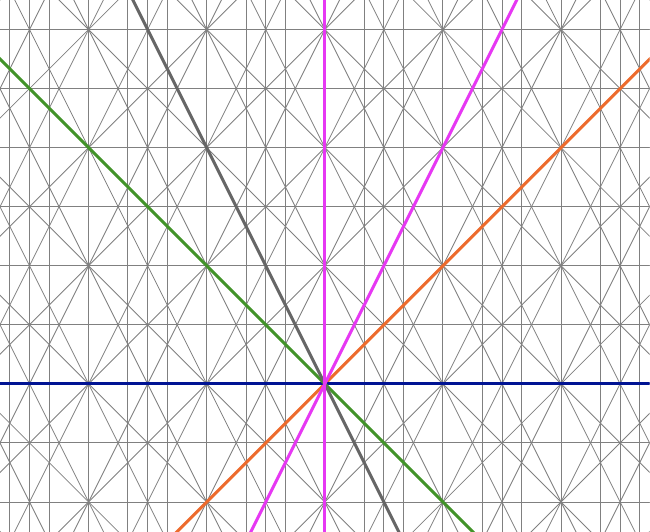}
\end{array}
\end{array}\label{running example infinite}
\end{equation}

Returning to a general flopping contraction $\scrX\to\Spec\scrR$, the next main result relates the Gromov--Witten (GW) theory of $\scrX$ to the associated infinite arrangement $\scrH^{\aff}_{\scrI}$. The GW invariants are virtual degrees of moduli spaces of stable maps, and provide a system of curve counts equivalent to the more enumerative GV invariants via multiple cover formulae. The GW invariants form the structure constants for the quantum cohomology algebra, but for our purposes it is more convenient to package this information into a generating function, called the \emph{quantum potential} (see Subsection~\ref{sec: GW invariants} for details). Combining Theorem~\ref{nonzero GV intro} and the multiple cover formulae gives the second main result.

\begin{thm}[\ref{cor: quantum is hyper}]\label{GW=hyper intro} 
Let $\scrX\to\Spec\scrR$ be a smooth $3$-fold flopping contraction. The pole locus of the quantum potential is the complexification of the infinite arrangement $\scrH^{\aff}_{\scrI}$. 
\end{thm}
\noindent By \cite{HW2}, the complement of the complexified arrangement $\scrH^\aff_\scrI$ forms the base of the Bridgeland stability covering map, for a natural compactly-supported subcategory of the derived category of $\scrX$. Theorem~\ref{GW=hyper intro} therefore connects quantum-cohomological Frobenius manifolds and spaces of stability conditions, a phenomenon which has been observed in other contexts \cite{BridgelandNonCompact,IkedaQiu,McAuleyThesis}.

\subsection{Flops via simultaneous resolution}
To track the change of GV/GW invariants under iterated flops requires us to first rework some of the theory of simultaneous resolutions, which may be of independent interest.  
 Our new contribution is to use the wall crossing formula from \cite{IyamaWemyssTits}, which indexes chambers of the movable cone by certain cosets, to construct iterated flops from simultaneous (partial) resolutions, and explain how the dual graph changes under flop.  This completes work of Reid \cite{Reid}, Pinkham \cite{Pinkham}, and Katz--Morrison \cite{KatzMorrison} in the 80s and 90s, rounding off a circle of ideas going back to Brieskorn \cite{Brieskorn}.

For any Kleinian singularity $\mathbb{C}^2/G$, consider the corresponding Dynkin diagram $\Delta$, root space $\mathfrak{h}$, and Weyl group $W$.  As is standard, $\mathbb{C}^2/G$ admits a versal deformation $\Spec\scrV$ over the base $\mathfrak{h}_{\mathbb{C}}/W$.  For any subset $\scrI\subseteq\Delta$, consider the parabolic subgroup $W_\scrI\colonequals \langle s_i\mid i\in \scrI\rangle$, and take the pullback to obtain
\[
\begin{tikzpicture}
%\node (X) at (0,0) {$\scrX$};
%\node (Z) at (2,0) {$\scrZ$};
%\node (R) at (0,-1.5) {$\Spec \scrR$};
\node (U) at (2,-1.5) {$\Spec \scrV_{\scrI}$};
\node (V) at (4,-1.5) {$\Spec \scrV$};
%\node (t) at (0,-3) {$\Spec \mathbb{C}\llsq t\rrsq$};
\node (PRes) at (2,-3) {$\mathfrak{h}_{\mathbb{C}}/W_\scrI$};
\node (Res) at (4,-3) {$\mathfrak{h}_{\mathbb{C}}/W$.};
%\draw[->] (X)--(Z);
%\draw[->] (X)--(R);
%\draw[->] (Z)--(U);
%\draw[->] (R)--(U);
\draw[->] (U)--(V);
\draw[->] (V)--(Res);
%\draw[->] (R)--(t);
\draw[->] (U)--node[right]{$\scriptstyle \mathsf{g}_{\scrI}$}(PRes);
%\draw[->] (t)--(PRes);
\draw[->] (PRes)--(Res);
\end{tikzpicture}
\]
When $\scrI=\emptyset$, the parabolic $W_\scrI=\mathds{1}$, and in this case classically $\Spec\scrV_\emptyset$ admits a simultaneous resolution.  

As is now standard, to describe smooth $3$-fold flops requires singular surface geometry, and so the ability to consider $\scrI\neq\emptyset$ is crucial. By \cite{KatzMorrison}, for each $\scrI$ there is a preferred, or \emph{standard} simultaneous partial resolution $\mathsf{h}_{\scrI}\colon \scrY_\scrI\to\Spec\scrV_\scrI$ (see Subsection~\ref{sec: sim partial res I}). Further, by \emph{loc.\ cit.\ }all smooth flops can be constructed via appropriate classifying maps $\upmu\colon\formaldisc\to\mathfrak{h}_{\mathbb{C}}/W_{\scrI}$ from the formal disc to $\mathfrak{h}_{\mathbb{C}}/W_{\scrI}$ for some $\scrI$, giving the following cartesian diagram
\[
\begin{tikzpicture}
\node (X) at (0,0) {$\scrX$};
\node (Z) at (2,0) {$\scrY_{\scrI}$};
\node (R) at (0,-1.5) {$\Spec \scrR$};
\node (U) at (2,-1.5) {$\Spec \scrV_\scrI$};
\node (V) at (4,-1.5) {$\Spec \scrV$};
\node (t) at (0,-3) {$\formaldisc$};
\node (PRes) at (2,-3) {$\mathfrak{h}_{\mathbb{C}}/W_{\scrI}$};
\node (Res) at (4,-3) {$\mathfrak{h}_{\mathbb{C}}/W$.};
\draw[->] (X)--(Z);
\draw[->] (X)--(R);
\draw[->] (Z)--node[right]{$\scriptstyle \mathsf{h}_{\scrI}$}(U);
\draw[->] (R)--(U);
\draw[->] (U)--(V);
\draw[->] (V)--(Res);
\draw[->] (R)--(t);
\draw[->] (U)--node[right]{$\scriptstyle \mathsf{g}_{\scrI}$}(PRes);
\draw[->] (t)--node[above]{$\scriptstyle\upmu$}(PRes);
\draw[->] (PRes)--(Res);
\end{tikzpicture}
\]
The question is, given $\scrX\to\Spec \scrR$, how to construct the flop at a given curve from the classifying map $\upmu$. This was solved in the case $\scrI=\emptyset$ by Reid \cite{Reid}, but the general case is harder, since the subset $\scrI$ changes under flop. Pinkham \cite{Pinkham} counts only the number of simultaneous resolutions.

We solve this problem by appealing to the wall crossing combinatorics of \cite{IyamaWemyssTits}. The key point is that when $\scrI\neq\emptyset$, chambers in the movable cone are indexed by cosets, not by elements of the Weyl group. For any subset $\Gamma \subseteq \Delta$ let $W_\Gamma \subseteq W$ denote the parabolic subgroup generated by reflections dual to the elements of $\Gamma$, and let $\ell_\Gamma \in W_\Gamma$ denote the longest element. For any curve $\Curve_i\subseteq\scrX$, set $w_i=\ell_{\scrI}\ell_{\scrI\cup\{i\}} \in W$. Then there is a unique subset $\upomega_i(\scrI)\subseteq\Delta$, described explicitly in Section~\ref{sec: flops via sim res}, for which $W_{\scrI}w_i=w_iW_{\upomega_i(\scrI)}$.

\begin{thm}[\ref{thm: produce flop}]
Post-composing $\upmu$ with $w_i^{-1}\colon \mathfrak{h}_{\mathbb{C}}/W_{\scrI}\to\mathfrak{h}_{\mathbb{C}}/W_{\omega_i(\scrI)}$ and taking the pullback constructs the flop $\scrX_i^+\to\Spec\scrR$ of the curve $\Curve_i\subset\scrX$.  In particular
\begin{enumerate}
\item $\upomega_i(\scrI)$ is the dual graph of the exceptional locus of $\scrX_i^+\to\Spec\scrR$.
\item All other crepant resolutions can be obtained from the fixed $\upmu$ by post-composing with $x^{-1}\colon\mathfrak{h}_{\mathbb{C}}/W_{\scrI}\to\mathfrak{h}_{\mathbb{C}}/W_{\scrK}$ and pulling back along $\scrY_\scrK$, as the pair $(x,\scrK)$ ranges over the (finite) indexing set $\mathsf{Cham}(\Delta,\scrI)$ described in Notation~\textnormal{\ref{notation: subset stuff}}.
\end{enumerate}
\end{thm}
We describe the above in detail in Section~\ref{sec: flops via sim res}, but emphasise here that everything is formed intrinsically from the Dynkin data, once $\scrX\to\Spec\scrR$ and thus $\upmu$ is fixed.
%In this way, the wall crossing combinatorics of \cite{IyamaWemyssTits} generates all crepant resolutions from a fixed $\upmu$, via simultaneous partial resolution.

\subsection{Tracking fundamental regions}
With the above in hand, tracking the change in GV/GW invariants under all possible flops becomes easy, and satisfyingly visual.  In our running Example~\ref{example: intro}, each of the 12 crepant resolutions $\scrX_i\to\Spec\scrR$ admits a fundamental region in \eqref{running example infinite}. As calibration, the fixed $\scrX\to\Spec \scrR$ corresponds to the unit box in the purple axes
\begin{equation}
\begin{array}{cc}
\begin{array}{c}
\begin{tikzpicture}[scale=0.5]
%01, the x-axis axis 
\draw[->,densely dotted] (0,-0.5)--($(0,-0.5)+(52.5:4.5cm)$);
\node at ($(0,-0.5)+(57.5:3.9cm)$) {$\scriptstyle x$};
%10, the y-axis axis
\draw[->,densely dotted] (0,-0.5)--(0,1.3);
\node at (90:1.75) {$\scriptstyle y$};
\end{tikzpicture}
\end{array}
%%%%%%%%%%%%%%%%%%%%%%%%%%%%
&
%%%%%%%%%%%%%%%%%%%%%%%%%%%%
\begin{array}{c}
\includegraphics[angle=0,width=7.5cm,height=4cm]{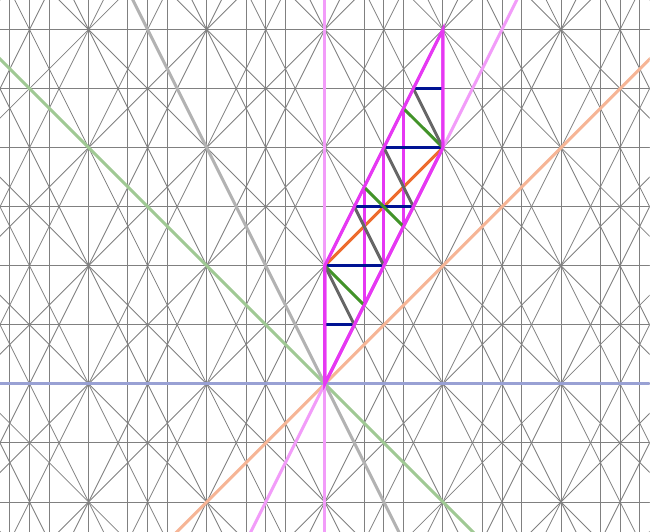}
\end{array}
\end{array}\label{running example infinite fund region 1}
\end{equation}
The other 12 chambers in the movable cone generate similar fundamental regions as in \eqref{running example infinite fund region 1} and this is illustrated in Figure~\ref{figure1}.
\begin{figure}[h] 
\[
\begin{array}{c}
\includegraphics[angle=0,width=7.5cm,height=6cm]{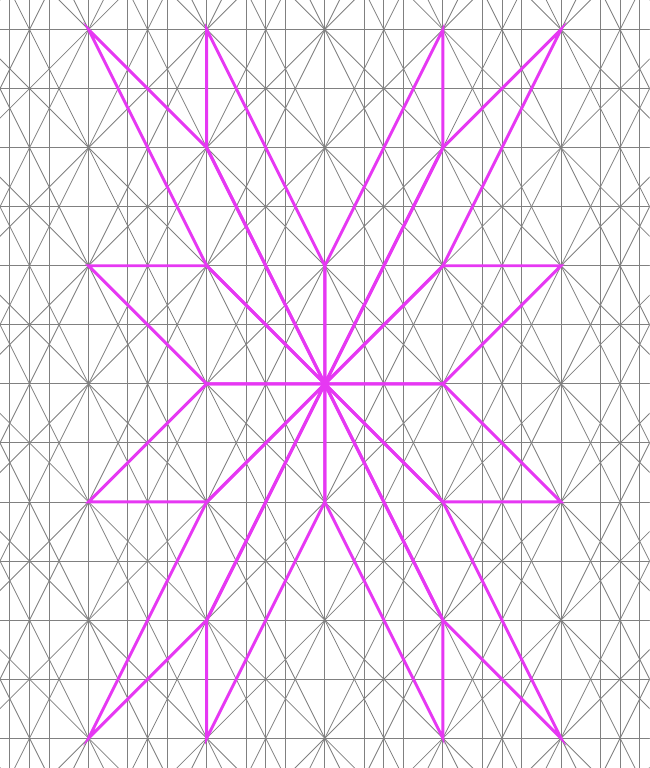}
\end{array}
\]
\caption{The 12 fundamental regions corresponding to the 12 crepant resolutions.}
\label{figure1}
\end{figure}

Reassuringly, flopping a single curve turns out to correspond to the neighbouring region. Although  Figure~\ref{figure1} only illustrates the two-curve flop in the running Example~\ref{example: intro}, similar things happen in full generality (see Section~\ref{sec: applications}).  The following is our third main result, which controls how GV invariants transform under iterated flops.  As notation, set $\scrX_i^+$ to be the scheme obtained from $\scrX$ after flopping only a single curve $\Curve_i$, and further write $n_{\upbeta,\scrX}$ for the GV invariant of curve class $\upbeta$ in $\scrX$. In what follows $\mathsf{M}_i$ is an explicit invertible matrix, defined in \eqref{defn Mi}, that can be easily built using Dynkin combinatorics.

\begin{theorem}[\ref{thm: GV under flop}]\label{thm: GV translate intro}
With the notation as above,
\[
n_{\upbeta,\scrX^+_i}=
\begin{cases}
n_{\upbeta,\scrX}&\mbox{if }\upbeta\in \mathbb{Z} \Curve_i\\
n_{\mathsf{M}_i \upbeta, \scrX}&\mbox{else}.
\end{cases}
\]
\end{theorem} 
\noindent In other words, the set of GV invariants is exactly the same before and after the flop; they are just reindexed by the change of basis $\mathsf{M}_i$. Example~\ref{ex: running example GV after flop} illustrates this re-indexing in the case of the running Example~\ref{example: intro}.

\subsection{Applications}
The above results have a series of corollaries. The first is a direct and explicit proof of the Crepant Transformation Conjecture for germs of $3$-fold flopping contractions. Indeed, combining Theorems~\ref{GW=hyper intro} and \ref{thm: GV translate intro} allows us to easily extract the following, which recovers the main result of \cite{LiRuan}. We remark here that our simplified approach also gives more refined information, in the form of the explicit matrix $\mathsf{N}_{i}=(\mathsf{M}_i^{-1})^\star$ which identifies the quantum potentials. Our proof also avoids the use of symplectic cuts, side-stepping the associated technical difficulties.

\begin{cor}[\ref{thm: CTC}] Under the identification of the Novikov parameters given by the explicit matrix $\Msf_i$ of Theorem~\ref{thm: GV translate intro}, the quantum potentials of $\scrX$ and $\scrX_i^+$ coincide, up to a correction term which does not depend on the Novikov parameters, namely
\begin{equation} \label{eqn: CTC intro}  
\Phi^{\scrX_i^+}_{\mathsf{r}}\big(\upgamma_1, \upgamma_2,\upgamma_3\big) - \Phi^{\scrX}_{\mathsf{r}}(\mathsf{N}_i  \upgamma_1,\mathsf{N}_i \upgamma_2,\mathsf{N}_i \upgamma_3) = -(\upgamma_1 \cdot \Curve_i^+)(\upgamma_2 \cdot \Curve_i^+)(\upgamma_3\cdot \Curve_i^+) \sum_{k \geq 1} k^3 n_{k\Curve_i,\scrX} .
\end{equation}
Here $\mathsf{N}_i \colon \mathrm{H}^2(\scrX_{i}^+;\C) \to \mathrm{H}^2(\scrX;\C)$ is the dual of $\mathsf{M}_i^{-1}$. The above identification holds after a specific analytic continuation in the quantum parameters. 
\end{cor}
\noindent The correction terms on the right-hand side arise due to the non-compactness of $\scrX$ (see Remarks~\ref{rmk: no algebra} and \ref{rmk: CTC non-compact}). The key point is that the quantum potential of $\scrX_i^+$ can be effectively reconstructed from the quantum potential of $\scrX$. Thus, whilst the GW invariants themselves are not combinatorial, their transformation across the flop \emph{is} combinatorial, which is why we obtain such an elementary proof; compare \cite{LiRuan} and \cite{McLean}. 

The explicit matrix $\mathsf{N}_i$ turns out to have many different incarnations: it arises naturally as the image in K-theory of Bridgeland's flop functor, but more interestingly it can be calculated using very simple Dynkin-style combinatorics (see Remark~\ref{rem: matrix Mk}). However, simply by iterating and multiplying matrices, it is possible to obtain a direct isomorphism between the generating functions of any two crepant resolutions of $\Spec\scrR$.\medskip

 The second corollary is algebraic. The flopping contraction $\scrX\to\Spec \scrR$ has an associated contraction algebra $\CA$ \cite{DW1, DW3}, and it is known by Hua--Toda \cite{HuaToda} for single curves, and Toda \cite{TodaUtah} in general (see \ref{thm: Toda dim formula}), that the dimension of the contraction algebra is determined as the weighted sum of GV invariants 
\[
\dim_{\mathbb{C}}\CA=\sum_{\upbeta\in \Achow_1(\scrX)}n_\upbeta (\upbeta\cdot \mathds{1})^2
\]
where $\upbeta\cdot \mathds{1}$ is the sum of the entries of $\upbeta$.  For any curve $\Curve_i\subseteq\scrX$, the contraction algebra can be intrinsically mutated to obtain $\upnu_i\CA$, and this is the contraction algebra for the flop $\scrX^+_i\to\Spec\scrR$.
  
It is known \cite{Dugas, AugustTiltingTheory} that $\CA$ and $\upnu_i\CA$ are derived equivalent via a two-term tilting complex, but what is surprising here is that their dimension transforms in a very elementary manner, dictated by the K-theory of that derived equivalence.  
\begin{cor}[\ref{cor: Toda formula iterate}]\label{cor: Toda formula iterate intro}
Under mutation at vertex $i$, 
\[
\dim_{\mathbb{C}}\upnu_i\CA=\sum_{\upbeta\in \Achow_1(\scrX)}n_\upbeta \big(\,(\mathsf{M}^{-1}_i\upbeta)\cdot \mathds{1}\big)^2
\]
where $\mathsf{M}_i$ is the explicit matrix in Corollary~\textnormal{\ref{thm: GV translate intro}}.  
\end{cor}
The above is remarkable: it says that not only are there just finitely many algebras in the derived equivalence class of the finite dimensional algebra $\CA$ (by \cite{AugustFinite}), furthermore the dimensions of all the other algebras can be easily obtained combinatorially from the first.  The proof of Corollary~\ref{cor: Toda formula iterate intro} is slightly subtle, since it is not a priori clear that the GV invariants defined by Toda are the same as the GV invariants defined here, but this is all discussed in Appendix~\ref{sec: n beta eq n beta}.

\subsection*{Acknowledgements} 
We thank Tom Coates and Misha Feigin for helpful discussions on quantum cohomology, and Jenny August, Ben Davison, Okke van Garderen and Yukinobu Toda for wider discussions on GV invariants and contraction algebras.

\subsection*{Conventions} 
All cDV singularities and related algebraic geometry takes place over $\mathbb{C}$.  Vector spaces will be over $\mathbb{R}$, unless stated otherwise, and the complexification of a vector space $V$ will be written $V_{\mathbb{C}}$.

\section{Root theory, deformations and perturbations}
\noindent Fix an isolated $3$-fold cDV singularity $\Spec\scrR$ with a crepant resolution $f \colon \scrX \to\Spec\scrR$. There is a finite collection of exceptional complete curves in $\scrX$, contracted to a point $p\in \Spec\scrR$, and such that $f$ restricts to an isomorphism on the complement. In other words, $f \colon \scrX \to \Spec \scrR$ is a germ of a smooth $3$-fold flopping contraction, and conversely every such germ arises in this way \cite{Reid}.

\subsection{Elephants}\label{sec: elephant}
The pullback along $f$ of a general hyperplane section through $p \in \Spec\scrR$ is a \emph{partial} crepant resolution of an ADE surface singularity, the so-called general elephant \cite[(1.14)]{Reid}
\begin{equation}
\begin{array}{c}
\begin{tikzpicture}
\node (A) at (0,0) {$\scrX$};
\node (B) at (-2,0) {$Y$};
\node at (-3.4,-1) {$ \mathbb{C}^2/G\cong$};
\node (b) at (-2,-1) {$\Spec \scrR/g$};
\node (a) at (0,-1) {$\Spec \scrR$.};
\draw[->] (B)--(A);
\draw[->] (b)--(a);
\draw[->] (B)--(b);
\draw[->] (A)--(a);
\end{tikzpicture}
\end{array}\label{elephant pullback}
\end{equation}
Let $\Delta$ be the Dynkin diagram associated to $\C^2/G$, and let the composition
\[
Z \to Y\to \C^2/G
\]
be the full minimal resolution. By the McKay correspondence, the exceptional curves $\Curve_i \subseteq Z$ are indexed by the nodes $i \in \Delta$.   We write
\[
 \scrI \subseteq \Delta
\]
for the subset indexing those curves $\Curve_i \subseteq Z$ which are contracted by the morphism $Z\to Y$,  so that the complement $\scrIc = \Delta \setminus \scrI$ indexes the curves that survive. In particular
\[\{ \Curve_i \mid i \in \scrIc \}\]
forms the set of exceptional curves in both $Y$ and $\scrX$, and the group $\Achow_1(Y)=\Achow_1(\scrX)$ is freely generated by their cycle classes.

\begin{notation}\label{notation: Y_I}
Write $Y=Y_{\scrI}$ for the partial resolution of $\mathbb{C}^2/G$ obtained from the full minimal resolution $Z$ by blowing down the curves in $\scrI$.
\end{notation}

The geometry of $\scrX$ will be studied by viewing it as the total space of a one-parameter deformation of $Y_\scrI$. This requires detailed control over the associated root theory, which we establish in the following subsections.
 
\subsection{Root theory}\label{sec: Dynkin recap}
For any Dynkin diagram $\Delta$, let $\mathfrak{h}$ be the $\R$-vector space based by the set of simple roots $\{\upalpha_i \mid i \in \Delta \}$, so that
\[
\mathfrak{h}=\bigoplus_{i\in\Delta}\mathbb{R}\upalpha_i ,
\]
and write $\Uptheta = \mathfrak{h}^\star$ for the dual. The Weyl group $W$ acts naturally on both $\mathfrak{h}$ and $\Uptheta$.  For every positive root $\upalpha \in \mathfrak{h}$, write $\scrD_{\upalpha}\subseteq\mathfrak{h}$ for the perpendicular hyperplane, and write $\mathsf{H}_{\upalpha} \subseteq \Theta$
for the dual hyperplane.

\begin{notation}\label{notation: subset stuff}
For any subset $\scrI\subseteq \Delta$, consider the following data.
\begin{enumerate}
\item The complement $\scrIc = \Delta \setminus \scrI$.
\item The parabolic subgroup $W_{\scrI}\colonequals \langle s_i\mid i\in\scrI\rangle  \subseteq W$.
\item  The $\mathbb{R}$-vector space $\mathfrak{h}_{\scrI}$ obtained as the quotient of $\mathfrak{h}$ by the $\R$-subspace spanned by $\{ \upalpha_i \mid i \in \scrI \}$.  The  associated quotient map will be written
 \[
 \uppi_{\scrI} \colon \mathfrak{h} \to \mathfrak{h}_{\scrI}.
 \] 
Note that $\mathfrak{h}_{\scrI}$ has basis $\{ \uppi_{\scrI}(\upalpha_i) \mid i \in \scrIc\}$ and may be identified with the subspace of $\mathfrak{h}$ based by $\{ \upalpha_i \mid i \in \scrIc \}$.
\item The \emph{restricted positive roots} in $\mathfrak{h}_\scrI$, which are precisely the non-zero images of positive roots under $\uppi_{\scrI}$.
\item For $\upvartheta_i\in\mathbb{R}$ with $i\in\Delta$, write $(\upvartheta_i)=\sum_{i\in\Delta}\upvartheta_i\upalpha_i^\star$, and consider
\[
\Uptheta_{\scrI}\colonequals \{ (\upvartheta_i)\in\Uptheta\mid \upvartheta_i=0\mbox{ for all }i\in\scrI\}\subseteq\Uptheta.
\]
The reflecting hyperplanes in $\Uptheta$ intersect $\Uptheta_{\scrI}$, and in this way $\Uptheta_{\scrI}$ inherits the structure of a finite hyperplane arrangement. Note that $\Uptheta_{\scrI}$ has basis $\{\upalpha_i^\star\mid i \in \scrIc\}$. Of course, $\mathfrak{h}_{\scrI}$ and $\Uptheta_{\scrI}$ are dual, and both have dimension $|\scrIc|$. 
 
\item The set $\mathsf{Cham}(\Delta,\scrI)$ which indexes chambers of $\Uptheta_\scrI$ \cite[1.8]{IyamaWemyssTits}. Combinatorially, $\mathsf{Cham}(\Delta,\scrI)$ can be defined as the set of all pairs $(x,\scrK)$ with $x\in W$ and $\scrK\subseteq\Delta$ for which $W_\scrI x=xW_\scrK$ and $\mathrm{length}(x)=\mathrm{min}\{\mathrm{length}(y)\mid y\in xW_\scrK\}$.
 \end{enumerate}
 \end{notation}
For a given restricted positive root $\upbeta \in \mathfrak{h}_{\scrI}$, there are in general many different positive roots $\upalpha \in \mathfrak{h}$ such that $ \uppi_{\scrI}(\upalpha)=\upbeta$. The following result controls the possible lifts, and is a very mild generalisation of \cite[Lemma~2.4]{BryanKatzLeung}. It is used in the proof of Theorem~\ref{thm: KM Theorem 1(c)}, which in turn is used to relate enumerative invariants to hyperplane arrangements in Section~\ref{sec: enumerative invariants}.

\begin{lemma} \label{lem: Dynkin combinatorics}
For any ADE Dynkin diagram $\Delta$, and any subset $\scrI\subseteq\Delta$, let $\upalpha,\upalpha^\prime \in \mathfrak{h}$ be positive roots such that $\uppi_\scrI(\upalpha), \uppi_{\scrI}(\upalpha^\prime) \in \mathfrak{h}_{\scrI}$ are non-zero. Then the following are equivalent.
\begin{enumerate}
\item $\uppi_{\scrI}(\upalpha)=\uppi_{\scrI}(\upalpha^\prime)$.
\item $\upalpha$ and ${\upalpha^\prime}$ are identified under the action of $W_{\scrI}$ on $\mathfrak{h}$.
\item $\scrD_\upalpha$ and $\scrD_{\upalpha^\prime}$ are identified under the action of $W_{\scrI}$ on $\mathfrak{h}$.
\end{enumerate}
\end{lemma}
\begin{proof} 
Since $\scrI$ is fixed, to ease notation set $\uppi=\uppi_{\scrI}$.\medskip

\noindent (1)$\Rightarrow$(2) This is the only difficult part, and proceeds by case analysis. Consider first the $A_n$ root system, where the positive roots are precisely the connected chains of $1$s on the Dynkin graph
%\upalpha_{ij}
\[
\begin{array}{c@{\hspace{2pt}}c@{\hspace{2pt}}c@{\hspace{2pt}}c@{\hspace{2pt}}c@{\hspace{1pt}}c@{\hspace{2pt}}c@{\hspace{2pt}}c@{\hspace{2pt}}c@{\hspace{2pt}}c@{\hspace{2pt}}c@{\hspace{2pt}}}
\upalpha_{ij}\colonequals \ 0&\hdots &0&1& \hdots &1&0&\hdots&0 
\quad (1 \leq i \leq j \leq n).\\[-0.5mm]
&&& i && j
\end{array}
\]
By definition of the action of $W$ on $\mathfrak{h}$, the reflection $s_i$ acts by
\begin{equation}
s_i(\upalpha_{ij})=\upalpha_{i+1\, j} \qquad \text{for $i<j$}\label{action Type A}
\end{equation}
which has the effect of replacing the leftmost $1$ with a $0$. Similarly $s_j(\upalpha_{ij})=\upalpha_{i\, j-1}$ for $i < j$, which has the effect of replacing the rightmost $1$ with a $0$.

Suppose now that we are given positive roots $\upalpha_{ij}$ and $\upalpha_{kl}$, and assume without loss of generality that $i \leq k$. Since $\uppi(\upalpha)=\uppi(\upalpha^\prime)$ is nonzero, the two chains of $1$s must overlap, and any position at which they do not overlap must be indexed by an element of $\scrI$. By iteratively applying \eqref{action Type A}, we can shorten both $\upalpha_{ij}$ and $\upalpha_{kl}$ using only elements of $W_\scrI$ and force them to line up on the left, at the leftmost element belonging to the overlap. They can similarly be forced to line up on the right as well, proving the claim. 

The case $D_n$ for $n \geq 4$ is similar in spirit, albeit more involved. As usual, write $\upalpha_1,\ldots,\upalpha_n$ for the simple roots, then the positive roots are given by the following linear combinations. Note that the support of each is a connected subset of the Dynkin graph.
\begin{align*}
%\upalpha_{ij}^{pq}
& \begin{array}{c@{\hspace{2pt}}c@{\hspace{2pt}}c@{\hspace{2pt}}c@{\hspace{2pt}}c@{\hspace{1pt}}c@{\hspace{2pt}}c@{\hspace{2pt}}c@{\hspace{2pt}}c@{\hspace{2pt}}c@{\hspace{2pt}}l@{\hspace{2pt}}}
&&&&&&&&q &&\qquad 1 \leq i \leq j \leq n-2,\\ 
\upalpha_{ij}^{pq}\colonequals \ 0 & \hdots &0&1& \hdots &1&0&\hdots &0&p& \qquad p,q\in \{0,1\},\\
%\multicolumn{1}{c}{}
%&&\multicolumn{3}{@{}c@{}}{${\upbracefill}$} \\[-1mm]
%\multicolumn{1}{c}{}
%&& \multicolumn{3}{c}{\scriptstyle i}
&&& i && j &&&&& \qquad (p,q) \neq (0,0) \Rightarrow j=n-2 \end{array}\\ \,
%\upbeta_{ij}
& \begin{array}{c@{\hspace{2pt}}c@{\hspace{2pt}}c@{\hspace{2pt}}c@{\hspace{2pt}}c@{\hspace{1pt}}c@{\hspace{2pt}}c@{\hspace{2pt}}c@{\hspace{2pt}}c@{\hspace{2pt}}c@{\hspace{2pt}}c@{\hspace{2pt}}}
&&&&&&&& 1\\ 
\upbeta_{ij}\colonequals \ 0&\hdots & 0 & 1 & \hdots & 1 & 2 & \hdots & 2 & 1 & \qquad 1 \leq i < j \leq n-2 \\[-0.5mm]
%\multicolumn{1}{c}{}
%&&\multicolumn{3}{@{}c@{}}{${\upbracefill}$} \\[-1mm]
%\multicolumn{1}{c}{}
%&& \multicolumn{3}{c}{\scriptstyle i}
&&& i &&& j
\end{array}
\end{align*}
In addition, the collection $\upalpha_{ij}^{pq}$ includes two special cases $\upalpha_{n-1}$ and $\upalpha_n$, which we interpret as $(p,q)=(1,0)$ and $(p,q)=(0,1)$ with the string of $1$s between $i$ and $j$ being empty.

For every $1 \leq i \leq n$ the reflection $s_i$ acts on the coefficient in the $i$th position by negating it and then adding the sum of the coefficients in adjacent positions. The coefficients in all other positions are left unchanged.

Consider now positive roots $\upalpha, \upalpha^\prime$ with $\uppi(\upalpha)=\uppi(\upalpha^\prime)\neq 0$. We work through the different cases. The first case is $\upalpha=\upalpha_{ij}^{pq}$ and $\upalpha^\prime = \upalpha_{kl}^{rs}$. If $(p,q) \neq (r,s)$ then we can use the reflections $s_{n-1},s_n$ to transform $\upalpha,\upalpha^\prime$ until $(p,q)=(r,s)$. For instance, if $(p,q)=(0,1)$ and $(r,s)=(0,0)$ then since $\uppi(\upalpha)=\uppi(\upalpha^\prime)$ we must have $n \in \scrI$. Applying $s_n \in W_{\scrI}$ to $\upalpha$ replaces $(p,q)=(0,1)$ with $(p,q)=(0,0)$. The other cases are similar, and so we may assume that $(p,q)=(r,s)$.  Given this, we follow the strategy in the $A_n$ case. The only difference is when $j=n-2$ and $l < n-2$, in which case $(p,q)=(r,s)=(0,0)$ since the chain of $1$s is connected.  Then $s_{n-2}\in W_{\scrI}$ and $s_{n-2} (\upalpha_{i\,n-2}^{00}) = \upalpha_{i\, n-3}^{00}$.  This, and all other situations, then simply mirror the  $A_n$ proof.

The next case is $\upalpha = \upalpha_{ij}^{pq}$ and $\upalpha^\prime = \upbeta_{kl}$.
If $i \leq k$ then necessarily $j \geq k$ since otherwise the supports do not overlap. In particular $s_i, s_{i+1}, \ldots, s_{k-1} \in W_{\scrI}$ which gives
\begin{equation*} 
(s_{k-1} \circ \cdots \circ s_i)(\upalpha_{ij}^{pq}) = \upalpha_{kj}^{pq}.
\end{equation*}
Similarly if $i \geq k$ we use $s_{k} \circ \cdots \circ s_{i-1}$ to achieve the same transformation. For the next step, notice that $s_{j+1},\ldots,s_{n-2} \in W_{\scrI}$, and applying these left-to-right gives
\begin{equation*} 
(s_{n-2} \circ \cdots \circ s_{j+1})(\upalpha_{kj}^{pq}) = \upalpha_{k\,n-2}^{pq}. 
\end{equation*}
As in the previous case, if $(p,q) \neq (1,1)$ then we may use the reflections $s_{n-1},s_n$ to transform $\upalpha_{k\,n-2}^{pq}$ into $\upalpha_{k\,n-2}^{11}$. Now, going back from right-to-left gives the required
\begin{equation*} 
(s_l \circ \cdots \circ s_{n-2})(\upalpha_{k\,n-2}^{11}) = \upbeta_{kl}.
\end{equation*}
The next case is $\upalpha=\upbeta_{ij}$ and $\upalpha^\prime = \upbeta_{kl}$, where without loss of generality $i \geq k$. But then $s_{i-1},\ldots,s_k \in W_{\scrI}$ which is applied directly to give $(s_k \circ \cdots \circ s_{i-1})(\upbeta_{ij}) = \upbeta_{kj}$. If $j \geq l$, then $s_{j-1},\ldots,s_l \in W_{\scrI}$, and $(s_{l}\circ \cdots \circ s_{j-1})(\upbeta_{kj}) = \upbeta_{kl}$. The case $j \leq l$ is similar. 

This completes the proof for $D_n$, except for some special cases involving $\upalpha_{n-1}
$ and $\upalpha_n$. These are more elementary, and so are left to the reader.

The remaining cases $E_6,E_7,E_8$ encompass finitely many possibilities. It is possible to verify these by hand, but it is also possible to employ computer algebra \cite{magma}. Source code is available from the authors upon request. This completes the proof of (1)$\Rightarrow$(2).\medskip

\noindent (2)$\Rightarrow$(1) This holds since applying elements of $W_\scrI$ to a given positive root cannot change the coefficients associated to elements of $\scrIc$.\medskip
 
\noindent (2)$\Rightarrow$(3) If $s \in W$ and $v \in \mathfrak{h}$ then $s(\scrD_v)=\scrD_{s(v)}$ since $s$ preserves the Cartan pairing. In particular this applies when $v=\upalpha$ is a positive root and $s(v)=\upalpha^\prime$ is its image.\medskip
 
\noindent (3)$\Rightarrow$(2) Suppose there is an element $s \in W_{\scrI}$ such that $s(\scrD_{\upalpha})=\scrD_{\upalpha^\prime}$. We then have $\scrD_{\upalpha^\prime} = \scrD_{s(\upalpha)}$ and since both $\upalpha^\prime$ and $s(\upalpha)$ are roots it follows that $s(\upalpha)$ equals either $\upalpha^\prime$ or $-\upalpha^\prime$. We claim that $s(\upalpha)$ must be a positive root, thus ruling out the latter possibility. If $s=s_i$ for some $i \in \scrI$, then $s_i$ permutes the set of positive roots excluding $\upalpha_i$. But certainly $\upalpha \neq \upalpha_i$ since $\uppi(\upalpha) \neq 0$, so $s_i(\upalpha)$ must be a positive root. The argument for general $s$ follows by induction, since $\uppi(\upalpha)=\uppi(s_i(\upalpha))$ for any $\upalpha$ and any $s_i$ with $i \in \scrI$.\end{proof}

\subsection{Hyperplane arrangements: finite and infinite} \label{sec: hyperplane arrangements} To the above data of a Dynkin diagram $\Delta$ and a subset $\scrI \subseteq \Delta$ it is possible to associate two hyperplane arrangements encoding the set of restricted positive roots: one finite and one infinite \cite{IyamaWemyssTits}.  Recalling Notation~\ref{notation: subset stuff}, both arrangements live inside $\Uptheta_{\scrI}\cong\mathbb{R}^{|\scrIc|}$.

The real vector space $\mathfrak{h}_\scrI$ is dual to $\Uptheta_{\scrI}$, and so for each restricted positive root $0\neq\upbeta=\uppi_\scrI(\upalpha)\in \mathfrak{h}_\scrI$ we may consider the dual hyperplane
\[ 
\Hyp_\upbeta\colonequals \{ (\upvartheta_i) \mid \textstyle\sum\upbeta_i\upvartheta_i=0 \} \subseteq \Uptheta_{\scrI}. 
\]
Since there are only finitely many restricted roots, the collection of $\Hyp_\upbeta$ forms a finite hyperplane arrangement in $\Uptheta_{\scrI}$, which we refer to as the \emph{finite linear arrangement}, namely
\begin{equation} \label{eqn: finite arrangement} 
\scrH_\scrI \colonequals \left\{ \Hyp_\upbeta \mid \upbeta \text{ is a restricted positive root}\right\}. 
\end{equation}
The above list includes repetitions, whenever two restricted roots are proportional. As in Example~\ref{example: intro}, we remember these repetitions by attaching a finite list of multiplicities to each hyperplane. We refer to this data as the \emph{enhanced finite arrangement}. 

\begin{remark}In the setting of Subsection~\ref{sec: elephant}, $\scrH_\scrI$ is the set of walls in the movable cone of $\scrX$ \cite{Pinkham,HomMMP}. However, the movable cone does not remember multiplicities.\end{remark}

To define the infinite arrangement, for each restricted positive root $\upbeta$, consider the infinite disjoint union of affine hyperplanes in $\Uptheta_\scrI$ defined by
\[ 
\Hyp_\upbeta^{\aff} = \{ (\upvartheta_i)\mid \textstyle\sum_{i\in\scrIc}\upbeta_i\upvartheta_i=z\mbox{ for some } z\in\mathbb{Z} \}.
\]
The infinite affine arrangement $\scrH_\scrI^\aff$ is then defined to be
\begin{equation} \label{eqn: infinite arrangement} 
\scrH_\scrI^\aff \colonequals \bigcup_\upbeta \Hyp_\upbeta^\aff .
\end{equation}
There is an inclusion $\scrH_\scrI \subseteq \scrH_\scrI^\aff$, as $\scrH_\scrI$ can be recovered from $\scrH_\scrI^\aff$ as the subset of hyperplanes which pass through the origin. On the other hand to construct $\scrH_\scrI^\aff$ from the finite arrangement $\scrH_\scrI$ it is necessary to remember the multiplicities, since $\Hyp_{2\upbeta}^{\aff}\neq\Hyp_\upbeta^{\aff} $.

\subsection{Simultaneous partial resolution} \label{sec: sim partial res I} The enumerative geometry of the $3$-fold $\scrX$ will be studied by replacing $\scrX$ with a generic perturbation, a strategy employed by many authors \cite{MorrisonKaehler,WilsonGW,BryanKatzLeung}. In the first instance this will allow us to qualitatively characterise the GV invariants, and extract the poles of the quantum product. However the real strength in this approach, and indeed our new contribution, is to use the wall crossing formula from \cite{IyamaWemyssTits} to construct iterated flops via simultaneous (partial) resolution. This iteration step is harder, and so will be delayed until Section~\ref{sec: flops via sim res}.

In this subsection we simply recall the necessary background and set notation, largely following \cite[Section~2]{BryanKatzLeung}, together with \cite{Brieskorn, Pinkham, Friedman, KatzMorrison}.

\subsubsection{Simultaneous Resolution} \label{sec: pref sim res}
With notation as in Subsection~\ref{sec: Dynkin recap}, given any Kleinian singularity $\mathbb{C}^2/G$ with associated Dynkin diagram $\Delta$, consider the complex vector space $\mathfrak{h}_{\mathbb{C}}=\bigoplus_{i\in\Delta}\mathbb{C}\upalpha_i$, based by the simple roots.   Write $\Spec\scrV$ for a versal deformation of $\mathbb{C}^2/G$, then as is very well known, base changing with respect to the Weyl group
\[
\begin{tikzpicture}
\node (U) at (2,-1.5) {$\Spec \scrW$};
\node (V) at (4,-1.5) {$\Spec \scrV$};
\node (PRes) at (2,-3) {$\mathfrak{h}_{\mathbb{C}}$};
\node (Res) at (4,-3) {$\mathfrak{h}_{\mathbb{C}}/W$};
\draw[->,densely dotted] (U)--(V);
\draw[->] (V)--(Res);
\draw[->,densely dotted] (U)--(PRes);
\draw[->] (PRes)--(Res);
\end{tikzpicture}
\]
gives $\Spec\scrW\to\mathfrak{h}_{\mathbb{C}}$, which admits a simultaneous resolution.  Since $\dim\scrW\geq 3$, there are in fact many such simultaneous resolutions, since minimal models are not unique.

In \cite[Theorem 1]{KatzMorrison} Katz--Morrison construct a particular simultaneous resolution, from a particular $\Spec\scrV\to\mathfrak{h}_{\mathbb{C}}/W$, for which positive roots and their hyperplanes control those curve classes that survive under deformation \cite[Theorem 1(c)]{KatzMorrison}. We will recap this result in greater generality in Theorem~\ref{thm: KM Theorem 1(c)} below, but for now write $\scrZ\to\Spec\scrW$ for this preferred resolution. Katz--Morrison refer to their particular choice of $\scrZ$ as the standard simultaneous resolution \cite[Section~6]{KatzMorrison}.

\subsubsection{Simultaneous Partial Resolution} 
Given any subset $\scrI\subseteq \Delta$, consider the standard simultaneous resolution $\scrZ\to\Spec\scrW$ from Subsection~\ref{sec: pref sim res} above.  As explained in \cite[above Theorem 3]{KatzMorrison} following \cite{Pinkham} it is possible to blow down $\scrZ$ at the curves in $\scrI$ and take the quotient by $W_{\scrI}$ to obtain $\scrY_{\scrI}$, which sits in the following commutative diagram
\begin{equation}
\begin{array}{c}
\begin{tikzpicture}
%\node (X) at (0,0) {$\scrZ$};
\node (ZZ) at (0,1.5) {$\scrZ$};
\node (B) at (0,0) {$\scrY_\scrI^\dag$};
\node (Z) at (2,0) {$\scrY_{\scrI}$};
\node (R) at (0,-1.5) {$\Spec \scrW$};
\node (U) at (2,-1.5) {$\Spec \scrV_{\scrI}$};
\node (V) at (4,-1.5) {$\Spec \scrV$};
\node (t) at (0,-3) {$\mathfrak{h}_\C$};
\node (PRes) at (2,-3) {$\mathfrak{h}_{\mathbb{C}}/W_{\scrI}$};
\node (Res) at (4,-3) {$\mathfrak{h}_{\mathbb{C}}/W$};
%\draw[->,densely dotted] (X)--(Z);
%\draw[->] (X)--(R);
\draw[->] (ZZ)--(B);
\draw[->] (B) -- (Z);
\draw[->] (B) -- (R);
%\draw[->] (ZZ) -- (R);
\draw[->] (Z)--node[right]{$\scriptstyle \mathsf{g}_{\scrI}$}(U);
\draw[->] (R)--(U);
\draw[->] (U)--(V);
\draw[->] (V)--(Res);
\draw[->] (R)--(t);
\draw[->] (U)--node[right]{$\scriptstyle \mathsf{h}_{\scrI}$}(PRes);
\draw[->] (t)--(PRes);
\draw[->] (PRes)--(Res);
\end{tikzpicture}
\end{array}\label{preferred sim partial res}
\end{equation}
with all squares cartesian.  By construction, the fibre $(\mathsf{h}_\scrI\circ \mathsf{g}_{\scrI})^{-1}(0)$ is the partial resolution of $\mathbb{C}^2/G$ obtained from the full minimal resolution by blowing down the curves in $\scrI$. Namely, recalling Notation~\ref{notation: Y_I}, $(\mathsf{h}_\scrI\circ \mathsf{g}_{\scrI})^{-1}(0)=Y_{\scrI}$. 

In a similar way as in Subsection~\ref{sec: pref sim res}, the middle morphism $\mathsf{h}_\scrI\colon\Spec\scrV_{\scrI}\to\mathfrak{h}_{\mathbb{C}}/W_{\scrI}$ admits simultaneous \emph{partial} resolutions.  Again these are not unique, however we will refer to the choice $\scrY_{\scrI}\to\Spec\scrV_\scrI$ constructed above as the standard simultaneous partial resolution associated to $\scrI$.  From our perspective, the point is that $\scrY_{\scrI}$ is precisely the partial simultaneous resolution for which Theorem~\ref{thm: KM Theorem 1(c)} below holds.

\subsection{Surface deformations via simultaneous partial resolution} \label{sec: deformation theory}
Fix a subset $\scrI\subseteq\Delta$ and consider the composition
\[ \mathsf{s}_\scrI = \mathsf{h}_{\scrI}\circ \mathsf{g}_{\scrI}\colon \scrY_{\scrI}\to \mathfrak{h}_{\mathbb{C}}/W_{\scrI}\]
from \eqref{preferred sim partial res}. This is a versal deformation of the surface $Y_\scrI$.

\begin{definition} \label{def: standard disc locus}
The standard discriminant locus
\[
\scrD_{\scrI} \subseteq \mathfrak{h}_{\mathbb{C}}/W_{\scrI}
\]
is the set of points $p \in \mathfrak{h}_{\mathbb{C}}/W_{\scrI}$ such that the fibre $\mathsf{s}_\scrI^{-1}(p)$ contains a complete curve. 
\end{definition}
There is a similar definition of a discriminant locus associated to any simultaneous partial resolution: the word standard in Definition~\ref{def: standard disc locus} emphasises the choice made in \eqref{preferred sim partial res}.  The following discussion draws heavily on \cite[Theorem 1]{KatzMorrison} as used in \cite[Proposition 2.2]{BryanKatzLeung}, while also incorporating Lemma~\ref{lem: Dynkin combinatorics} above to relate the resulting combinatorics to the enhanced movable cone. 

To set notation, recall that $\scrD_{\upalpha}\subseteq\mathfrak{h}$ is the hyperplane perpendicular to $\upalpha$, and let $\scrD_{\upalpha,\C}\subseteq \mathfrak{h}_\C$ denote its complexification. Recall from Notation~\ref{notation: subset stuff} that for $\scrI \subseteq \Delta$ there is a quotient map $\uppi_\scrI \colon\mathfrak{h}\to\mathfrak{h}_{\scrI}$, where the vector space $\mathfrak{h}_{\scrI} $ has basis $\{ \uppi_{\scrI}(\upalpha_i) \mid i \in \scrIc\}$, and that there is a   natural identification
\begin{align} 
\mathfrak{h}_{\scrI} & \cong \Achow_1(Y_{\scrI})_{\mathbb{R}}\label{eq:ident A with comb}\\
\uppi_{\scrI}(\upalpha_i) & \mapsto \Curve_i.\nonumber
\end{align}
Every restricted positive root $\uppi_\scrI(\upalpha) \in \mathfrak{h}_{\scrI}$ has non-negative integer coefficients, and so may be interpreted as a curve class $\uppi_\scrI(\upalpha)=\upbeta \in \Achow_1(Y_{\scrI})$.

\begin{thm}[Katz--Morrison] \label{thm: KM Theorem 1(c)} For any subset $\scrI\subseteq\Delta$, the following statements hold.
\begin{enumerate}
\item The standard discriminant locus $\scrD_\scrI \subseteq \mathfrak{h}_\C/W_\scrI$ from Definition~\ref{def: standard disc locus} decomposes as
\[
\scrD_{\scrI}=\bigcup_{\uppi_{\scrI}(\upalpha)\neq 0}\,\,\scrD_{\upalpha,\C}/W_\scrI
\]
where the union is over all positive roots $\upalpha$ such that $\uppi_{\scrI}(\upalpha)\neq 0$. The irreducible components $\scrD_{\upalpha,\C}/W_\scrI \subseteq \scrD_\scrI$ are indexed by the restricted positive roots $\uppi_\scrI(\upalpha)$.
\item For $p\in\scrD_{\upalpha,\C}/W_\scrI$ the fibre $\mathsf{s}_\scrI^{-1}(p)$ is a deformation of $Y_{\scrI}$ containing a complete curve of class $\upbeta \colonequals \uppi_\scrI(\upalpha)$. If in addition $p$ does not belong to any other component of $\scrD_\scrI$, then this is the only complete curve in $\mathsf{s}_\scrI^{-1}(p)$.
\item If $\upbeta \in \Achow_1(Y_\scrI)$ is not a restricted positive root, then there are no deformations of $Y_\scrI$ containing a complete curve of class $\upbeta$.
\end{enumerate}
\end{thm}
\begin{proof} 
(1) If $\scrI=\emptyset$ then by \cite[Theorem 1(3)]{KatzMorrison} (see also \cite[Proposition~2.2]{BryanKatzLeung}) there is a decomposition of the standard discriminant locus
\[ 
\scrD_\emptyset = \bigcup_\upalpha \scrD_{\upalpha,\C} \subseteq \mathfrak{h}_\C 
\]
where the union is over all positive roots $\upalpha$. The analogous decomposition for general $\scrI$ follows by considering the standard simultaneous resolution $\scrZ$ of the standard simultaneous partial resolution $\scrY_\scrI$ from \ref{sec: sim partial res I}
\bcd
\scrZ \ar[r] \ar[d,"\mathsf{t}_\scrI"] & \scrY_\scrI \ar[d,"\mathsf{s}_\scrI"] \\
\mathfrak{h}_\C \ar[r,"\upphi_\scrI"] & \mathfrak{h}_\C/W_\scrI.
\ecd
The map $\scrZ \to \scrY_\scrI$ is given by blowing down $\scrZ$ at the curves in $\scrI$ and then taking the quotient by $W_\scrI$.

Fixing a point $p \in \mathfrak{h}_\C$, it follows that the fibre $\mathsf{s}_\scrI^{-1}(\upphi_\scrI(p))$ contains a complete curve if and only if the fibre $\mathsf{t}_\scrI^{-1}(p)$ contains a complete curve which is not blown down. Again by \cite[Theorem 1(3)]{KatzMorrison}, the fibre $\mathsf{t}_\scrI^{-1}(p)$ contains a complete curve if and only if $p \in \scrD_{\upalpha,\C}$ for some positive root $\upalpha$, and this curve is not blown down if and only if $\uppi_\scrI(\upalpha) \neq 0$. This produces the desired decomposition of $\scrD_\scrI$. It then follows from Lemma~\ref{lem: Dynkin combinatorics} that the components $\scrD_{\upalpha,\C}/\scrW_\scrI$ are indexed by the restricted positive roots $\uppi_\scrI(\upalpha)$, i.e. $\scrD_{\upalpha,\C}/W_{\scrI}=\scrD_{\upalpha^\prime,\C}/W_{\scrI}$ if and only if $\uppi_\scrI(\upalpha)=\uppi_\scrI(\upalpha^\prime)$.

\noindent (2)(3) First recall what it means for a curve in $\mathsf{s}_\scrI^{-1}(p)$ to have class $\upbeta \in \Achow_1(Y_\scrI)$. The inclusion of the central fibre $\mathsf{i}_0 \colon \mathsf{s}_\scrI^{-1}(0) = Y_\scrI \hookrightarrow \scrY_\scrI$ induces an isomorphism
\[ 
\mathsf{i}_{0\star} \colon \Achow_1(Y_\scrI) \xrightarrow{\sim} \Achow_1(\scrY_\scrI).
\]
Now consider an arbitrary fibre $\mathsf{s}_\scrI^{-1}(p)$ with inclusion $\mathsf{i}_p \colon \mathsf{s}_\scrI^{-1}(p) \hookrightarrow \scrY_\scrI$. If  $\Curve \subseteq \mathsf{s}_\scrI^{-1}(p)$ is a complete curve, then $\Curve$ has class $\upbeta \in \Achow_1(Y_\scrI)$ if $\mathsf{i}_{p\star} \Curve = \mathsf{i}_{0\star} \upbeta$. The same definition applies to the full simultaneous resolution $\scrZ$. Both (2) and (3) are known for $\scrZ$ by \emph{loc. cit.}, and the general case follows by tracking curve classes from $\scrZ$ to $\scrY_\scrI$.
\end{proof}

\begin{remark} \label{rmk: proportional but distinct}
Consider classes $\upbeta,\upbeta^\prime \in \Achow_1(Y_{\scrI})$ which are proportional but distinct,  i.e.\ $k\upbeta = k^\prime \upbeta^\prime$ for some distinct integers $k,k^\prime \geq 1$. If both $\upbeta, \upbeta^\prime$ are  restricted positive roots, then the lifts $\upalpha,\upalpha^\prime$ will \emph{not} be proportional, by the root system axioms. In particular, the corresponding components of the standard discriminant locus $\scrD_{\scrI}$ will be distinct. Every component of the discriminant locus therefore corresponds to a unique curve class $\upbeta$. This is in contrast to the components of the hyperplane arrangement $\scrH_\scrI \subseteq \Uptheta_\scrI$ from Subsection~\ref{sec: hyperplane arrangements}. \end{remark}

\begin{remark} \label{rem: def of E}
The description of the standard discriminant locus $\scrD_\scrI$ in Theorem~\ref{thm: KM Theorem 1(c)} is a union over positive roots $\upalpha \in \mathfrak{h}$ such that $\uppi_\scrI(\upalpha) \neq 0$. The complementary union of hyperplane quotients
\[
\scrE_{\scrI}=\bigcup_{\uppi_{\scrI}(\upalpha)= 0}\,\,\scrD_{\upalpha,\C}/W_\scrI
\]
parametrise points $p \in \mathfrak{h}_{\mathbb{C}}/W_{\scrI}$ such that the fibre $\mathsf{s}_\scrI^{-1}(p)$ is singular. Clearly
\[ 
\scrD_\emptyset/W_\scrI = \scrD_\scrI \cup \scrE_\scrI.
\]
For a generic point $p \in \scrE_\scrI$, the fibre $\mathsf{s}_\scrI^{-1}(p)$ contains a single $A_1$ singularity. The locus $\scrE_\scrI$ will play a less central role than $\scrD_{\scrI}$. The translation between our notation and that of \cite[Section~2]{BryanKatzLeung} is as follows: $\scrD_{\scrI} = D^{\operatorname{curv}}, \scrE_{\scrI} = D^{\operatorname{sing}}$ and $\scrD_{\scrI} \cup \scrE_{\scrI} = D$.
\end{remark}

\subsection{$3$-fold perturbations via surface deformations} \label{sec: deforming to Xt} \label{sec: perturbing target}
Given the flopping contraction $\scrX\to\Spec\scrR$, a choice of local equation for the hypersurface $\Spec\scrR/g \subseteq \Spec \scrR$ produces a flat family $\Spec \scrR \to \formaldisc$ over a formal disc, with central fibre an ADE surface singularity. By composition this produces a flat family $\scrX \to \formaldisc$ with central fibre the partial resolution $Y\cong Y_{\scrI}$ of the ADE singularity \cite{Reid}. This exhibits $\scrX$, respectively $\Spec\scrR$, as the total space of a one-parameter deformation of the surface $Y_{\scrI}$, respectively $\Spec \scrR/g$. These deformations are induced by an associated classifying map
\[
\upmu \colon \formaldisc \to  \mathfrak{h}_{\mathbb{C}}/W_{\scrI}
\]
where $\mathfrak{h}_{\mathbb{C}}/W_{\scrI}$ is as described in the previous subsection, and the contraction $\scrX \to \Spec \scrR$ is obtained from the simultaneous partial resolution of Subsection~\ref{sec: sim partial res I} by base change
\bcd
\scrX \ar[d,"f" left] \ar[r] & \scrY_{\scrI} \ar[d,"\mathsf{g}_{\scrI}"] \\
\Spec \scrR \ar[r] \ar[d] & \Spec \scrV_\scrI \ar[d,"\mathsf{h}_\scrI"] \\
\formaldisc \ar[r,"\upmu"] & \mathfrak{h}_\C/W_\scrI.
\ecd
The central fibre of $\Spec\scrR \to \formaldisc$ is the ADE singularity corresponding to $\Delta$, so $\upmu(0) =  0$. On the other hand, since $\scrX$ contains no complete curves outside of $Y$ the map $\upmu$ does not intersect the discriminant locus away from the origin, so $\upmu^{-1}(\scrD_\scrI) = \upmu^{-1}(0) = 0$. 

As in \cite[below Lemma 2.7]{BryanKatzLeung} there exists a one-parameter perturbation of $\upmu$
\[ 
(\upmu_t)_{t\in [0,\upvarepsilon]} \colon \formaldisc \times [0,\upvarepsilon] \to \mathfrak{h}_{\mathbb{C}}/W_{\scrI} 
\]
such that $\upmu_0=\upmu$ and for $t \neq 0$ the following transversality condition is satisfied
\begin{equation}\label{condition transversality of muprime} 
\text{$\upmu_{t}$ intersects $\scrD_{\scrI}\cup \scrE_{\scrI}$ transversely and away from codimension two strata}
 \end{equation}
where $\scrE_{\scrI}$ is defined in Remark~\ref{rem: def of E}. Furthermore, making $\upvarepsilon$ smaller if necessary, we can assume that $\upmu_t^{-1}(\scrD_{\scrI})$ is bounded away from the boundary of $ \formaldisc$.

The spaces $\scrX_{t\neq 0}$ give generic perturbations of the target $\scrX_0=\scrX$. The following is well-known, and will be used to reduce the enumerative geometry of $\scrX_t$, locally, to that of the Atiyah flop.
 
 \begin{lemma} 
For any $t \neq 0$, the total space $\scrX_t$ of the family of surfaces associated to $\upmu_t$ is a smooth $3$-fold. Further, every complete curve in $\scrX_t$ is isolated, smooth, and rational, with normal bundle isomorphic to $\scrO_{\PP^1}(-1)\oplus \scrO_{\PP^1}(-1)$. 
\end{lemma}
 \begin{proof} 
 This is essentially \cite[Proposition 2.2]{BryanKatzLeung}, which itself is extracted from the proof of \cite[Theorem 1]{KatzMorrison}. \end{proof}
 
Write $\mathfrak{X}$ for the $4$-dimensional total space of the entire family $(\upmu_t)_{t}$. Then as explained in \cite[Section~3]{Wilson}, pulling back along inclusions of fibres induces isomorphisms
\begin{equation}
\mathrm{H}^2(Y_{\scrI};\mathbb{Z})\xleftarrow{\cong} \mathrm{H}^2(\scrX;\mathbb{Z})\xleftarrow{\cong} \mathrm{H}^2(\mathfrak{X};\mathbb{Z})\xrightarrow{\cong} \mathrm{H}^2(\scrX_t;\mathbb{Z})\label{eqn: lift line bundles}
\end{equation}
for any $t$.  Any class $L$ in $\mathrm{H}^2(\scrX,\mathbb{Z})\cong\Pic(\scrX)$ thus induces an invertible sheaf $\scrL$ on $\mathfrak{X}$ with $\scrL|_{\scrX_0}=L$. Similarly, pushing forward curve classes along the inclusion of fibres induces isomorphisms
 \begin{equation}
 \Achow_1(Y_{\scrI})\xrightarrow{\cong} \Achow_1(\scrX) \xrightarrow{\cong} \Achow_1(\mathfrak{X}) \xleftarrow{\cong} \Achow_1(\scrX_t) \label{eq:ident A three ways}
 \end{equation}
 for any $t$. Given $\upbeta \in \Achow_1(\scrX)$ we abuse notation and let $\upbeta \in \Achow_1(\scrX_t)$ denote the image of $\upbeta$ under the composition of the natural isomorphisms above.  Further, combining \eqref{eq:ident A three ways} and \eqref{eq:ident A with comb} it makes sense to ask when curve classes are restricted roots.
 
 \begin{cor}\label{secret reason why nonzero}\label{lem: properties of Xt}\label{cor: curves on perturbation}
Fix $t \neq 0$ and $\upbeta \in \Achow_1(\scrX)$ non-zero. Then the following statements hold:
\begin{enumerate}
\item If $\upbeta$ is not a restricted positive root, i.e. there does not exist a positive root $\upalpha$ with $\uppi_\scrI(\upalpha)=\upbeta$, then there is no complete curve in $\scrX_t$ of class $\upbeta$.
\item If $\upbeta$ is a restricted positive root, with $\uppi_{\scrI}(\upalpha)=\upbeta$, then the number of complete curves in $\scrX_t$ of class $\upbeta$ is equal to $|\upmu_t^{-1}(\scrD_{\upalpha,\C}/W_\scrI)|$ and is always $\geq 1$.
\end{enumerate}
\end{cor}
\begin{proof} This follows from Theorem~\ref{thm: KM Theorem 1(c)}, the only new claim being that $|\upmu_t^{-1}(\scrD_{\upalpha,\C}/W_\scrI)| \geq 1$. We observed above that $\upmu(0)=0$. Since $\upmu=\upmu_0$ and $0 \in \scrD_{\upalpha,\C}/W_\scrI$ for every positive root $\upalpha$, it follows that $|\upmu_0^{-1}(\scrD_{\upalpha,\C}/W_\scrI)| \geq 1$ for every positive root $\upalpha$. This quantity does not change under a small perturbation of $\upmu_0$ to $\upmu_t$.\end{proof}

\section{Curve counting and hyperplane arrangements}\label{sec: enumerative invariants}

\noindent The previous section established strong control over generic perturbations of the target geometry $\scrX$. This section exploits this, and describes the qualitative structure of the systems of enumerative invariants attached to $\scrX$. We show that these are intimately related to the combinatorial hyperplane arrangements from Subsection~\ref{sec: hyperplane arrangements}, with the infinite affine arrangement in particular giving the pole locus of the so-called quantum potential.

\subsection{Gopakumar--Vafa}\label{GV section} \label{sec: GV invariants} Curve counting invariants of $\scrX$ will be defined using the perturbed target $\scrX_t$ (for some fixed $t \neq 0$) constructed in Subsection~\ref{sec: perturbing target}. Given a curve class $\upbeta \in \Achow_1(\scrX)$ the associated genus-zero Gopakumar--Vafa (GV) invariant
\[ 
n_\upbeta = n_{\upbeta,\scrX} \in \ZZ_{\geq 0}
\]
is defined as the number of complete curves in $\scrX_t$ of class $\upbeta$. By Corollary~\ref{lem: properties of Xt} this is zero if $\upbeta$ is not a restricted positive root, and otherwise is equal to the number of intersection points of $\upmu_t$ with the appropriate component of the discriminant locus, i.e.
\begin{equation} \label{eqn: nbeta equals intersection with hyperplane} n_{\upbeta} = |\upmu_t^{-1}(\scrD_{\upalpha,\C}/W_\scrI)| \end{equation}
where $\upalpha$ is any positive root with $\uppi_{\scrI}(\upalpha)=\upbeta$. This number is independent of the choice of small perturbation $\upmu_t$.

In what follows, for a curve class $\upbeta$ consider the dual hyperplane $\Hyp_\upbeta\subseteq \Uptheta_\scrI$.
\begin{corollary} \label{cor: GV nonzero}
If $\upbeta \in \Achow_1(\scrX)$ then $n_\upbeta$ is non-zero if and only if $\upbeta$ is a restricted positive root, equivalently if and only if $\Hyp_\upbeta$ belongs to the enhanced finite arrangement $\scrH_\scrI$.
\end{corollary}
\begin{proof} 
This follows immediately from Corollary~\ref{cor: curves on perturbation}, together with the definition of enhanced finite arrangement in Subsection~\ref{sec: hyperplane arrangements}.
\end{proof}
Note that $\Hyp_\upbeta$ and $\Hyp_{2\upbeta}$ should be considered as different hyperplanes in the enhanced finite arrangement. See also the discussion in Subsection~\ref{sec: hyperplane arrangements}, and Remark~\ref{rmk: proportional but distinct}.

\begin{remark} It follows from Corollary~\ref{cor: GV nonzero} that there are only finitely many non-zero GV invariants. There is already a known range outside of which the GV invariants are guaranteed to vanish. Indeed, every simple root $i \in \Delta$ has an associated length $\updelta_i$, given by the coefficient of $\upalpha_i$ in the maximal root, and writing $\upbeta = \Sigma_i m_i \Curve_i$ it is known that $n_{\upbeta} = 0$ unless $m_i \leq \updelta_i$ for all $i$. However, this bound is far from sharp, while Corollary~\ref{cor: GV nonzero} provides a precise characterisation.
\end{remark}

\subsection{Gromov--Witten} \label{sec: GW invariants}  We refer to \cite[Section~7]{CoxKatz} for an introduction to Gromov--Witten theory. For every non-zero curve class $\upbeta \in \Achow_1(\scrX)$ there is an associated genus-zero Gromov--Witten (GW) invariant
\[
 N_\upbeta=N_{\upbeta,\scrX} \in \QQ
 \]
defined as the virtual degree of the corresponding moduli space of stable maps to $\scrX$. By deformation invariance this coincides with the virtual degree of the moduli space of stable maps to $\scrX_t$ for $t \neq 0$, as constructed in Subsection~\ref{sec: deforming to Xt}.  

The latter space decomposes as a disjoint union of spaces of stable maps to $\PP^1$, and applying the Aspinwall--Morrison multiple cover formula \cite{AspinwallMorrison,VoisinAspinwallMorrison} for the local invariants of $\scrO_{\PP^1}(-1)\oplus \scrO_{\PP^1}(-1)$ gives the following relationship between the GW invariants and the GV invariants from Subsection~\ref{sec: GV invariants}, namely 
\begin{equation} 
N_\upbeta = \sum_{d | \upbeta} \dfrac{n_{\upbeta/d}}{d^3}\label{eqn: multiple cover formula}.
\end{equation}
More generally,  given $k \geq 0$ and homogeneous classes $\upgamma_1,\ldots,\upgamma_k \in \mathrm{H}^\star(\scrX;\C)$, the associated GW invariant with cohomological insertions at marked points is defined to be
\[
\langle \upgamma_1,\ldots,\upgamma_k \rangle^{\scrX}_{0,k,\upbeta} \colonequals  \int_{[\ol{\mathcal{M}}_{0,k}(\scrX,\upbeta)]^{\operatorname{virt}}} \prod_{i=1}^k \ev_i^\star \upgamma_i 
\]
and provides a virtual count of rational curves in $\scrX$ of class $\upbeta$ passing through the cycles $\upgamma_1,\ldots,\upgamma_k$. Note that in particular $N_{\upbeta} = \langle \rangle^{\scrX}_{0,0,\upbeta}$. Since $\scrX$ is a Calabi--Yau $3$-fold, the invariant vanishes unless the input data satisfies the dimension constraint
\begin{equation}\label{eqn: dimension constraint} \sum_{i=1}^k \deg \upgamma_i = 2k.\end{equation}
The cohomology of $\scrX$ is well-understood, see e.g.\ \cite[5.2]{CaibarThesisPaper}. In particular
\begin{equation*} 
\mathrm{H}^0(\scrX; \C) = \C \cdot \mathds{1},\quad \mathrm{H}^1(\scrX; \C)=0,\quad \mathrm{H}^2(\scrX; \C) = \Pic \scrX \otimes \C.
\end{equation*}
Moreover, as we work in the complete local setting, by e.g.\ \cite[3.4.4]{VdB1d} $\Pic\scrX$ is dual to the group $\Achow_1(\scrX)$ of curve classes, since there is a basis of divisor classes $\Pic\scrX = \langle D_i \mid i \in \scrIc \rangle_{\Z}$ which satisfies $D_i \cdot \Curve_j = \updelta_{ij}$.

Given a GW invariant $\langle \upgamma_1,\ldots,\upgamma_k \rangle^{\scrX}_{0,k,\upbeta}$, if any $\upgamma_i=\mathds{1}$ then the invariant vanishes by the string equation. It follows from \eqref{eqn: dimension constraint} that the invariant vanishes unless each $\upgamma_i \in \mathrm{H}^2(\scrX;\C)$. But then the $\upgamma_i$ are divisors, and the $k$-pointed invariants with divisorial insertions are related to the $0$-pointed invariants by the divisor equation
\[ 
\langle D_{j_1}, \ldots, D_{j_k} \rangle^{\scrX}_{0,k,\upbeta} = \left( \prod_{i=1}^k D_{j_i}\cdot \upbeta \right)  N_{\upbeta}.
\]
In this way, the non-zero GW invariants are controlled entirely by the $N_{\upbeta}$, which by \eqref{eqn: multiple cover formula} are controlled entirely by the GV invariants $n_{\upbeta}$. The latter constitutes a finite list of numbers.

\subsection{Quantum cohomology} \label{sec: quantum potential} As is well known, the GW invariants form the structure constants for quantum cohomology. The information defining quantum cohomology is equivalent to the \emph{quantum potential}, defined in our setting as
\begin{equation} \label{eqn: quantum potential} 
\Phi^{\scrX}_{\mathsf{t}}( \upgamma_1,\upgamma_2,\upgamma_3) \colonequals \sum_{\substack{\upbeta \in \Achow_1(\scrX)\\ \upbeta \neq 0}} \,\sum_{k \geq 0} \,\,\dfrac{1}{k!} \langle \upgamma_1,\upgamma_2,\upgamma_3, \mathsf{t} , \ldots, \mathsf{t} \rangle^{\scrX}_{0,k+3,\upbeta}. \end{equation}
Here we exclude the case $\upbeta = 0$ from the sum, since for non-compact $\scrX$ such invariants are not defined (see Remark~\ref{rmk: no algebra} below). We view \eqref{eqn: quantum potential}  as a family of multilinear maps
\[
\Phi^{\scrX}_\mathsf{t} \colon \mathrm{H}^\star(\scrX;\C)^{\otimes 3} \to \C	
\]
parametrised by the formal variable $\mathsf{t} \in \mathrm{H}^\star(\scrX;\C)$. By the earlier dimension arguments, we see that the quantum potential only depends on the component of $\mathsf{t}$ in the $\mathrm{H}^2(\scrX;\C)$ direction. Thus we may assume $\mathsf{t} \in \mathrm{H}^2(\scrX;\C)$ and write
\[
\mathsf{t} = (\mathsf{t}_i)_{i \in \scrIc} = \sum_{i \in \scrIc} \mathsf{t}_i D_i.
\]
The parameter space for the quantum potential is thus co-ordinatised by the $\mathsf{t}_i$. An alternative co-ordinate system is given by the Novikov parameters, defined by
\[ 
\mathsf{q}_i \colonequals \exp(\mathsf{t}_i).
\]
The following result resembles expressions appearing in earlier work \cite{MorrisonKaehler, Wilson}, but is more explicit, being given in terms of canonical bases for $\mathrm{H}^2(\scrX;\C)$ and $\Achow_1(\scrX)$.  This refined information will allow us in Corollary~\ref{cor: quantum is hyper} to pinpoint the non-vanishing terms using Dynkin combinatorics, and in Corollary~\ref{thm: CTC} to track the change in quantum potential under iterated flops.

\begin{theorem} \label{thm: structure of quantum potential} The quantum potential has a natural analytic continuation over the parameter space, given as a finite sum of terms indexed by the non-vanishing GV invariants
\begin{equation} \label{eqn: rational expression for quantum potential} \Phi^{\scrX}_\mathsf{t}(\upgamma_1,\upgamma_2,\upgamma_3) = \sum_{\upbeta=(m_i)} n_\upbeta (\upgamma_1 \cdot \upbeta) (\upgamma_2 \cdot \upbeta) (\upgamma_3 \cdot \upbeta) \dfrac{\prod_{i \in \scrIc} \mathsf{q}_i^{m_i}}{1-\prod_{i \in \scrIc} \mathsf{q}_i^{m_i}}. \end{equation}
\end{theorem}
\noindent The sum is over non-zero curve classes
\[\upbeta=(m_i)_{i \in \scrIc} = \sum m_i \Curve_i \in \Achow_1(\scrX).\]
Each term is a cubic polynomial in the input variables, multiplied by a specific rational function in the Novikov parameters, and weighted by the GV invariant $n_{\upbeta}$. We will also use the term `quantum potential' to refer to this analytic continuation.

\begin{proof} Write the formal parameter $\mathsf{t}$ and the curve class $\upbeta$ as sums
\[
\mathsf{t} = \sum_{i \in \scrIc} \mathsf{t}_i D_i, \qquad
\upbeta = \sum_{i \in \scrIc} m_i \Curve_i.
\]
Applying the divisor equation together with the multiple cover formula \eqref{eqn: multiple cover formula} then gives
\begin{align*}
\Phi^{\scrX}_{\mathsf{t}}(\upgamma_1,\upgamma_2,\upgamma_3) & = \sum_{\upbeta} \,\sum_{k \geq 0} \,\,\dfrac{1}{k!} \langle \upgamma_1,\upgamma_2,\upgamma_3, \mathsf{t} , \ldots, \mathsf{t} \rangle^{\scrX}_{0,k+3,\upbeta}\\
& = \sum_{\upbeta} \langle \upgamma_1,\upgamma_2,\upgamma_3 \rangle^{\scrX}_{0,3,\upbeta} \,\big(\textstyle\sum_{k\geq 0} \tfrac{(\mathsf{t} \cdot \upbeta)^k}{k!} \big)\\
& = \sum_{\upbeta} N_\upbeta  (\upgamma_1 \cdot \upbeta) (\upgamma_2 \cdot \upbeta) (\upgamma_3 \cdot \upbeta) \exp\!\big( \textstyle\sum_{i \in \scrIc} m_i \mathsf{t}_i\big) \\
& = \sum_{\upbeta} n_{\upbeta} \sum_{d \geq 1} \dfrac{1}{d^3} (\upgamma_1 \cdot d\upbeta)(\upgamma_2 \cdot d\upbeta)(\upgamma_3 \cdot d\upbeta)  \exp\!\big( \textstyle\sum_{i \in \scrIc} dm_i \mathsf{t}_i\big) \\
& = \sum_{\upbeta} n_\upbeta  (\upgamma_1 \cdot \upbeta) (\upgamma_2 \cdot \upbeta) (\upgamma_3 \cdot \upbeta) \sum_{d \geq 1} \exp\!\big( \textstyle\sum_{i \in \scrIc} m_i \mathsf{t}_i\big)^d \\
& = \sum_{\upbeta} n_\upbeta (\upgamma_1 \cdot \upbeta) (\upgamma_2 \cdot \upbeta) (\upgamma_3 \cdot \upbeta) \sum_{d \geq 1} \big( \Pi_{i \in \scrIc} \mathsf{q}_i^{m_i}\big)^d.
\end{align*}
Note that this sum is finite, since $n_\upbeta=0$ for all but finitely many $\upbeta$ (Corollary~\ref{cor: GV nonzero}). For fixed inputs $(\upgamma_1,\upgamma_2,\upgamma_3)$, the above is a formal power series in the Novikov parameters. It is the Taylor series for the following rational function, expanded about the point $(\mathsf{q}_i)_i = (0,\ldots,0)$, equivalently $(\mathsf{t}_i)_i = (-\infty, \ldots, -\infty)$, 
\begin{align*} 
\Phi^{\scrX}_\mathsf{t}(\upgamma_1,\upgamma_2,\upgamma_3) = \sum_\upbeta n_\upbeta (\upgamma_1 \cdot \upbeta) (\upgamma_2 \cdot \upbeta) (\upgamma_3 \cdot \upbeta) \dfrac{\prod_{i \in \scrIc} \mathsf{q}_i^{m_i}}{1-\prod_{i \in \scrIc} \mathsf{q}_i^{m_i}}.
\end{align*}
This expression provides a natural analytic continuation of the quantum potential beyond the radius of convergence $\{ |\mathsf{q}_i| < 1 \mid i \in \scrIc \}$.
\end{proof}

Recall that after combining \eqref{eq:ident A three ways} with \eqref{eq:ident A with comb} we can ask which curve classes are restricted roots.  
\begin{corollary}\label{cor: quantum is hyper}
Under the uniformly rescaled co-ordinates $\mathsf{p}_i \colonequals \mathsf{t}_i/2\uppi\sqrt{-1}$ on $\mathrm{H}^2(\scrX;\C)$, the pole locus of the quantum potential is given by
\[
\bigcup_{\upbeta=(m_i)}\big\{\textstyle\sum_{i \in \scrIc} m_i \mathsf{p}_i \in \Z \big\}
\]
where the union is over all restricted positive roots $\upbeta$. This is precisely the complexification of $\scrH^{\aff}_{\scrI}$ under the natural identification $\mathrm{H}^2(\scrX;\R)\cong\Uptheta_{\scrI}$ dual to \eqref{eq:ident A with comb}.
\end{corollary}
\begin{proof} 
The pole locus of \eqref{eqn: rational expression for quantum potential} is the union of loci in the parameter space given by
\[ 
\prod_{i \in \scrIc} \mathsf{q}_i^{m_i} = 1 \
\Leftrightarrow \ \sum_{i \in \scrIc} m_i \mathsf{t}_i \in 2 \uppi \sqrt{-1} \cdot \Z \
\Leftrightarrow \ \sum_{i \in \scrIc} m_i \mathsf{p}_i \in \Z
\]
where $\upbeta=(m_i)_{i \in \scrIc}=\sum m_i \Curve_i$ is such that $n_\upbeta$ is non-zero. The first statement then follows from Corollary~\ref{cor: GV nonzero}, since $n_\upbeta$ is non-zero if and only if $\upbeta$ is a restricted positive root.  The second is an immediate consequence of the definition of $\scrH^{\aff}_{\scrI}$ in \eqref{eqn: infinite arrangement}.
\end{proof}

\begin{example}
For a single-curve flop with $\scrI=\Eseven{B}{B}{P}{B}{B}{B}{B}$, the complexification of  $\scrH^{\aff}_{\scrI}$ is 
\def\HypScalex{4}
\[
\begin{tikzpicture}[scale=0.9]
\draw[densely dotted,->] ($(-\HypScalex,0)+(-0.66,0)$) -- ($(\HypScalex,0)+(0.66,0)$);
\node at ($(\HypScalex,0)+(1,0)$) {$\scriptstyle\mathbb{R}$};
\draw[densely dotted,->] (0,-1) -- (0,1);
\node at (0,1.2) {$\scriptstyle\mathrm{i}\mathbb{R}$};

{\foreach \i in {-2,-1,0,1,2}
\filldraw[fill=white,draw=black] (\HypScalex*1/2*\i,0) circle (2pt);
}
{\foreach \i in {-2,-1,1,2}
\filldraw[fill=white,draw=black] (\HypScalex*1/3*\i,0) circle (2pt);
}
{\foreach \i in {-3,-2,-1,1,2,3}
\filldraw[fill=white,draw=black] (\HypScalex*1/4*\i,0) circle (2pt);
}

\node at (-\HypScalex,-0.25) {$\scriptstyle -1$};

\node at (0,-0.25) {$\scriptstyle 0$};

\node at (\HypScalex*0.25,-0.3) {$\scriptstyle \frac{1}{4}$};
\node at (\HypScalex*1/3,-0.3) {$\scriptstyle \frac{1}{3}$};
\node at (\HypScalex*1/2,-0.3) {$\scriptstyle \frac{1}{2}$};
\node at (\HypScalex*2/3,-0.3) {$\scriptstyle \frac{2}{3}$};
\node at (\HypScalex*3/4,-0.3) {$\scriptstyle \frac{3}{4}$};
\node at (\HypScalex,-0.25) {$\scriptstyle 1$};
\end{tikzpicture}
\]
extended to infinity in both directions. The non-zero GV invariants are $n_{k\Curve}$ for $1 \leq k \leq 4$.
\end{example}

\begin{example}
In the running Example~\ref{example: intro}, namely a two-curve flop with $\scrI=\Eeight{B}{P}{B}{B}{B}{B}{B}{P}$, the complexification of  $\scrH^{\aff}_{\scrI}$ is 
the complexification of the real arrangement in \eqref{running example infinite}.
\end{example}

\begin{remark} \label{rmk: no algebra} The definition of the quantum cohomology algebra requires a perfect pairing on cohomology in order to raise indices, but since here $\scrX$ is non-compact, such a pairing does not exist. This technical issue is often circumvented by localising to a torus-fixed locus which is compact, see e.g.\ \cite{BryanGholampour,CoatesIritaniJiang}.  Since our geometries do not always carry a suitable torus action, instead we simply equate `quantum cohomology' with the data of the quantum potential \eqref{eqn: quantum potential}, as this is consistent with other approaches \cite{LiRuan}. In cases where a natural quantum cohomology algebra can be defined, our results apply equally well to that algebra. The only modification required is to reinstate the $\upbeta=0$ terms in the quantum potential, which encode the given perfect pairing. \end{remark}

\section{Flops via simultaneous partial resolutions}\label{sec: flops via sim res}

\noindent This section constructs flops, and describes how their dual graph changes, via simultaneous partial resolutions, completing work of Pinkham \cite{Pinkham}. As a consequence we obtain an explicit change-of-basis matrix, which in Subsection~\ref{sec: GV under flop} is used to track the change of GV invariants under iterated flop.

The construction requires some more Dynkin notation, so for a subset $\scrI\subseteq\Delta$, $j\in \scrI$ and $i\in \scrIc$ write
\[
\scrI+i= \scrI\cup \{i\}\quad\mbox{and}\quad
\scrI-j= \scrI\setminus \{j\}.
\]
Further, to every Dynkin diagram $\Gamma$ is an associated \emph{Dynkin involution}, which we will denote $\upiota_\Gamma$.  For Type $A_n$ and $E_6$ this is the obvious reflection, for $E_7$ and $E_8$ it is trivial, and for $D_n$ the behaviour depends on the parity of $n$, see e.g.\ \cite[(1.2.B)]{IyamaWemyssTits}.  If $\Gamma$ is a disjoint union of Dynkin diagrams, then $\upiota_\Gamma$ by definition acts separately on each component.  Further, if $\Delta$ is ADE, and $\Gamma$ is a subset of $\Delta$, then automatically the subgraph $\Gamma$ is a disjoint union of ADE diagrams, and so there is an associated $\upiota_\Gamma$.

\begin{notation}[\hspace{1sp}{\cite[1.16]{IyamaWemyssTits}}]
For $i\in\scrIc$, the wall crossing $\upomega_i(\scrI)$ is defined by the rule
\[
\upomega_i(\scrI) \colonequals \scrI+i-\upiota_{\scrI+i}(i) \subseteq \Delta.
\]
\end{notation}

\begin{example}\label{ex: changing dual graph}
Consider the running Example~\ref{example: intro}, namely $\scrI=\Eeight{B}{P}{B}{B}{B}{B}{B}{P}$, where by convention $\scrI$ equals the six black dots.  There are two choices for $i\in \scrIc$, namely the two pink nodes.  Let $i$ be the rightmost.  Then $\scrI+i$ equals the black dots in the following.
\[
\scrI+i=
\begin{array}{c}
\begin{tikzpicture}[scale=0.21]
\node at (0,0) [B] {};
\node at (1,0) [P] {};
\node at (2,0) [B] {};
\node at (2,1) [B] {};
\node at (3,0) [B] {};
\node at (4,0) [B] {};
\node at (5,0) [B] {};
\node at (6,0) [B] {};
%\node at (7,0) [#9] {};
\draw[densely dotted] (1.5,1.5)--(2.5,1.5) -- (2.5,0.5)--(6.5,0.5)--(6.5,-0.5)--(1.5,-0.5)--cycle;
\draw[densely dotted] (-0.5,0.5)--(0.5,0.5) --(0.5,-0.5)--(-0.5,-0.5)--cycle;
\draw[<->,bend right] (6,0.5) to (2.5,1.5);
\end{tikzpicture}
\end{array}
\]
The black dots form $A_1\times A_6$, and so applying the Dynkin involution $\upiota_{\scrI+i}$ illustrated, we see that $\upiota_{\scrI+i}(i)$ is the top node.  Thus, for this choice of $i$, 
\[
\upomega_i(\scrI)=\scrI+i-\upiota_{\scrI+i}(i)=\Eeight{B}{P}{B}{P}{B}{B}{B}{B}.
\]
\end{example}

Consider now the fixed flopping contraction $\scrX\to\Spec\scrR$ which slices under \eqref{elephant pullback} to give $Y_{\scrI} \to \C^2/G$.  Pick a flopping curve $\Curve_i$ in $\scrX$. This corresponds to a choice of $i\in\scrIc$, so we can form $\upomega_i(\scrI)$. In what follows consider $\ell_\scrI\ell_{\scrI+i}$, where $\ell_\scrI$ and $\ell_{\scrI+i}$ are the longest elements in the parabolic subgroups $W_\scrI$ and $W_{\scrI+i}$ respectively.  

\begin{lemma}\label{lem: induced iso A}
The left action by $\ell_\scrI\ell_{\scrI+i}$ induces an isomorphism $\mathfrak{h}_{\mathbb{C}}/W_{\upomega_i(\scrI)}\to \mathfrak{h}_{\mathbb{C}}/W_{\scrI}$.
\end{lemma}
\begin{proof}
Writing $w=\ell_\scrI\ell_{\scrI+i}$, the point is that $wW_{\upomega_i(\scrI)}=W_{\scrI}w$ (see e.g.\ \cite[1.20(1)(a)]{IyamaWemyssTits}).  The action by $w$ is an isomorphism $\mathfrak{h}_\C\to\mathfrak{h}_\C$, and under this isomorphism any orbit $(W_{\upomega_i(\scrI)}) p$ gets sent to $w (W_{\upomega_i(\scrI)} )p=(W_{\scrI}) w p$, which is an orbit under $W_{\scrI}$.
\end{proof}
In what follows, we will always consider left actions. When labelling arrows in commutative diagrams, we will often write e.g.\ $\ell_\scrI\ell_{\scrI+i}\cdot$ to denote the morphism given by left multiplication by $\ell_\scrI\ell_{\scrI+i}$.

Note that $\ell_{\upomega_i(\scrI)}\ell_{\scrI+i}\ell_\scrI\ell_{\scrI+i}=1$ \cite[1.2(3)]{IyamaWemyssTits}, and thus $\ell_{\upomega_i(\scrI)}\ell_{\scrI+i}\colon \mathfrak{h}_{\mathbb{C}}/W_{\scrI}\to \mathfrak{h}_{\mathbb{C}}/W_{\upomega_i(\scrI)}$ is the inverse map. Recall the notation in \eqref{preferred sim partial res}. Using the universal property of the pullback gives a non-obvious isomorphism between  $\scrV_{\scrI}$ and $\scrV_{\upomega_i(\scrI)}$, which sits in the following commutative diagram.
\def\myx{-1}
\def\myy{-1}
\[
\begin{tikzpicture}[>=stealth,scale=1.2]
%\node[black!70!white] (X) at (0,0) {$\scrX^+$};
%\node[black!70!white] (Z) at (2,0) {$\scrY_{\upomega_i(\scrI)}$};
%\node[black!70!white] (R) at (0,-1.5) {$\Spec \scrR'$};
\node[black!70!white] (U) at (2,-1.5) {$\Spec \scrV_{\upomega_i(\scrI)}$};
\node[black!70!white] (V) at (4,-1.5) {$\Spec \scrV$};
%\node[black!70!white] (t) at (0,-3) {$\formaldisc$};
\node[black!70!white] (PRes) at (2,-3) {$\mathfrak{h}_{\mathbb{C}}/W_{\upomega_i(\scrI)}$};
\node[black!70!white] (Res) at (4,-3) {$\mathfrak{h}_{\mathbb{C}}/W$};
%\draw[black!70!white,->] (X)--(Z);
%\draw[black!70!white,->] (X)--node[gap,pos=0.75]{\phantom .}(R);
%\draw[black!70!white,->] (Z)--(U);
%\draw[black!70!white,->] (R)--(U);
\draw[black!70!white,->] (U)--(V);
\draw[black!70!white,->] (V)--(Res);
%\draw[black!70!white,->] (R)--node[gap,pos=0.72]{\phantom .}(t);
\draw[black!70!white,->] (U)--node[gap,pos=0.76]{\phantom .}(PRes);
%\draw[black!70!white,->] (t)--node[gap,pos=0.735]{\phantom .}(PRes);
\draw[black!70!white,->] (PRes)--node[gap,pos=0.31]{\phantom .}(Res);
%%%%%%%%%%%%%%%%%%%%
%\node (fX) at (\myx,\myy) {$\scrX$};
%\node (fZ) at ($(2,0)+(\myx,\myy)$) {$\scrY_{\scrI}$};
%\node (fR) at ($(0,-1.5)+(\myx,\myy)$) {$\Spec \scrR$};
\node (fU) at ($(2,-1.5)+(\myx,\myy)$) {$\Spec \scrV_{\scrI}$};
\node (fV) at ($(4,-1.5)+(\myx,\myy)$) {$\Spec \scrV$};
%\node (ft) at ($(0,-3)+(\myx,\myy)$) {$\formaldisc$};
\node (fPRes) at ($(2,-3)+(\myx,\myy)$) {$\mathfrak{h}_{\mathbb{C}}/W_{\scrI}$};
\node (fRes) at ($(4,-3)+(\myx,\myy)$) {$\mathfrak{h}_{\mathbb{C}}/W$};
%\draw[->] (fX)--(fZ);
%\draw[->] (fX)--(fR);
%\draw[->] (fZ)--(fU);
%\draw[->] (fR)--(fU);
\draw[->] (fU)--(fV);
\draw[->] (fV)--(fRes);
%\draw[->] (fR)--(ft);
\draw[->] (fU)--(fPRes);
%\draw[->] (ft)--(fPRes);
\draw[->] (fPRes)--(fRes);
%%%%%%%%%%%%%%%%%%%
\draw[double] (fRes)--(Res);
\draw[double] (fV)--(V);
\draw[<-] (fPRes)--node[pos=0.35,right]{$\scriptstyle \ell_{\scrI}\ell_{\scrI+i}\cdot$}(PRes);
\draw[<-] (fU)--node[right]{$\scriptstyle \sim$}(U);
%\draw[double] (ft)--(t);
%\draw[<-] (fR)--node[right]{$\scriptstyle \sim$}(R);
\end{tikzpicture}
\]

Now let $\scrY_\scrI$, respectively $\scrY_{\upomega_i(\scrI)}$, be the standard simultaneous resolution associated to $\scrI$, respectively $\upomega_i(\scrI)$. As explained in Subsection~\ref{sec: deforming to Xt}, $\scrX$ can be obtained from $\upmu\colon\formaldisc\to\mathfrak{h}_{\mathbb{C}}/W_{\scrI}$ by pulling back $\scrY_\scrI$. Hence, setting $\upnu=( \ell_{\upomega_i(\scrI)}\ell_{\scrI+i} )\circ\upmu=(\ell_\scrI\ell_{\scrI+i})^{-1}\circ\upmu$ and pulling back to $\scrY_{\upomega_i(\scrI)}$ constructs a variety $\scrX^+_i$, sitting within the following commutative diagram.
\begin{equation}
\begin{array}{c}
\begin{tikzpicture}[>=stealth,scale=1.2]
\node[black!70!white] (X) at (0,0) {$\scrX_i^+$};
\node[black!70!white] (Z) at (2,0) {$\scrY_{\upomega_i(\scrI)}$};
\node[black!70!white] (R) at (0,-1.5) {$\Spec \scrR'$};
\node[black!70!white] (U) at (2,-1.5) {$\Spec \scrV_{\upomega_i(\scrI)}$};
\node[black!70!white] (V) at (4,-1.5) {$\Spec \scrV$};
\node[black!70!white] (t) at (0,-3) {$\formaldisc$};
\node[black!70!white] (PRes) at (2,-3) {$\mathfrak{h}_{\mathbb{C}}/W_{\upomega_i(\scrI)}$};
\node[black!70!white] (Res) at (4,-3) {$\mathfrak{h}_{\mathbb{C}}/W$};
\draw[black!70!white,->] (X)--(Z);
\draw[black!70!white,->] (X)--node[gap,pos=0.75]{\phantom .}(R);
\draw[black!70!white,->] (Z)--(U);
\draw[black!70!white,->] (R)--(U);
\draw[black!70!white,->] (U)--(V);
\draw[black!70!white,->] (V)--(Res);
\draw[black!70!white,->] (R)--node[gap,pos=0.72]{\phantom .}(t);
\draw[black!70!white,->] (U)--node[gap,pos=0.76]{\phantom .}(PRes);
\draw[black!70!white,->] (t)--node[gap,pos=0.735]{\phantom .}node[gap,pos=0.4]{$\scriptstyle\upnu$}(PRes);
\draw[black!70!white,->] (PRes)--node[gap,pos=0.31]{\phantom .}(Res);
%%%%%%%%%%%%%%%%%%%%
\node (fX) at (\myx,\myy) {$\scrX$};
\node (fZ) at ($(2,0)+(\myx,\myy)$) {$\scrY_{\scrI}$};
\node (fR) at ($(0,-1.5)+(\myx,\myy)$) {$\Spec \scrR$};
\node (fU) at ($(2,-1.5)+(\myx,\myy)$) {$\Spec \scrV_{\scrI}$};
\node (fV) at ($(4,-1.5)+(\myx,\myy)$) {$\Spec \scrV$};
\node (ft) at ($(0,-3)+(\myx,\myy)$) {$\formaldisc$};
\node (fPRes) at ($(2,-3)+(\myx,\myy)$) {$\mathfrak{h}_{\mathbb{C}}/W_{\scrI}$};
\node (fRes) at ($(4,-3)+(\myx,\myy)$) {$\mathfrak{h}_{\mathbb{C}}/W$};
\draw[->] (fX)--(fZ);
\draw[->] (fX)--(fR);
\draw[->] (fZ)--(fU);
\draw[->] (fR)--(fU);
\draw[->] (fU)--(fV);
\draw[->] (fV)--(fRes);
\draw[->] (fR)--(ft);
\draw[->] (fU)--(fPRes);
\draw[->] (ft)--node[above]{$\scriptstyle\upmu$}(fPRes);
\draw[->] (fPRes)--(fRes);
%%%%%%%%%%%%%%%%%%%
\draw[double] (fRes)--(Res);
\draw[double] (fV)--(V);
\draw[<-] (fPRes)--node[pos=0.35,right]{$\scriptstyle \ell_{\scrI}\ell_{\scrI+i}\cdot$}(PRes);
\draw[<-] (fU)--node[right]{$\scriptstyle \sim$}(U);
\draw[double] (ft)--(t);
\draw[<-] (fR)--node[right]{$\scriptstyle \sim$}(R);
\end{tikzpicture}\label{create flop comm diagram}
\end{array}
\end{equation}
Composing the map $\scrX^+_i\to\Spec\scrR'$ with the isomorphism $\Spec\scrR'\to\Spec\scrR$ yielding a morphism $\scrX^+_i\to\Spec \scrR$.

\begin{thm}\label{thm: produce flop} \label{thm: flop by sim res}
With notation as above, $\scrX_i^+\to\Spec\scrR$ is the flop of $\scrX$ at the curve $\Curve_i$.  In particular, the following statements hold.
\begin{enumerate}
\item $\upomega_i(\scrI) \subseteq \Delta$ is the Dynkin data associated to the flopping contraction $\scrX_i^+\to\Spec\scrR$.
\item All other crepant resolutions of $\Spec\scrR$ can be obtained from the fixed $\upmu$ by post-composing with $x^{-1}\colon\mathfrak{h}_{\mathbb{C}}/W_{\scrI}\to\mathfrak{h}_{\mathbb{C}}/W_{\scrK}$ and pulling back along $\scrY_\scrK$, as the pair $(x,\scrK)$ ranges over the (finite) indexing set  $\mathsf{Cham}(\Delta,\scrI)$ of Notation~\ref{notation: subset stuff}.
\end{enumerate}
\end{thm}
\begin{proof}
 Since the exceptional locus of $\mathsf{g}_\scrI\colon\scrY_\scrI\to\Spec\scrV_\scrI$ has codimension two, $\Cl(\scrV_{\scrI})\cong\Cl(\scrY_\scrI)$. But $\scrY_\scrI$ is smooth, so the latter is isomorphic to $\Pic(\scrY_\scrI)$, which in turn is isomorphic to $\mathbb{Z}^{|\scrIc|}$ based by divisors dual to the $|\scrIc|$ curves above the origin. Choosing this basis, we may write $\Cl(\scrV_{\scrI})\cong\bigoplus_{j\in\scrIc} \mathbb{Z}e_j^\star$, and further by definition $\scrY_\scrI$ is obtained by blowing up any element in the characteristic cone
\[
C_{\scrI}\colonequals\{ (z_j)\mid z_j>0\mbox{ for all }j\}\subseteq \,\,\bigoplus_{j\in\scrIc} \mathbb{Z}e_j^\star.
\] 
Applying a similar analysis to the standard simultaneous resolution $\scrZ\to\Spec\scrW$, we have $\Cl(\scrW)\cong\bigoplus_{j\in\Delta} \mathbb{Z}e_j^\star$, and since by construction $Y_{\scrI}$ is obtained from $\scrZ$ by simultaneously blowing down curves, it is clear that
\begin{equation}
\begin{array}{c}
\begin{tikzpicture}
\node (A) at (0,0) {$\Cl(\scrV_\scrI)$};
\node (B) at (2.5,0) {$\Cl(\scrW)$};
\node (a) at (0,-1.25) {$\bigoplus_{j\in\scrIc} \mathbb{Z}e_j^\star$};
\node (b) at (2.5,-1.25) {$\bigoplus_{j\in\Delta} \mathbb{Z}e_j^\star$};
\draw[->] (A)--node[left]{$\scriptstyle \sim$}(a);
\draw[->] (B)--node[left]{$\scriptstyle \sim$}(b);
\draw[right hook->] (A)--(B);
\draw[right hook->] (a)--(b);
\end{tikzpicture}
\end{array}\label{class group inclusions}
\end{equation}
where the bottom morphism is the obvious inclusion induced by the inclusion $\scrIc \subseteq \Delta$. There is a natural identification of $\Cl(\scrV_\scrI)$, respectively $\Cl(\scrW_\scrI)$, with the lattice inside $\Theta_\scrI$, respectively $\Theta$.

By the universal property of the pullback, the action of any $s_j\in W$ on $\mathfrak{h}_{\C}$ induces an automorphism of $\scrW$
\[
\begin{tikzpicture}[>=stealth,scale=1.2];
\node[black!70!white] (U) at (2,-1.5) {$\Spec \scrW$};
\node[black!70!white] (V) at (4,-1.5) {$\Spec \scrV$};
\node[black!70!white] (PRes) at (2,-3) {$\mathfrak{h}_{\mathbb{C}}$};
\node[black!70!white] (Res) at (4,-3) {$\mathfrak{h}_{\mathbb{C}}/W$};
\draw[black!70!white,->] (U)--(V);
\draw[black!70!white,->] (V)--(Res);
\draw[black!70!white,->] (U)--node[gap,pos=0.74]{\phantom .}(PRes);
\draw[black!70!white,->] (PRes)--node[gap,pos=0.6]{\phantom .}(Res);
%%%%%%%%%%%%%%%%%%%%
\node (fU) at ($(2,-1.5)+(\myx,\myy)$) {$\Spec \scrW$};
\node (fV) at ($(4,-1.5)+(\myx,\myy)$) {$\Spec \scrV$};
\node (fPRes) at ($(2,-3)+(\myx,\myy)$) {$\mathfrak{h}_{\mathbb{C}}$};
\node (fRes) at ($(4,-3)+(\myx,\myy)$) {$\mathfrak{h}_{\mathbb{C}}/W$};
\draw[->] (fU)--(fV);
\draw[->] (fV)--(fRes);
%\draw[->] (fR)--(ft);
\draw[->] (fU)--(fPRes);
%\draw[->] (ft)--(fPRes);
\draw[->] (fPRes)--(fRes);
%%%%%%%%%%%%%%%%%%%
\draw[double] (fRes)--(Res);
\draw[double] (fV)--(V);
\draw[<-] (fPRes)--node[pos=0.4,right]{$\scriptstyle s_j\cdot$}(PRes);
\draw[<-] (fU)--node[right]{$\scriptstyle \sim$}(U);
%\draw[double] (ft)--(t);
%\draw[<-] (fR)--node[right]{$\scriptstyle \sim$}(R);
\end{tikzpicture}
\]
The effect on homology of $\scrZ$ is via the Weyl reflection $s_j$ on roots \cite{Reid}, and thus the action on $\Cl(\scrW)$ is the dual, namely the action of the Weyl reflection $s_j$ on coroots.   

Composing these $s_j$, we can consider the action of $\ell_\scrI\ell_{\scrI+i}$.  We already know that the bottom right square in \eqref{create flop comm diagram} commutes, and hence the bottom two squares in the following are well defined, and commute.
\[
\begin{array}{c}
\begin{tikzpicture}[>=stealth,scale=1.2]
\node[black!70!white] (R) at (0,-1.5) {$\Spec \scrW$};
\node[black!70!white] (U) at (2,-1.5) {$\Spec \scrV_{\upomega_i(\scrI)}$};
\node[black!70!white] (V) at (4,-1.5) {$\Spec \scrV$};
\node[black!70!white] (t) at (0,-3) {$\mathfrak{h}_{\mathbb{C}}$};
\node[black!70!white] (PRes) at (2,-3) {$\mathfrak{h}_{\mathbb{C}}/W_{\upomega_i(\scrI)}$};
\node[black!70!white] (Res) at (4,-3) {$\mathfrak{h}_{\mathbb{C}}/W$};
\draw[black!70!white,->] (R)--(U);
\draw[black!70!white,->] (U)--(V);
\draw[black!70!white,->] (V)--(Res);
\draw[black!70!white,->] (R)--node[gap,pos=0.74]{\phantom .}(t);
\draw[black!70!white,->] (U)--node[gap,pos=0.76]{\phantom .}(PRes);
\draw[black!70!white,->] (t)--node[gap,pos=0.77]{\phantom .}(PRes);
\draw[black!70!white,->] (PRes)--node[gap,pos=0.31]{\phantom .}(Res);
%%%%%%%%%%%%%%%%%%%%
\node (fR) at ($(0,-1.5)+(\myx,\myy)$) {$\Spec \scrW$};
\node (fU) at ($(2,-1.5)+(\myx,\myy)$) {$\Spec \scrV_{\scrI}$};
\node (fV) at ($(4,-1.5)+(\myx,\myy)$) {$\Spec \scrV$};
\node (ft) at ($(0,-3)+(\myx,\myy)$) {$\mathfrak{h}_{\mathbb{C}}$};
\node (fPRes) at ($(2,-3)+(\myx,\myy)$) {$\mathfrak{h}_{\mathbb{C}}/W_{\scrI}$};
\node (fRes) at ($(4,-3)+(\myx,\myy)$) {$\mathfrak{h}_{\mathbb{C}}/W$};
\draw[->] (fR)--(fU);
\draw[->] (fU)--(fV);
\draw[->] (fV)--(fRes);
\draw[->] (fR)--(ft);
\draw[->] (fU)--(fPRes);
\draw[->] (ft)--(fPRes);
\draw[->] (fPRes)--(fRes);
%%%%%%%%%%%%%%%%%%%
\draw[double] (fRes)--(Res);
\draw[double] (fV)--(V);
\draw[<-] (fPRes)--node[pos=0.4,right]{$\scriptstyle \ell_{\scrI}\ell_{\scrI+i}\cdot$}(PRes);
\draw[<-] (fU)--node[right]{$\scriptstyle \sim$}(U);
\draw[<-] (ft)--node[right]{$\scriptstyle\ell_\scrI\ell_{\scrI+i}\cdot$}(t);
\draw[<-] (fR)--node[right]{$\scriptstyle \sim$}(R);
\end{tikzpicture}
\end{array}
\]
The universal property of pullbacks give the induced isomorphisms in the top squares.  The top left square induces the following bottom commutative square on class groups, and the top square is obtained by applying \eqref{class group inclusions} to both $\scrI$ and to $\upomega_i(\scrI)$. 
\[
\begin{tikzpicture}[>=stealth,scale=1.2];
\node[black!70!white] (U) at (2,-1.5) {$\bigoplus_{j\in\Delta} \mathbb{Z}e_j^\star$};
\node[black!70!white] (V) at (4.5,-1.5) {$\bigoplus_{j\in\upomega_i(\scrI)^{\mathrm{c}}} \mathbb{Z}e_j^\star$};
\node[black!70!white] (PRes) at (2,-3) {$\Cl(\scrW)$};
\node[black!70!white] (Res) at (4.5,-3) {$\Cl(\scrV_{\upomega_i(\scrI)})$};
\draw[black!70!white,left hook->] (V)--node[above]{$\scriptstyle i_1$}(U);
\draw[black!70!white,<-] (V)--(Res);
\draw[black!70!white,<-] (U)--node[gap,pos=0.745]{\phantom .}(PRes);
\draw[black!70!white,left hook->] (Res)--node[gap,pos=0.19]{\phantom .}(PRes);
%%%%%%%%%%%%%%%%%%%%
\node (fU) at ($(2,-1.5)+(\myx,\myy)$) {$\bigoplus_{j\in\Delta} \mathbb{Z}e_j^\star$};
\node (fV) at ($(4.5,-1.5)+(\myx,\myy)$) {$\bigoplus_{j\in\scrIc} \mathbb{Z}e_j^\star$};
\node (fPRes) at ($(2,-3)+(\myx,\myy)$) {$\Cl(\scrW)$};
\node (fRes) at ($(4.5,-3)+(\myx,\myy)$) {$\Cl(\scrV_\scrI)$};
\draw[left hook->] (fV)--node[pos=0.4,above]{$\scriptstyle i_2$}(fU);
\draw[<-] (fV)--(fRes);
%\draw[->] (fR)--(ft);
\draw[<-] (fU)--(fPRes);
%\draw[->] (ft)--(fPRes);
\draw[left hook->] (fRes)--(fPRes);
%%%%%%%%%%%%%%%%%%%
\draw[<-] (fRes)--node[pos=0.4,right]{$\scriptstyle \ell_\scrI\ell_{\scrI+i}\cdot$}(Res);
\draw[densely dotted,<-] (fV)--(V);
\draw[<-] (fPRes)--node[pos=0.4,right]{$\scriptstyle \ell_\scrI\ell_{\scrI+i}\cdot$}(PRes);
\draw[<-] (fU)--node[pos=0.6,left]{$\scriptstyle \ell_\scrI\ell_{\scrI+i}\cdot$}(U);
%\draw[double] (ft)--(t);
%\draw[<-] (fR)--node[right]{$\scriptstyle \sim$}(R);
\end{tikzpicture}
\]
In the above diagram, every non-hooked arrow is an isomorphism.   By construction, the dotted arrow takes the characteristic cone 
\[
C_{\upomega_i(\scrI)}\colonequals\{ (z_j)\mid z_j>0\mbox{ for all }j\}\subseteq \bigoplus_{j\in\upomega_i(\scrI)^{\mathrm{c}}} \mathbb{Z}e_j^\star
\]
of $Y_{\upomega_i(\scrI)}$ to the region $\ell_{\scrI}\ell_{\scrI+i}C_{\upomega_i(\scrI)}\subseteq \bigoplus_{j\in\scrIc} \mathbb{Z}e_j^\star$.  Since $\ell_{\scrI}\ell_{\scrI+i}$ acts via the action on coroots, this matches the conventions in \cite[Section~3]{IyamaWemyssTits}. It thus immediately follows from \cite[1.20(1)(d)]{IyamaWemyssTits} that $\ell_{\scrI}\ell_{\scrI+i}C_{\upomega_i(\scrI)}$ and $C_{\scrI}$ are neighbouring regions, adjacent via the wall $z_i=0$.  

We next restrict this information to $3$-folds. As in \eqref{create flop comm diagram},  consider the following commutative diagram.
\begin{equation}
\begin{array}{c}
\begin{tikzpicture}[>=stealth,scale=1.2]
\node[black!70!white] (X) at (0,0) {$\scrX_i^+$};
\node[black!70!white] (Z) at (2,0) {$\scrY_{\upomega_i(\scrI)}$};
\node[black!70!white] (R) at (0,-1.5) {$\Spec \scrR'$};
\node[black!70!white] (U) at (2,-1.5) {$\Spec \scrV_{\upomega_i(\scrI)}$};
\node[black!70!white] (t) at (0,-3) {$\formaldisc$};
\node[black!70!white] (PRes) at (2,-3) {$\mathfrak{h}_{\mathbb{C}}/W_{\upomega_i(\scrI)}$};
\draw[black!70!white,->] (X)--(Z);
\draw[black!70!white,->] (X)--node[gap,pos=0.75]{\phantom .}(R);
\draw[black!70!white,->] (Z)--(U);
\draw[black!70!white,->] (R)--node[gap,pos=0.72]{\phantom .}(U);
\draw[black!70!white,->] (R)--node[gap,pos=0.72]{\phantom .}(t);
\draw[black!70!white,->] (U)--(PRes);
% THIS COMMENTED LINE HAD THE PREVIOUS ARROW. COMMENT EITHER THIS LINE:
%\draw[black!70!white,->] (t)--node[gap,pos=0.735]{\phantom .}node[gap,pos=0.4]{$\scriptstyle\upnu$}(PRes);
% OR THIS ONE:
\draw[black!70!white,->] (t)--node[gap,pos=0.735]{\phantom .}node[below,pos=0.4]{$\scriptstyle\upnu$}(PRes);

%%%%%%%%%%%%%%%%%%%%
\node (fX) at (\myx,\myy) {$\scrX$};
\node (fZ) at ($(2,0)+(\myx,\myy)$) {$\scrY_{\scrI}$};
\node (fR) at ($(0,-1.5)+(\myx,\myy)$) {$\Spec \scrR$};
\node (fU) at ($(2,-1.5)+(\myx,\myy)$) {$\Spec \scrV_{\scrI}$};
\node (ft) at ($(0,-3)+(\myx,\myy)$) {$\formaldisc$};
\node (fPRes) at ($(2,-3)+(\myx,\myy)$) {$\mathfrak{h}_{\mathbb{C}}/W_{\scrI}$};
\draw[->] (fX)--(fZ);
\draw[->] (fX)--(fR);
\draw[->] (fZ)--(fU);
\draw[->] (fR)--(fU);
\draw[->] (fR)--(ft);
\draw[->] (fU)--(fPRes);
\draw[->] (ft)--node[above]{$\scriptstyle\upmu$}(fPRes);
%%%%%%%%%%%%%%%%%%%
\draw[<-] (fPRes)--node[pos=0.35,right]{$\scriptstyle \ell_{\scrI}\ell_{\scrI+i}\cdot$}(PRes);
\draw[<-] (fU)--node[right]{$\scriptstyle \sim$}(U);
\draw[double] (ft)--(t);
\draw[<-] (fR)--node[right]{$\scriptstyle \sim$}(R);
\end{tikzpicture}
\end{array}\label{create flop comm diagram 2}
\end{equation}
As explained by Pinkham \cite{Pinkham}, $\Cl(\scrR)\cong \Cl(\scrV_{\scrI})$, and further $\Cl(\scrV_\scrI)\cong\bigoplus_{j\in\scrIc} \mathbb{Z}e_j^\star$ as explained above.  Hence $\Cl(\scrR)\cong \bigoplus_{j\in\scrIc} \mathbb{Z}e_j^\star$, and under \emph{this} choice of basis $\scrX$ is obtained as the blowup of the characteristic cone $C_\scrI$.  The same analysis holds for $\scrX_i^+$, which is the blowup of the characteristic cone $C_{\upomega_i(\scrI)}$ under the choice of basis $\Cl(\scrR')\cong \bigoplus_{j\in\upomega_i(\scrI)^{\mathrm{c}}} \mathbb{Z}e_j^\star$ induced from $\Cl(\scrR')\cong \Cl(\scrV_{\upomega_i(\scrI)})$.

Pulling across the middle horizontal plane in \eqref{create flop comm diagram 2} it thus follows that the map $\scrX_i^+\to\Spec\scrR$ is obtained by blowing up the region $\ell_\scrI\ell_{\scrI+i}C_{\upomega_i(\scrI)}$ in $\Cl(\scrR)\cong\bigoplus_{j\in\scrIc} \mathbb{Z}e_j^\star$.  Since these are neighbouring regions, separated by the codimension one wall $e_i^\star=0$, it is implicit in \cite{Pinkham} (see also \cite{HomMMP}) that $\scrX_i\to\Spec\scrR$ is the flop at the curve $\Curve_i$.   Since $\scrX^+_i$ is obtained from $\scrY_{\upomega_i(\scrI)}$ via pullback, the statement on Dynkin data follows.

The final statement about all other crepant resolutions follows by iterating over all possible simple flops.  Indeed, the finite indexing set $\mathsf{Cham}(\Delta,\scrI)$ is precisely the combinatorial object which indexes all the chambers \cite[Section~1]{IyamaWemyssTits}, and each chamber $(x,\scrK)$ can be obtained from $(1,\scrI)$ by iteratively applying the wall crossing rule \cite[1.20(2)]{IyamaWemyssTits}.
\end{proof}

\section{Applications}\label{sec: applications}

\noindent As before, consider a smooth $3$-fold flopping contraction $\scrX\to\Spec\scrR$. The main applications of the previous sections are to GV and GW invariants, the Crepant Transformation Conjecture, and to the associated contraction algebras.

\subsection{Tracking GV invariants under flop}\label{sec: GV under flop}
The benefit of Theorem~\ref{thm: produce flop} is that both $\scrX$ and a flop $\scrX_i^+$ can be perturbed using essentially the same classifying map, and thus their curve invariants can be easily compared.  This requires three combinatorial results.

\begin{lemma}\label{lem: induced iso B}
$\ell_{\scrI}\ell_{\scrI+i}\colon\mathfrak{h}\to\mathfrak{h}$ induces an isomorphism $\mathsf{M}_i \colon \mathfrak{h}_{\upomega_i(\scrI)}\to\mathfrak{h}_{\scrI}$.
\end{lemma}
\begin{proof}
Set $w=\ell_{\scrI}\ell_{\scrI+i}$, then it suffices to prove that $w\cdot$ restricts to an isomorphism between the subspace spanned by $\{\upalpha_j\mid j\in\upomega_i(\scrI)\}$ and the subspace spanned by $\{\upalpha_j\mid j\in\scrI\}$.  To see this, recall that $\ell_\Gamma\upalpha_i=-\upalpha_{\upiota_\Gamma(i)}$ for all Dynkin $\Gamma$ and all $i\in\Gamma$, where $\upiota_\Gamma$ is the Dynkin involution on $\Gamma$.

For all $j\in \upomega_i(\scrI)\subseteq\scrI+i$, it follows that $\ell_\scrI\ell_{\scrI+i}\upalpha_j=-\ell_{\scrI}\upalpha_{\upiota_{\scrI+i}(j)}$.  Now $\upiota_{\scrI+i}\colon\upomega_i(\scrI)\to\scrI$ is a bijection. Indeed, $\upiota_{\scrI+i}\colon\scrI+i\to\scrI+i$ is a bijection, sending $\upiota_{\scrI+i}(i)$ to $i$, and so removing these elements gives the claimed bijection. Hence for all $j\in\upomega_i(\scrI)$, it follows that $\upiota_{\scrI+i}(j)\in\scrI$ and so $\ell_{\scrI}\upalpha_{\upiota_{\scrI+i}(j)}=-\upalpha_{\upiota_\scrI\upiota_{\scrI+i}(j)}$. Combining gives $\ell_\scrI\ell_{\scrI+i}\upalpha_j=\upalpha_{\upiota_\scrI\upiota_{\scrI+i}(j)}$ for all $j\in\upomega_i(\scrI)$. Since $\upiota_\scrI\upiota_{\scrI+i}(j)\in\scrI$, this proves the claim.
\end{proof}
Write $\mathsf{M}_i$ for the induced isomorphism in Lemma~\ref{lem: induced iso B}, so that the following diagrams commute.
\begin{equation}
\begin{array}{c}
\begin{tikzpicture}
\node (A) at (0,0) {$\mathfrak{h}$};
\node (B) at (2.5,0) {$\mathfrak{h}$};
\node (a) at (0,-1.5) {$\mathfrak{h}_{\upomega_i(\scrI)}$};
\node (b) at (2.5,-1.5) {$\mathfrak{h}_{\scrI}$};
\draw[->] (A)--node[above]{$\scriptstyle  \ell_{\scrI}\ell_{\scrI+i}\cdot$}(B);
\draw[->] (A)--node[left]{$\scriptstyle \uppi_{\upomega_i(\scrI)}$}(a);
\draw[->] (B)--node[right]{$\scriptstyle \uppi_{\scrI}$}(b);
\draw[->,densely dotted] (a)--node[above]{$\scriptstyle \mathsf{M}_i$}(b);
\end{tikzpicture}
\end{array}
\quad
\begin{array}{c}
\begin{tikzpicture}
\node (A) at (0,0) {$\mathfrak{h}$};
\node (B) at (2.5,0) {$\mathfrak{h}$};
\node (a) at (0,-1.5) {$\mathfrak{h}_{\scrI}$};
\node (b) at (2.5,-1.5) {$\mathfrak{h}_{\upomega_i(\scrI)}$};
\draw[->] (A)--node[above]{$\scriptstyle  \ell_{\upomega_i(\scrI)}\ell_{\scrI+i}\cdot$}(B);
\draw[->] (A)--node[left]{$\scriptstyle \uppi_{\scrI}$}(a);
\draw[->] (B)--node[right]{$\scriptstyle \uppi_{\upomega_i(\scrI)}$}(b);
\draw[->,densely dotted] (a)--node[above]{$\scriptstyle \mathsf{M}_i^{-1}$}(b);
\end{tikzpicture}
\end{array}
\label{defn Mi}
\end{equation}
In essence, the bases of $\mathfrak{h}_{\upomega_i(\scrI)}$ and $\mathfrak{h}_{\scrI}$ really only differ at one element. Indeed, setting $e_j = \uppi_\scrI(\upalpha_j)$ then $\{ e_j \mid j \not\in \scrI \}$ is a basis for $\mathfrak{h}_\scrI$. On the other hand for $\mathfrak{h}_{\omega_i(\scrI)}$ we abuse notation, setting $e_j=\uppi_{\upomega_i(\scrI)}(\upalpha_j)$ whenever $j\notin\scrI+i$, and $e_i=\uppi_{\upomega_i(\scrI)}(\upalpha_{\upiota_{\scrI+i}(i)})$. Then $\{e_j,e_i\mid j\notin\scrI+i\}$ is a basis for $\mathfrak{h}_{\upomega_i(\scrI)}$.
\begin{lemma}\label{lem: action of Mi}
For $i \in \scrIc$ the action of $\mathsf{M}_i$ is given in terms of the above bases as
\[
e_k
\mapsto
\begin{cases}
e_k+\uplambda_{k}e_i&\mbox{if }k\notin\scrI+i\\
-e_i&\mbox{if }k=i
\end{cases}
\]
for some $\uplambda_k \in \ZZ_{\geq 0}$.
\end{lemma}
\begin{proof}
In \eqref{defn Mi}, given any $\sum a_i\upalpha_i$, since $\ell_\scrI\ell_{\scrI+i}$ consists only of reflections $s_i$ with $i\in\scrI+i$, the map $\ell_\scrI\ell_{\scrI+i}$ cannot change the coefficient of any  $a_j$ with $j\notin\scrI+i$.  The claim that the induced map $\mathsf{M}_i$ sends $e_k\mapsto e_k+\uplambda_{k}e_i$  if $k\notin\scrI+i$ follows. We next claim that $\uplambda_k$ is positive.  In the decomposition of $\ell_\scrI\ell_{\scrI+i}\upalpha_k$ into simple roots, there is at least some positive coefficient (namely the coefficient of $\upalpha_k$, which is $1$). Hence all coefficients must be positive, in particular the coefficient of $\upalpha_{\upiota_{\scrI+i}(i)}$. But under the induced map, this coefficient is what gives $\uplambda_k$ in the claim.   

Now as in Lemma~\ref{lem: induced iso B}, for all Dynkin $\Gamma$ and all $i\in\Gamma$, $\ell_\Gamma\upalpha_i=-\upalpha_{\upiota_\Gamma(i)}$.  Thus since $\upiota_{\scrI+i}(i)\in\scrI+i$, it follows that
\[
\mathsf{M}_ie_i
\stackrel{\scriptstyle\eqref{defn Mi}}{=}
\uppi_\scrI(\ell_\scrI\ell_{\scrI+i}\upalpha_{\upiota_{\scrI+i}(i)})
=\uppi_\scrI(-\ell_\scrI\upalpha_i)=\uppi_\scrI(-\upalpha_i)=-e_i,
\]
where we have used the fact that $\ell_\scrI$ only changes coefficients in $\scrI$, and $\uppi_\scrI$ forgets these.
\end{proof}

\begin{cor}\label{cor: action of Mi pos neg}
If $\upbeta\in\mathfrak{h}_{\upomega_i(\scrI)}$ is a restricted root, the following hold.
\begin{enumerate}
\item If $\upbeta\in\mathbb{Z}e_i$, say $\upbeta=ze_i$, then $\mathsf{M}_i\cdot(ze_i)=-ze_i$.
\item If $\upbeta\notin\mathbb{Z}e_i$, then all entries of $\mathsf{M}_i\cdot\upbeta$ are positive.
\end{enumerate}
\end{cor}
\begin{proof}
The first part is an immediate consequence of the $k=i$ case in Lemma~\ref{lem: action of Mi}. For the second part, by Lemma~\ref{lem: action of Mi} the only coefficient of $\upbeta=\sum\upmu_ie_i$ that can change under $\mathsf{M}_i$ is the coefficient on $e_i$. Hence, provided there is some other positive coefficient $\upmu_k$, this survives under $\mathsf{M}_i$, so $\mathsf{M}_i\cdot\upbeta$ has at least one positive entry.  Now by assumption $\upbeta$ is a restricted root, say $\uppi_{\upomega_i(\scrI)}(\upalpha)=\upbeta$, Under \eqref{defn Mi}, $w\cdot\upalpha$ is a root restricting to $\mathsf{M}_i\cdot \upbeta$ and this root $w\cdot\upalpha$ must contain at least one positive coefficient, since $\mathsf{M}_i\cdot \upbeta$ does. Hence  all must be positive.  In particular, all entries of $\mathsf{M}_i\cdot \upbeta$ must also be positive.
\end{proof}

The following is one of our main results. In order to obtain a unified statement, set $|\mathsf{M}_i\cdot\upbeta|$ to be the curve class obtained from $\mathsf{M}_i\cdot\upbeta$ by making every coefficient positive.

\begin{theorem}\label{thm: GV under flop}
With the notation as above, for any curve class $\upbeta\in\Achow_1(\scrX_i^+)\cong\mathfrak{h}_{\upomega_i(\scrI)}$,
\begin{align*}
n_{\upbeta,\scrX^+_i}&=
\begin{cases}
n_{\upbeta,\scrX}&\mbox{if }\upbeta\in \mathbb{Z}e_i=\mathbb{Z}\Curve_i^+\\
n_{\kern 1pt\mathsf{M}_i\cdot \upbeta, \scrX}&\mbox{else}
\end{cases}\\
&=
n_{\kern 1pt|\mathsf{M}_i\cdot \upbeta|, \scrX}
\end{align*}
\end{theorem} 
\begin{proof}
With respect to the notation in \eqref{create flop comm diagram}, perturbing $\upmu$ to $\upmu_t$ gives, by composition, a perturbation of $\upnu$ to $\upnu_t$.

Set $w=\ell_{\scrI}\ell_{\scrI+i}$. Then for any positive root $\upalpha$ for which $\uppi_{\upomega_i(\scrI)}(\upalpha)=\upbeta$, 
\begin{align*}
n_{\upbeta,\scrX^+_i}&=|\upnu_t^{-1}(\scrD_{\upalpha,\C}/W_{\omega_i(\scrI)})|\tag{by \eqref{eqn: nbeta equals intersection with hyperplane} applied to $\scrX_i^+$}\\
&=|\upmu_t^{-1}(\scrD_{w\cdot\upalpha,\C}/W_\scrI)|.\tag{since $\upnu=(w \cdot)^{-1} \circ\upmu$}
\end{align*}
Now  by \eqref{defn Mi} we have $\uppi_{\scrI}(w\cdot \upalpha)=\mathsf{M}_i\circ\uppi_{\upomega_i(\scrI)}(\upalpha)=\mathsf{M}_i\cdot \upbeta$ and so $w\cdot\upalpha$ is a lift of $\mathsf{M}_i\cdot\upbeta$, albeit not necessarily a positive one.  

\noindent
\emph{Case 1.} If $\upbeta\notin\mathbb{Z}\Curve^+_i$, then by Corollary~\ref{cor: action of Mi pos neg} all entries of $\mathsf{M}_i\cdot\upbeta$ are positive, and further as argued in the proof, $w\cdot\upalpha$ is a positive root restricting to $\mathsf{M}_i\cdot \upbeta$. \eqref{eqn: nbeta equals intersection with hyperplane} then implies that $|\upmu_t^{-1}(\scrD_{w\cdot\upalpha,\C}/W_\scrI)|=n_{\kern 1pt\mathsf{M}_i\cdot \upbeta, \scrX}$. 

\noindent
\emph{Case 2.} If $\upbeta\in\mathbb{Z}\Curve^+_i$, then by  Corollary~\ref{cor: action of Mi pos neg}, $\mathsf{M}_i\cdot\upbeta=-\upbeta$.  Arguing as above, it follows that  $w\cdot\upalpha$ is negative root restricting to $\mathsf{M}_i\cdot\upbeta=-\upbeta$, and thus $-w\cdot\upalpha$ is positive root restricting to $\upbeta$. But negating a root does not affect the hyperplane, and combining this fact with \eqref{eqn: nbeta equals intersection with hyperplane} it follows that
\[
|\upmu_t^{-1}(\scrD_{w\cdot\upalpha,\C}/W_\scrI)|=|\upmu_t^{-1}(\scrD_{-w\cdot\upalpha,\C}/W_\scrI)|=n_{\upbeta,\scrX}.
\]
This covers both cases.  For the final equality, note in case 1 that $|\mathsf{M}_i\cdot\upbeta|=\mathsf{M}_i\cdot\upbeta$ since all coefficients are already positive, and in case 2 that $|\mathsf{M}_i\cdot\upbeta|=|-\upbeta|=\upbeta$.
\end{proof}

\begin{example}\label{ex: running example GV after flop}
Consider the running Example~\ref{example: intro}.  Then after flop of the right pink curve, by Theorem~\ref{thm: produce flop} and Example~\ref{ex: changing dual graph} we obtain $\upomega_i(\scrI)=\Eeight{B}{P}{B}{Or}{B}{B}{B}{B}$. Hence the restricted roots, and thus curve classes giving nonzero GV invariants, on the flopped space $\scrX_i^+$ are as follows, where the hyperplanes are drawn in $\Uptheta_{\upomega_i(\scrI)}$.
\[
\begin{array}{cccc}
\begin{array}{c}
\begin{tikzpicture}[scale=0.5]
%01, the x-axis axis 
\draw[->,densely dotted] (180:2cm)--(0:2cm);
\node at (0:2.5) {$\scriptstyle x$};
%10, the y-axis axis
\draw[->,densely dotted] (-90:2cm)--(90:2cm);
\node at (90:2.5) {$\scriptstyle y$};
%\draw[densely dotted,gray] (0,0) circle (2cm);
\end{tikzpicture}
\end{array}
%%%%%%%%%%%%%%%%%%%%%%%%%%%%%
&
%%%%%%%%%%%%%%%%%%%%%%%%%%%%%%
\begin{array}{c}
\begin{tikzpicture}[scale=1]
%01, the x-axis axis 
\draw[line width=\mythick mm,Pink] (180:2cm)--(0:2cm);
\node at (180:2.2) {$\scriptstyle 1$};
%11 
\draw[line width=\mythick mm,Pink] (135:2cm)--(-45:2cm);
\node at (135:2.2) {$\scriptstyle 3$};
%21
\draw[line width=\mythick mm, Grey] (126.87:2cm)--(-53.13:2cm);
\node at (126.87:2.2) {$\scriptstyle 1$};
%31
\draw[line width=\mythick mm, Green] (123.69:2cm)--(-56.31:2cm);
\node at (123.69:2.2) {$\scriptstyle 1$};
%41
\draw[line width=\mythick mm, Blue] (116.55:2cm)--(-63.45:2cm);
\node at (116.55:2.2) {$\scriptstyle 2$};
%10, the y-axis axis
\draw[line width=\mythick mm,Orange] (90:2cm)--(-90:2cm);
\node at (90:2.2) {$\scriptstyle 1$};
%\draw[densely dotted,gray] (0,0) circle (2cm);
\end{tikzpicture}
\end{array}&
\begin{array}{c}
\begin{tabular}{ccc}
\toprule
Restricted Root&\\
\midrule
$01$&$\tikz\draw[line width=\mythick mm, Pink] (0,0) -- (0.25,0);$\\
$11,22,33$&$\tikz\draw[line width=\mythick mm, Pink] (0,0) -- (0.25,0);$\\
$43$&$\tikz\draw[line width=\mythick mm, Grey] (0,0) -- (0.25,0);$\\
$32$&$\tikz\draw[line width=\mythick mm, Green] (0,0) -- (0.25,0);$\\
$21,42$&$\tikz\draw[line width=\mythick mm,Blue] (0,0) -- (0.25,0);$\\
$10$&$\tikz\draw[line width=\mythick mm, Orange] (0,-0.15) -- (0,0.15);$\\
\bottomrule
\end{tabular}
\end{array}
\end{array}
\]
Write $1$ for the leftmost pink node, $2$ for the orange node, and $2'$ for the rightmost pink node (in $\scrI$).  Under this wall crossing $\ell_{\scrI}\ell_{\scrI+i}$ is very large, however the morphism
\[
\mathsf{M}_i\colon\mathfrak{h}_{\upomega_i(\scrI)}\to\mathfrak{h}_{\scrI}
\]
is easily described: in the notation of Lemma~\ref{lem: action of Mi}, $\uplambda_1=1$ and thus $\mathsf{M}_i$ sends $\upmu_1e_1+\upmu_2e_2\mapsto \upmu_1e_1+(\upmu_1-\upmu_2)e_{2'}$.  Under the dual transformation between the hyperplane arrangements in Example~\ref{example: intro} and here, the pictures are drawn so that hyperplanes are sent to hyperplanes in such a way that the colours are preserved.

Indeed, $\mathsf{M}_i$ sends $01\mapsto0-\!1$, with all other restricted roots being permuted; e.g.\ $31\mapsto 32$. In particular, by Theorem~\ref{thm: GV under flop} the GV invariants on $\scrX_i^+$ can be obtained from the GV invariants on $\scrX$ as follows
\[
\begin{tabular}{ccccc}
\toprule
GV on $\scrX^+_i$&GV on $\scrX$&\\
\midrule
$01$&$01$&$\tikz\draw[line width=\mythick mm, Pink] (0,0) -- (0.25,0);$\\
$10$&$11$&$\tikz\draw[line width=\mythick mm, Orange] (0,0) -- (0.25,0);$\\
$21,42$&$21, 42$&$\tikz\draw[line width=\mythick mm, Blue] (0,0) -- (0.25,0);$\\
$32$&$31$&$\tikz\draw[line width=\mythick mm, Green] (0,0) -- (0.25,0);$\\
$43$&$41$&$\tikz\draw[line width=\mythick mm, Grey] (0,0) -- (0.25,0);$\\
$11,22,33$&$10, 20, 30$&$\tikz\draw[line width=\mythick mm, Pink] (0,0) -- (0.25,0);$\\
\bottomrule
\end{tabular}
\] 
\end{example}

\begin{remark}\label{rem: movable via alg geom}
As explained in the introduction, the finite arrangement $\scrH_{\scrI}$ equals the movable cone. The multiplicities of the restricted roots are assigned to each wall, and this enhancement is required in order to describe the curve-counting invariants.  It is possible, albeit not a priori obvious, to enhance the movable cone without Dynkin combinatorics.  Given a chamber corresponding to some crepant resolution $\scrX^\dag \to\Spec\scrR$ say, then the multiplicities on the walls of that chamber turn out to correspond to the lengths of all the individual single-curve contractions obtained from $\scrX^\dag$. The issue with this method is that, whilst it explains walls, it does not explain \emph{hyperplanes}: it is not so clear that \emph{every} chamber touching the hyperplane containing the said wall should be enriched with the same scheme-theoretic length.  This geometric fact falls out from our approach. 
\end{remark}

\subsection{Tracking fundamental regions}\label{sec:track fund}
The previous subsection tracked GV invariants from $\scrX$ to $\scrX^+_i$.  As with the movable cone, it is possible to fix $\scrX$ and track all other crepant resolutions back to $\scrX$.   

As notation, recall that the fixed $\scrX\to\Spec\scrR$ has an associated $\Uptheta_\scrI$ in Notation~\ref{notation: subset stuff}, and recall from \eqref{defn Mi} that there is a map $\mathsf{M}_i^{-1}\colon\mathfrak{h}_{\scrI}\to\mathfrak{h}_{\upomega_i(\scrI)}$. Write
\[
\mathsf{N}_i\colon \Uptheta_{\upomega_i(\scrI)}\to\Uptheta_\scrI
\] 
for the dual.  Below, $\Uptheta_\scrI$ will be temporarily be written $\Uptheta_{\scrX}$, to allow for the flexibility of considering another crepant resolution $\scrY\to\Spec\scrR$ which has  associated $\Uptheta_{\scrY}$. 
\begin{definition}
Let  $\scrY\to\Spec\scrR$ be a crepant resolution.  Consider a chain of flops, each flopping a single irreducible curve, that links $\scrY$ to $\scrX$, and the resulting maps
\[
\Uptheta_{\scrY}\xrightarrow{\mathsf{N}_{i_1}}\hdots\xrightarrow{\mathsf{N}_{i_t}}\Uptheta_\scrX.
\]
The composition will be called the comparison map, and will be written $\mathsf{N}\colon\Uptheta_\scrY\to\Uptheta_\scrX$.
\end{definition}
By \cite[4.8]{HW2} the comparison map $\mathsf{N}$ is independent of the choice of chain of flops.

\begin{definition}
Given a crepant resolution $\scrY\to\Spec\scrR$, the fundamental region $\Fund_\scrY$ of $\Uptheta_\scrY$ is defined as the intersection of the infinite hyperplane arrangement inside $\Uptheta_\scrY$ with the unit box $\{ (\upvartheta_i) \mid 0\leq \upvartheta_i\leq 1\mbox{ for all }i\}$.
\end{definition}
%An example of $\Fund_{\scrX}$ is illustrated in \eqref{running example infinite fund region 1}.
\begin{prop}
For any crepant resolution $\scrY\to\Spec\scrR$, $\mathsf{N}(\Fund_{\scrY})$ generates $\Uptheta_\scrI$ via translation.  Furthermore two $\mathsf{N}(\Fund_{\scrX_1})$ and  $\mathsf{N}(\Fund_{\scrX_2})$ share a codimension one wall if and only if $\scrX_1$ and $\scrX_2$ are connected by a flop at a single curve.
\end{prop}
\begin{proof}
Since the axes belong to the finite hyperplane arrangement in $\Uptheta_{\scrY}$, and the definition of the infinite arrangement involves translating this finite collection of hyperplanes over $\mathbb{Z}$ or at worst $\tfrac{1}{k}\mathbb{Z}$ (see Subsections~\ref{sec: GW intro} and \ref{sec: hyperplane arrangements}), it is clear that the fundamental region $\Fund_\scrY$ generates the arrangement in $\Uptheta_\scrY$.  The first statement then follows, since $\mathsf{N}$ is known to preserve the infinite arrangements \cite[Section~9]{IyamaWemyssTits}.  Since the only codimension one wall that the fundamental regions can share belong to the finite arrangement, the last statement is really a statement on the movable cone, which is e.g.\ \cite[Sections~5--6]{HomMMP}. 
\end{proof}

\begin{example}
Write $\scrY\to\Spec\scrR$ for the crepant resolution obtained after flop in Example~\ref{ex: running example GV after flop}. Then the region $\mathsf{N}(\Fund_{\scrY})$ is illustrated below, where for clarity we have illustrated the images of the $x$ and $y$ co-ordinates in Example~\ref{ex: running example GV after flop} under the map $\mathsf{N}$.
\[
\begin{array}{cc}
\begin{array}{c}
\begin{tikzpicture}[scale=0.5]
%01, the x-axis axis 
\draw[->,densely dotted] (0,-0.5)--($(0,-0.5)+(52.5:4.5cm)$);
\node at ($(0,-0.5)+(60:4cm)$) {$\scriptstyle \mathsf{N}x$};
%spacing
\node at (0,-0.6){};
%10, the y-axis axis
\draw[->,densely dotted] (0,-0.5)--($(0,-0.5)+(33:3.25cm)$);
\node at ($(0,-0.5)+(40:3.25cm)$)  {$\scriptstyle \mathsf{N}y$};
\end{tikzpicture}
\end{array}
%%%%%%%%%%%%%%%%%%%%%%%%%%%%
&
%%%%%%%%%%%%%%%%%%%%%%%%%%%%
\begin{array}{c}
\includegraphics[angle=0,width=7.5cm,height=4cm]{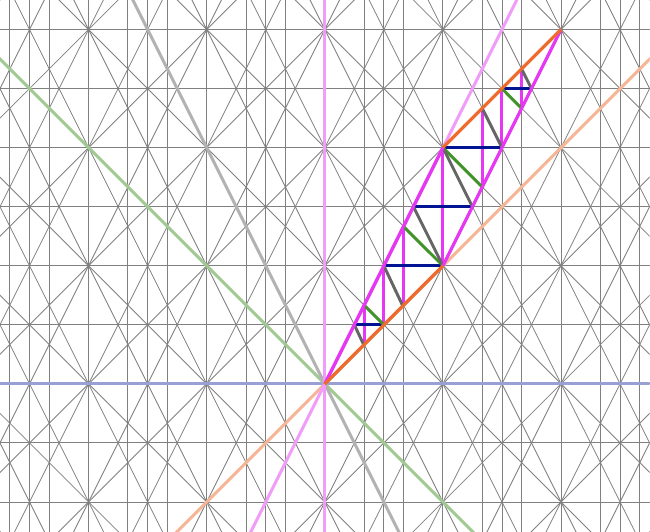}
\end{array}
\end{array}
\]
It is visually clear that both $\Fund_{\scrX}$ in \eqref{running example infinite fund region 1} and $\mathsf{N}(\Fund_{\scrY})$ above individually generate $\scrH_{\scrI}^{\aff}$, via translation, and that $\Fund_{\scrX}$  and $\mathsf{N}(\Fund_{\scrY})$ are different.
\end{example}

\begin{remark}
The above example gives a visual proof of the Crepant Transformation Conjecture of Subsection~\ref{sec:CTC proof} below. The regions $\Fund_{\scrX}$  and $\mathsf{N}(\Fund_{\scrY})$ are different. But they generate the same object, namely $\scrH_{\scrI}^{\aff}$, which by Corollary~\ref{cor: quantum is hyper} is the pole locus of the GW quantum potential. Thus, although the curve invariants of $\scrX$ and $\scrY$, captured in the fundamental regions, are technically different, after a change in variables (namely $\mathsf{N}$) they can be compared, where they generate the same object. 
\end{remark}

\begin{remark}\label{rem: matrix Mk}
The matrix $\mathsf{N}_i$ appears via moduli tracking in the HomMMP \cite[5.4]{HomMMP}, and via the K-theory of contraction algebras \cite[2.4]{AugustWemyss}. In contrast, $\mathsf{M}_i$ from \eqref{defn Mi} is the dual, and it arises from the change in dimension vector \cite[5.4]{HomMMP}, or in the K-theory of projective modules \cite[3.2]{HW2}.  For more details, see \cite[2.4]{AugustWemyss} and references therein.
\end{remark}

\subsection{Crepant Transformation Conjecture} \label{sec:CTC proof}
We make no attempt at a comprehensive summary of the Crepant Transformation Conjecture (CTC), and instead refer the reader to \cite{CoatesRuan, CoatesIritaniJiang, BryanGraber,YPLeeLectures}. For a pair of smooth varieties related via a sequence of flops, the conjecture asserts that their quantum potentials should coincide, under a suitable identification of (co)homologies and analytic continuation in the Novikov parameters. There has been extensive work on this conjecture within both algebraic and symplectic geometry  \cite{LiRuan, McLean, LLWMotives, LLW1, LLW2, LLQW3}.

Here we prove the CTC for germs of isolated $3$-fold flops, as a direct application of the expression for the quantum potential in Theorem~\ref{thm: structure of quantum potential} together with the construction of flops via simultaneous partial resolutions in Theorem~\ref{thm: flop by sim res}.  This gives the first algebraic-geometric proof of the CTC for flops of arbitrary type (for recent symplectic developments, see \cite{McLean}).

\medskip
As before, consider a curve $\Curve_i \subseteq \scrX$ and let $\scrX_i^+$ be the flop of $\scrX$ at $\Curve_i$. Recall that the following vector spaces are based by the sets of exceptional curves
\[ 
\mathfrak{h}_{\scrI,\C}= \Achow_1(\scrX)_{\C}= \langle \Curve_j \mid j \in \scrIc \rangle_{\C}, \qquad \mathfrak{h}_{\upomega_i(\scrI),\C} = \Achow_1(\scrX_i^+)_{\C} = \langle \Curve_j,\Curve_i^+\mid j\notin\scrI+i  \rangle_{\C}. 
\]
where as in Lemma~\ref{lem: action of Mi} we abuse notation by denoting the flopped curve $\Curve_i^+$ instead of $\Curve_{\raisemath{10pt}{\upiota_{\scrI+i}(i)}}^+$. As explained in Subsection~\ref{sec: GV under flop}, there is an explicit transformation matrix
\[ \mathsf{M}_i \colon \Achow_1(\scrX_i^+)_{\C} \to \Achow_1(\scrX)_{\C}.\]
This is the complexification of the matrix $\mathsf{M}_i$ from earlier, but we use the same symbol. Let $\mathsf{N}_i$ be the matrix dual to $\mathsf{M}_i^{-1}$ which can be viewed as a linear map
\[ 
\mathsf{N}_i \colon \mathrm{H}^2(\scrX_i^+;\C) \to \mathrm{H}^2(\scrX;\C) 
\]
with the property that $\Nsf_i \upgamma \cdot \upbeta = \upgamma \cdot \Msf_i^{-1} \upbeta$ for $\upgamma \in \mathrm{H}^2(\scrX_i^+;\C)$ and $\upbeta \in \Achow_1(\scrX)_{\C}$. We  notate the Novikov co-ordinates on the parameter spaces for the quantum potentials by
\[
\begin{array}{ll}
\{ \qsf_j \mid j \in \scrIc\} &\text{on $\Achow_1(\scrX)_{\C}$}, \\
\{ \rsf_j \mid j \in \scrIc \} &  \text{on $\Achow_1(\scrX_i^+)_{\C}$.}
\end{array}
\]
The matrix $\mathsf{M}_i^{-1}$ defines a monomial co-ordinate transformation relating the two sets of Novikov parameters. Writing monomials as
\[ \qsf^{\upbeta} \colonequals \prod_{j \in \scrIc} \qsf_j^{m_j},\qquad \rsf^{\upbeta} \colonequals \prod_{j \in \scrIc} \rsf_j^{m_j}\]
this is given by the equation
\begin{equation} \label{eqn: transformation of Novikov parameters} 
\qsf^{\upbeta} = \rsf^{\Msf_i^{-1} \upbeta}.
\end{equation}
By Lemma~\ref{lem: action of Mi}, $\qsf_i = \rsf_i^{-1}$ whereas every other $\qsf_{j}$ is a monomial in the $\rsf_j$ with non-negative coefficients.

Lastly, recall the quantum potentials of $\scrX$ and $\scrX_i^+$ from Subsection~\ref{sec: quantum potential}, which to ease notation will be written $\Phi$ and $\Phi^+$, namely
\begin{align*} \Phi_{\qsf}(\upgamma_1,\upgamma_2,\upgamma_3) & \colonequals \Phi_{\qsf}^{\scrX}(\upgamma_1,\upgamma_2,\upgamma_3), \\
\Phi_{\rsf}^+(\upgamma_1,\upgamma_2,\upgamma_3) & \colonequals \Phi_{\rsf}^{\scrX_i^+}(\upgamma_1,\upgamma_2,\upgamma_3).
\end{align*}
The equation \eqref{eqn: transformation of Novikov parameters} will be used to express the quantum potential of $\scrX$ in the variables $\rsf$, and this will be denoted $\Phi_{\rsf}(\upgamma_1,\upgamma_2,\upgamma_3)$.

\begin{corollary}[Crepant Transformation Conjecture] \label{thm: CTC} On the $\rsf$ parameter space, the quantum potentials of $\scrX$ and $\scrX_i^+$ coincide, up to the following explicit term which does not depend on the Novikov parameters
\begin{equation} \label{eqn: CTC} \Phi^{+}_{\mathsf{r}}\big(\upgamma_1, \upgamma_2,\upgamma_3\big) - \Phi_{\mathsf{r}}(\mathsf{N}_i  \upgamma_1,\mathsf{N}_i \upgamma_2,\mathsf{N}_i \upgamma_3) = -(\upgamma_1 \cdot \Curve_i^+)(\upgamma_2 \cdot \Curve_i^+)(\upgamma_3\cdot \Curve_i^+) \sum_{k \geq 1} k^3 n_{k\Curve_i,\scrX}. 
\end{equation}
\end{corollary}
\noindent
More precisely the pole loci are transformed into each other via \eqref{eqn: transformation of Novikov parameters}, and away from these the analytic continuations constructed in Theorem~\ref{thm: structure of quantum potential} coincide.

\begin{remark} \label{rmk: CTC non-compact} The quantum potentials of $\scrX$ and $\scrX_i^+$ have no constant terms in their respective Novikov parameters, due to the absence of a perfect pairing on cohomology (see Remark~\ref{rmk: no algebra}). However, the change of parameters \eqref{eqn: transformation of Novikov parameters} introduces constant terms into $\Phi_{\mathsf{r}}(\upgamma_1,\upgamma_2,\upgamma_3)$, which form the right-hand side of \eqref{eqn: CTC}. It follows that the quantum potentials agree once these extraneous constant terms are removed from $\Phi_{\rsf}(\upgamma_1,\upgamma_2,\upgamma_3)$. In particular, $\Phi^+$ can be effectively reconstructed from $\Phi$. In situations where an ordinary cup product can be defined, the additional terms on the right-hand side account for the defect between the cup products on $\scrX$ and $\scrX_i^+$, see e.g. \cite[Subsection~4.3 and Equation~(4.4)]{MorrisonKaehler}.
\end{remark}

\begin{remark} 
The expansion points for the quantum potentials differ, as
\[ (\rsf_j)_{j \in \scrIc} = (0,\ldots,0) \ \Leftrightarrow \ (\qsf_j)_{j \in \scrIc} = (0,\ldots,0,\infty,0,\ldots,0)\]
with $\infty$ in the $i$th position. Thus, the term $\Phi_{\rsf}(\Nsf_i \upgamma_1,\Nsf_i \upgamma_2,\Nsf_i \upgamma_3)$ is analytically continued from $\qsf_i=0$ to $\qsf_i=\infty$, the analytic continuation occurring precisely in the Novikov parameter corresponding to the flopped curve.
\end{remark}

\begin{proof}[Proof of \ref{thm: CTC}] 
We explicitly match both sides, using our knowledge of the structure of the quantum potentials (Theorem~\ref{thm: structure of quantum potential}) and the behaviour of the GV invariants under the flop (Theorem~\ref{thm: GV under flop}).

Separating curve classes according to whether or not they are a multiple of $\Curve_i$,  the quantum potential for $\scrX$ may be written as the sum of contributions
\begin{align*} 
\Phi_{\qsf}(\Nsf_i\upgamma_1,\Nsf_i\upgamma_2,\Nsf_i\upgamma_3) & = \mathsf{G}_{\qsf}(\Nsf_i \upgamma_1,\Nsf_i \upgamma_2, \Nsf_i \upgamma_3) + \mathsf{H}_{\qsf}(\Nsf_i \upgamma_1,\Nsf_i \upgamma_2, \Nsf_i \upgamma_3) 
\end{align*}
where
\begin{align*}
\mathsf{G}_{\qsf}(\Nsf_i \upgamma_1,\Nsf_i \upgamma_2, \Nsf_i \upgamma_3) & =\quad\, \sum_{k \geq 1} n_{k\Curve_i,\scrX} \, (\Nsf_i\upgamma_1 \cdot k\Curve_i)(\Nsf_i\upgamma_2 \cdot k\Curve_i) (\Nsf_i\upgamma_3 \cdot k\Curve_i) \dfrac{\qsf_i^k}{1-\qsf_i^k} \\ 
\mathsf{H}_{\qsf}(\Nsf_i \upgamma_1,\Nsf_i \upgamma_2, \Nsf_i \upgamma_3) & = \sum_{\substack{\upbeta \in \Achow_1(\scrX) \\ \upbeta \neq k\Curve_i}} n_{\upbeta,\scrX} \, (\Nsf_i\upgamma_1\cdot\upbeta)(\Nsf_i\upgamma_2\cdot\upbeta)(\Nsf_i\upgamma_3\cdot\upbeta) \dfrac{\qsf^\upbeta}{1-\qsf^\upbeta}. 
\end{align*}
Similarly, write the quantum potential of $\scrX_i^+$ as
\begin{align*}
\Phi^{+}_{\rsf}(\upgamma_1,\upgamma_2,\upgamma_3) & = \mathsf{G}^{+}_{\rsf}(\upgamma_1,\upgamma_2,\upgamma_3) +\mathsf{H}^{+}_{\rsf}(\upgamma_1,\upgamma_2,\upgamma_3),
\end{align*}
where
\begin{align*} 
\mathsf{G}^{+}_{\rsf}(\upgamma_1,\upgamma_2,\upgamma_3) & =  \quad\, \sum_{k \geq 1} n_{k\Curve_i^+,\scrX_i^+} \, (\upgamma_1 \cdot k\Curve_i^+)(\upgamma_2 \cdot k\Curve_i^+) (\upgamma_3 \cdot k\Curve_i^+) \dfrac{\rsf_i^{k}}{1-\rsf_i^{k}} \\ 
\mathsf{H}^{+}_{\rsf}(\upgamma_1,\upgamma_2,\upgamma_3) & =  \sum_{\substack{\upbeta \in \Achow_1(\scrX) \\ \upbeta \neq k\Curve_i}} n_{\upbeta,\scrX_i^+} \, (\upgamma_1\cdot\upbeta)(\upgamma_2\cdot\upbeta)(\upgamma_3\cdot\upbeta) \dfrac{\rsf^{\upbeta}}{1-\rsf^{\upbeta}}. 
\end{align*}
We begin with the $\mathsf{G}$ terms. Using \eqref{eqn: transformation of Novikov parameters} to change variables from $\qsf$ to $\rsf$ gives
\begin{align*} 
\mathsf{G}_{\rsf}(\Nsf_i\upgamma_1,\Nsf_i\upgamma_2,\Nsf_i\upgamma_3) & = \sum_{k \geq 1} n_{k\Curve_i,\scrX} \, (\Nsf_i\upgamma_1 \cdot k\Curve_i)(\Nsf_i\upgamma_2 \cdot k\Curve_i) (\Nsf_i\upgamma_3 \cdot k\Curve_i) \dfrac{\rsf_i^{-k}}{1-\rsf_i^{-k}} \\
& = (\upgamma_1 \cdot \Curve_i^+) (\upgamma_2 \cdot \Curve_i^+) (\upgamma_3 \cdot \Curve_i^+) \sum_{k \geq 1} k^3 n_{k\Curve_i,\scrX} \dfrac{1}{1-r_i^k}. 
\end{align*}
where the second equality follows from $\Nsf_i \upgamma \cdot \Curve_i = \upgamma \cdot \Msf_i^{-1} \Curve_i = - \upgamma \cdot \Curve_i^+$ and the equality
\[ 
\dfrac{r_i^{-k}}{1-r_i^{-k}} = \dfrac{1}{r_i^k - 1}. 
\]
Using $n_{k\Curve_i,\scrX} = n_{k\Curve_i^+,\scrX_i^+}$ by Theorem~\ref{thm: GV under flop}, the difference $\mathsf{G}^{+}_{\rsf}(\upgamma_1,\upgamma_2,\upgamma_3) - \mathsf{G}_{\rsf}(\Nsf_i\upgamma_1,\Nsf_i\upgamma_2,\Nsf_i\upgamma_3)$ is equal to 
\begin{align*} 
&\phantom{=}(\upgamma_1 \cdot \Curve_i^+)(\upgamma_2 \cdot \Curve_i^+)(\upgamma_3\cdot \Curve_i^+) \sum_{k \geq 1} k^3 n_{k\Curve_i,\scrX} \left( \dfrac{r_i^k}{1-r_i^k} - \dfrac{1}{1-r_i^k} \right) \\
& = -(\upgamma_1 \cdot \Curve_i^+)(\upgamma_2 \cdot \Curve_i^+)(\upgamma_3\cdot \Curve_i^+) \sum_{k \geq 1} k^3 n_{k\Curve_i,\scrX}. 
\end{align*}
We next examine the $\mathsf{H}$ terms. Note that for $\upbeta \in \Achow_1(\scrX)$ we have $\upbeta \in \Z \Curve_i$ if and only if $\Msf_i^{-1}\upbeta \in \Z \Curve_i^+$. Consequently
\begin{align*} 
\mathsf{H}_{\rsf}(\Nsf_i \upgamma_1,\Nsf_i \upgamma_2,\Nsf_i\upgamma_3) & = \sum_{\substack{\upbeta \in \Achow_1(\scrX) \\ \upbeta \neq k\Curve_i}} n_{\upbeta,\scrX} \, (\Nsf_i\upgamma_1\cdot\upbeta)(\Nsf_i\upgamma_2\cdot\upbeta)(\Nsf_i\upgamma_3\cdot\upbeta) \dfrac{\rsf^{\Msf_i^{-1} \upbeta}}{1-\rsf^{\Msf_i^{-1} \upbeta}} \\
& = \sum_{\substack{\upbeta \in \Achow_1(\scrX_i^+) \\ \upbeta \neq k\Curve_i^+}} n_{\Msf_i \upbeta,\scrX} \, (\Nsf_i \upgamma_1 \cdot \Msf_i \upbeta)(\Nsf_i \upgamma_2 \cdot \Msf_i \upbeta)(\Nsf_i \upgamma_3 \cdot \Msf_i \upbeta) \dfrac{r^{\upbeta}}{1-r^{\upbeta}} \\
& = \sum_{\substack{\upbeta \in \Achow_1(\scrX_i^+) \\ \upbeta \neq k\Curve_i^+}} n_{\upbeta,\scrX_i^+} \, (\upgamma_1 \cdot \upbeta)(\upgamma_2 \cdot \upbeta)(\upgamma_3 \cdot \upbeta) \dfrac{r^{\upbeta}}{1-r^{\upbeta}} \\
& = \mathsf{H}^{+}_{\rsf}(\upgamma_1,\upgamma_2,\upgamma_3).
\end{align*}
where the penultimate equality holds since $\Nsf_i \upgamma \cdot \Msf_i \upbeta = \upgamma \cdot \Msf_i^{-1} \Msf_i \upbeta = \upgamma \cdot \upbeta$ and $n_{\Msf_i \upbeta,\scrX} = n_{\upbeta,\scrX_i^+}$ again by Theorem~\ref{thm: GV under flop}. Combining the comparison of the $\mathsf{G}$ terms with the comparison of the $\mathsf{H}$ terms gives \eqref{eqn: CTC}, as required.
\end{proof}

\subsection{The contraction algebra under flop}
The flopping contraction $\scrX\to\Spec \scrR$ has an associated contraction algebra $\CA$, defined using noncommutative deformation theory \cite{DW1, DW3}. After flopping a single curve $\Curve_i$ to obtain $\scrX_i^+\to\Spec \scrR$, noncommutative deformation theory associates to this another contraction algebra, written  $\upnu_i\CA$.  The algebra $\upnu_i\CA$ can be intrinsically obtained from $\CA$ via a certain mutation procedure, and in fact $\CA$ and $\upnu_i\CA$ are derived equivalent \cite{AugustTiltingTheory}.   Both $\CA$ and $\upnu_i\CA$ are finite dimensional algebras \cite[2.13]{DW1}.

Fix the GV invariants $n_\upbeta$ associated to $\scrX\to\Spec \scrR$, then Toda's dimension formula (see \ref{thm: Toda dim formula}) asserts that
\[
\dim_{\mathbb{C}}\CA=\sum_{\upbeta\in \Achow_1(\scrX)}n_\upbeta \big(\,\upbeta\cdot \mathds{1}\big)^2
\]
In many, but not all, cases the dimension of $\CA$ is in fact enough to recover the $n_\upbeta$.  The next result asserts that the numbers $n_\upbeta$ associated to $\CA$, together with the matrix $\mathsf{M}^{-1}_i$, completely determine the dimension of $\upnu_i\CA$.
\begin{cor}\label{cor: Toda formula iterate}
Under mutation at vertex $i$, equivalently flop at $\Curve_i$,
\[
\dim_{\mathbb{C}}\upnu_i\CA=\sum_{\upbeta\in \Achow_1(\scrX)}n_\upbeta \big(\,(\mathsf{M}^{-1}_i\upbeta)\cdot \mathds{1}\big)^2
\]
where $\mathsf{M}_i$ is the explicit matrix in \eqref{defn Mi}.
\end{cor}
\begin{proof}
Combining previous results, it follows that
\begin{align*}
\dim_{\mathbb{C}}\upnu_i\CA
&=\sum_{\upgamma\in \Achow_1(\scrX_i^+)}n_{\upgamma}(\upgamma\cdot \mathds{1})^2\tag{by \ref{thm: Toda dim formula}}\\
&=\sum_{\upbeta\in \Achow_1(\scrX)}n_{\upbeta}\big(\,|\mathsf{M}^{-1}_i\upbeta|\cdot \mathds{1}\big)^2\tag{$\upgamma=|\mathsf{M}^{-1}_i\upbeta|$ in Theorem~\ref{thm: GV under flop}}\\
&=\sum_{\upbeta\in \Achow_1(\scrX)}n_{\upbeta}\big(\,(\mathsf{M}_i^{-1}\upbeta)\cdot \mathds{1}\big)^2\tag{by Corollary~\ref{cor: action of Mi pos neg}}
\end{align*}
where the point is that, by Corollary~\ref{cor: action of Mi pos neg}, the sign issue doesn't matter once we square. 
\end{proof}
In particular, it is possible to compute the dimension of $\upnu_i\CA$ without first having to present it.
\begin{example}
As in \cite{SmithWemyss}, consider the $cA_2$ example $\scrR_k=\frac{\mathbb{C}[\![u,v,x,y]\!]}{uv-xy(x^k+y)}$ for $k\geq 1$, and the specific crepant resolution $\scrX\to\Spec\scrR_k$ described in \cite[3.1]{SmithWemyss}, obtained first by blowing up $(u,y)$ then $(u,x)$.  In this case, as explained in \cite[Subsection~6.2]{SmithWemyss} $\CA$ can be presented as
\[
\begin{array}{c}
\begin{tikzpicture}[scale=0.8,>=stealth]
\node (A) at (0,0) [Bquiv] {};
\node (B) at (2,0) [Bquiv] {};
\draw[->, bend left] (A) to node[above]{$\scriptstyle a$} (B);
\draw[->, bend left] (B) to node[below]{$\scriptstyle b$} (A);
\end{tikzpicture}
\end{array}
\quad
\begin{array}{c}
(ab)^ka=0=b(ab)^k.
\end{array}
\]
We can immediately read off the GV invariants, namely $n_{1,0}=1$, $n_{0,1}=1$, and $n_{1,1}=k$.  Thus $\dim_\mathbb{C}\CA=n_{1,0}\cdot (1+0)^2+n_{0,1}\cdot (0+1)^2+n_{1,1}\cdot (1+1)^2$, which equals $1+1+4k=4k+2$.

We now flop the right hand curve.  In this Type $A$ example $\mathsf{M}_i^{-1}$ sends $(1,0)\mapsto (1,1)$, $(1,1)\mapsto (1,0)$ and $(0,1)\mapsto (0,-1)$.  Thus, by Corollary~\ref{cor: Toda formula iterate},
\[
\dim_\mathbb{C}\upnu_i\CA=n_{1,0}\cdot (1+1)^2+n_{0,1}\cdot (0+1)^2+n_{1,1}\cdot (0-1)^2,
\] 
which equals $4+1+k=k+5$.  In particular, $\upnu_i\CA\ncong\CA$ provided that $k\geq 2$.
\end{example}

\appendix
\section{Toda's dimension formula}\label{sec: n beta eq n beta}

\noindent This appendix contains a proof of the general form of Toda's dimension formula, which relates the dimension of the contraction algebra to a weighted sum of GV invariants.  This formula first appeared in \cite{TodaGV} for single-curve flops, then in \cite{TodaUtah} in general.  Alas, the GV invariants in \cite{TodaUtah} are defined with respect to moduli spaces of the contraction algebra, and not via perturbation as done here (Subsection~\ref{sec: GV invariants}), and furthermore \cite[Subsection~4.4]{TodaUtah} contains no proof. As such, we briefly sketch the argument to convince the reader (and ourselves!) that the formula contains nothing specific to single-curve flops.

\medskip
As notation, let $f\colon \scrX\to\Spec\scrR$ be a smooth $3$-fold flopping contraction, with contraction algebra $\CA$.   Write $\scrV$ for the Van den Bergh tilting bundle on $\scrX$ \cite{VdB1d} which generates zero perverse sheaves, and set $\mathrm{A}=\End_\scrX(\scrV)$.

\begin{lemma}\label{lem: CA 3CY dual}
$\RHom_{\mathrm{A}}(\CA,\mathrm{A})\cong \mathrm{M}[-3]$ for some $\CA$-bimodule $\mathrm{M}$ for which $\dim_{\mathbb{C}}\mathrm{M}=\dim_{\mathbb{C}}\CA$.
\end{lemma}
\begin{proof}
First, by CY duality 
\[
\Ext^i_{\mathrm{A}}(\CA,\mathrm{A})\cong \Ext^{3-i}_{\mathrm{A}}(\mathrm{A},\CA)^\star,
\]
which is zero unless $i=3$, when it equals $\Hom_{\mathrm{A}}(\mathrm{A},\CA)^\star\cong (\CA)^\star$.  In particular, it follows that the cohomology of $\RHom_{\mathrm{A}}(\CA,\mathrm{A})$ is concentrated in a single degree (namely, three), where as a vector space it has dimension $\dim_{\mathbb{C}}\CA$.  Truncating in the category of bimodules then yields the result.
\end{proof}

\begin{remark}
$M\cong\CA$ as $\CA$-bimodules \cite[2.6]{AugustTiltingTheory}, but we will not need this fact.
\end{remark}

In what follows, set $\mathrm{B}\colonequals \mathrm{A}\otimes\mathrm{A}^{\mathrm{op
}}$, so that $\mathrm{B}$-modules are the same as $\mathrm{A}$-bimodules.  The following is then \cite[2.3]{TodaGV} adapted to our setting.

\begin{cor}\label{cor:sheaf filtration}
$\RHom_{\mathrm{A}}(\CA,\mathrm{A})\otimes^{\bf L}_{\mathrm{B}}(\scrV^\star\boxtimes\scrV)\cong \scrG[-2]$ for some $\scrG\in\coh \scrX\times\scrX$ which admits a filtration
\[
0=\scrG_0\subset\hdots\subset\scrG_{\dim_{\mathbb{C}}\CA}=\scrG
\]
such that each $\scrG_t/\scrG_{t-1}$ is isomorphic to $\scrO_{\Curve_i}(-1)\boxtimes\scrO_{\Curve_j}(-1)$ for some $i,j$ with $1\leq i,j\leq n$.
\end{cor}
\begin{proof}
As notation, let $\scrT_0,\scrT_1,\hdots,\scrT_n$ be the simple left $\mathrm{A}$-modules, and $\scrS_0,\scrS_1,\hdots,\scrS_n$ be the simple right $\mathrm{A}$-modules.  All have dimension one, as a vector space, and by convention $\scrT_0$ (respectively $\scrS_0$) is the only simple which is not an $\CA$-module. 

Now $\mathrm{M}$ from Lemma~\ref{lem: CA 3CY dual} is a finite dimensional $\mathrm{B}$-module, so it is filtered by finite dimensional simples. But these all have the form $\scrT_i\otimes_{\mathbb{C}}\scrS_j$ for some $0\leq i,j\leq n$ (see e.g.\ \cite[3.10.2]{EtingofBook}).  Since $\mathrm{M}$ is an $\CA$-bimodule, $i=0$ or $j=0$ is not possible.  Hence $\mathrm{M}$ admits a filtration with quotients all of the form $\scrT_i\otimes_{\mathbb{C}}\scrS_j$ where $i,j\neq 0$.  The length of the filtration must be $\dim_{\mathbb{C}}\CA$, since each $\scrT_i\otimes_{\mathbb{C}}\scrS_j$ is one-dimensional.

Now, as observed by Toda \cite{TodaGV}
\begin{align*}
&\phantom{\cong}\RHom_{\scrX\times\scrX}\big(\scrV^\star\boxtimes\scrV, \scrO_{\Curve_i}(-1)\boxtimes\scrO_{\Curve_j}(-1)[-2]\big)\\
&\cong \RHom_{\scrX}(\scrV^\star, \scrO_{\Curve_i}(-1)[1])\otimes_{\mathbb{C}} \RHom_{\scrX}(\scrV, \scrO_{\Curve_j}(-1))[-3]\\
&\cong \scrT_i\otimes_{\mathbb{C}} \scrS_j[-3] \tag{by \cite[3.5.6, 3.5.7]{VdB1d}}
\end{align*}
Since $\RHom_{\mathrm{A}}(\CA,\mathrm{A})\cong \mathrm{M}[-3]$, applying the inverse functor $-\otimes^{\bf L}_{\mathrm{B}}\scrV^\star\boxtimes\scrV$ and inducting along the filtration of $\mathrm{M}[-3]$ gives the result. 
\end{proof}
Now as explained in Subsection~\ref{sec: perturbing target}, there exists a flat deformation
\[
\begin{tikzpicture}[>=stealth,scale=1.25]
\node (X) at (0,0) {$\mathfrak{X}$};
\node (Y) at (0,-1) {$\Spec \mathfrak{R}$};
\node (T) at (1,-1.75) {$T$};
\draw[->] (X)--node[left]{$\scriptstyle g$}(Y);
\draw[->] (Y)--(T);
\draw[->,densely dotted] (X)--(T);
\end{tikzpicture}
\]
for some Zariski open neighbourhood $T$ of $0\in\mathbb{A}^1$, such that
\begin{itemize}
\item the central fibre $g_0\colon \Xfrak_0\to \Spec \mathfrak{R}_0$ is isomorphic to $f \colon \scrX \to \Spec \scrR$.
\item all other fibres $g_t\colon \Xfrak_t\to \Spec \mathfrak{R}_t$ are flopping contractions with exceptional locus a disjoint union of $(-1,-1)$-curves.
\end{itemize}
Regarding the flopping curves $\Curve_1,\hdots,\Curve_n$ of $f$ as curves in the central fibre of $\mathfrak{X}\to T$, and thus as a curve in $\mathfrak{X}$, then the GV invariant $n_\upbeta$ is defined in Subsection~\ref{sec: GV invariants} to be the number of $g_t$-exceptional $(-1,-1)$-curves $\Curve$ such that for every line bundle $\scrL$ on $\mathfrak{X}$, 
\[
\deg (\scrL|_{\Curve})=\upbeta\cdot \big(\deg(\scrL|_{\Curve_1}),\hdots,\deg(\scrL|_{\Curve_n})\big)
\colonequals\sum_{i=1}^n\,\upbeta_i\deg(\scrL|_{\Curve_i}).
\] 
Here we are using the identification \eqref{eqn: lift line bundles} of line bundles on each fibre of $\Xfrak \to T$ with line bundles on the total space $\Xfrak$.

\begin{thm}[Toda]\label{thm: Toda dim formula}
If $\scrX\to\Spec\scrR$ is a smooth $3$-fold flopping contraction, then
\[
\dim_{\mathbb{C}}\CA=\sum_{\upbeta\in \Achow_1(\scrX)}n_\upbeta (\upbeta\cdot \mathds{1})^2
\]
where the $n_\upbeta$ are the GV invariants defined as in Subsection~\textnormal{\ref{sec: GV invariants}}.
\end{thm}
\begin{proof}
The proof very closely follows the single-curve strategy in \cite{TodaGV}, and so we only outline the parts where some care is required.

The $4$-fold $g\colon\Xfrak\to\Spec\mathfrak{R}$ is a flopping contraction, and so by \cite[Section~6]{ChenFlop} $g$ admits a flop $\Xfrak^+\to\Spec\mathfrak{R}$ and a derived equivalence $\Db(\Xfrak)\to\Db(\Xfrak^+)$. Flopping back gives another equivalence $\Db(\Xfrak^+)\to\Db(\Xfrak)$, and thus the composition gives rise to an autoequivalence
\[
\Uppsi\colon\Db(\coh\Xfrak)\to\Db(\coh\Xfrak)
\]
with FM kernel $\scrP\in\Db(\Xfrak\times_T\Xfrak)$, say.   Define $\scrP_t=\mathbf{L}j_t^\star\scrP$, where $j_t$ is the inclusion $\Xfrak_t\times\Xfrak_t\to\Xfrak\times_T\Xfrak$.

On one hand, for $t=0$, as in \cite[(18)]{TodaGV} the restriction of $\Uppsi$ to the zero fibre results in the NC twist functor of \cite{DW1}, so by uniqueness of FM kernels
\[
\scrP_0\cong \mathrm{Cone}(\scrF_0[-2]\to\scrO_{\Delta_{\mathfrak{X}_0}})
\]
where $\scrF_0=\scrG$ in Corollary~\ref{cor:sheaf filtration}.

On the other hand, for $t\neq 0$, the birational map from $\Xfrak_t$ to $\Xfrak_t$ is the composition of the flops of all the curves in that fibre.  Since all the curves have normal bundle $(-1,-1)$,  the restriction of $\Uppsi$ to the neighbouring fibre $\Xfrak_t$ results in composition of the corresponding (classical) spherical twists.  Grouping the curves in $\Xfrak_t$ together via their curve class, namely $\{\Curve_{\upbeta,i}\mid  1\leq i\leq n_\upbeta\}_{\upbeta}$, then again as in \cite[(19)]{TodaGV}, uniqueness of FM kernels yields
\[
\scrP_t\cong \mathrm{Cone}(\scrF_t[-2]\to\scrO_{\Delta_{\mathfrak{X}_t}})
\] 
where $\scrF_t\in\coh\Xfrak_t\times\Xfrak_t$ is the sheaf now defined by 
\begin{equation}
\scrF_t=\bigoplus_{\upbeta}\bigoplus_{i=1}^{n_\upbeta}\scrO_{\upbeta,i}(-1)\boxtimes\scrO_{\upbeta,i}(-1)\label{eqn: def Ft}
\end{equation}
where $\scrO_{\upbeta,i}$ is the structure sheaf of $\Curve_{\upbeta,i}$. Using \cite[3.1, 3.2]{TodaGV}, which are general, it follows that $\scrF_t$ with $t\neq 0$ is a flat deformation of $\scrF_0$.  Since both $\scrF_0$ and $\scrF_t$ have compact supports, their Hilbert polynomials must therefore be equal. In particular, let $\scrL$ be the $g$-ample line bundle on $\Xfrak$ such that $\deg(\scrL|_{\Curve_i})=1$ for all $i=1,\hdots,n$, which exists by \eqref{eqn: lift line bundles} and the sentence underneath, then
\begin{equation}
\upchi(\scrF_0\otimes(\scrL\boxtimes\scrL))=\upchi(\scrF_t\otimes(\scrL\boxtimes\scrL)).\label{eqn: Hilb polys same}
\end{equation}
But $\upchi$ is additive on filtrations, and so by Corollary~\ref{cor:sheaf filtration}
\begin{equation}
\upchi(\scrF_0\otimes(\scrL\boxtimes\scrL))=\sum_{i=1}^{\dim_{\mathbb{C}}\CA}\upchi(\scrO_{\Curve_i})\upchi(\scrO_{\Curve_j})=\dim_{\mathbb{C}}\CA.\label{eqn: Hilb LHS}
\end{equation}
Conversely, given the form of $\scrF_t$ in \eqref{eqn: def Ft}, it is clear that
\begin{equation}
\upchi(\scrF_t\otimes(\scrL\boxtimes\scrL))
=\sum_{\upbeta}\sum_{i=1}^{n_\upbeta}\upchi(\scrO_{\upbeta,i}(-1)\otimes\scrL)^2
=\sum_{\upbeta}n_\upbeta(\upbeta_1+\hdots+\upbeta_n)^2\label{eqn: Hilb RHS}
\end{equation}
since $\deg(\scrL|_{\Curve_{\upbeta,i}})=\upbeta_1+\hdots+\upbeta_n$. Combining \eqref{eqn: Hilb polys same}, \eqref{eqn: Hilb LHS} and \eqref{eqn: Hilb RHS} gives the result.
\end{proof}\bigskip

% SPRINGER-MANDATED DECLARATIONS COMMENTED BELOW

%\noindent \textbf{Data Availability Statement.}  Data sharing is not applicable to this article as no datasets were generated or analysed during the current study. \medskip

%\noindent \textbf{Conflict of Interest Statement.} On behalf of all authors, the corresponding author states that there is no conflict of interest.

\bibliographystyle{alpha}
\bibliography{Bibliography}

\end{document}